\documentclass[11pt,a4paper]{article}
\usepackage{amsmath}
\usepackage{indentfirst}
\usepackage{geometry}
\usepackage{amsfonts}
\usepackage{amssymb}
\usepackage{amsthm}
\usepackage{mathrsfs}
\usepackage{enumerate}
\newtheorem{Definition}{Definition}[section]

\newtheorem{Theorem}{Theorem}[section]
\newtheorem{Remark}[Theorem]{Remark}
\newtheorem{Lemma}[Theorem]{Lemma}

\newtheorem{Proposition}[Theorem]{Proposition}
\numberwithin{equation}{section}
\newcommand{\norm}[1]{\left\Vert#1\right\Vert}
\newcommand{\bb}[1]{\mathbb{#1}}

\def\h{\mathcal{H}}
\def\E{\hat{\mathbb{E}}}
\def\Et{\tilde{\mathbb{E}}}
\def\siup{\bar{\sigma}}
\def\sidown{\underline{\sigma}}
\def\Bq{\left\langle B\right\rangle}
\def\Exp{\mathscr{E}}
\def\C_lip{\mathcal{C}_{l,lip}}

\title{Reflected Quadratic BSDEs driven by $G$-Brownian Motions\footnotemark[1]}
\author{Dong Cao\footnotemark[2]
\quad and\quad Shanjian Tang\footnotemark[3]  }

\begin{document}
\maketitle
\begin{abstract}
In this paper, we consider a  reflected backward  stochastic differential equation driven by a $G$-Brownian motion ($G$-BSDE), with  the generator  growing quadratically  in the second unknown. We obtain the existence by the penalty method, and a priori estimates which implies the uniqueness, for solutions of the $G$-BSDE. Moreover, focusing our discussion at the Markovian setting, we give a nonlinear Feynman-Kac formula for solutions of a fully nonlinear partial differential equation.
\end{abstract}

\footnotetext[1]{Partially supported by National Science Foundation of China (Grant No. 11631004)
and Science and Technology Commission of Shanghai Municipality (Grant No. 14XD1400400). }
\footnotetext[2]{School of Mathematical Sciences, Fudan University, Shanghai 200433, China (\textit{e-mail: dcao14@fudan.edu.cn}).}
 \footnotetext[3]{Department of Finance and Control Sciences, School of
Mathematical Sciences, Fudan University, Shanghai 200433, China (\textit{e-mail: sjtang@fudan.edu.cn}).}

\section{Introduction}
A general  backward stochastic differential equation (BSDE) takes  the following form:
$$Y_t=\xi+\int_t^Tf(s,Y_s,Z_s)ds-\int_t^TZ_sdW_s,~~t\in[0,T]. $$
The function $f$ is conventionally called the generator and the random variable $\xi$ is called the terminal value.  Bismut~\cite{Bismut1976,Bismut1978} initially gave a complete linear theory,   where the generator is linear in both unknown variables, and derived the stochastic Riccati equation as a particular nonlinear BSDE where the generator is quadratic in the second unkown variable.  Pardoux and Peng~\cite{Pardoux1990} established the existence and uniqueness result when the generator $f$ is uniformly Lipschitz continuous in both unknown variables and the terminal value $\xi$ is square integrable.
Subsequently, an intensive  attention has been  given to relax the assumption of the uniformly Lipschitz continuity on the generator.  In particular,  the one-dimensional BSDE with a quadratic generator (i.e., the so-called quadratic BSDE) was studied by Kobylanski~\cite{Kobylanski2000} for a bounded terminal value  $\xi$, and by  Briand and Hu \cite{Briand2006,Briand2008} for an unbounded terminal value $\xi$ of some suitable exponential moments.  The multi-dimensional quadratic BSDE was discussed by Tang~\cite{Tang2003} and Hu and Tang~\cite{HuTang2016}.

 As a constrained BSDE, a reflected backward stochastic differential equation (RBSDE) was formulated and studied by El Karoui et al.~\cite{El Karoui1997}, where the first unknown $Y$ is required to stay up a given continuous process $S$ and an additional increasing process which satisfies the Skorohod condition, is thus introduced into the equation.   Subsequently, much efforts have been made to relax the  Lipschitz assumption on the generator. For the quadratic case, see Kobylanski et al.~\cite{Kobylanski2002} with bounded terminal values, and  Lepeltier and Xu \cite{Lepeltier2007}  with unbounded terminal values.

To incorporate the  Knightian uncertainty,  Peng~\cite{Peng2007,Peng2008,Peng2010,Peng2010_1}  introduced the notion of  $G$-expectation as a  time-consistent sub-linear expectation, and constructed (via a fully nonlinear PDE)  the so-called $G$-Brownian motion $\{B_t,~t\in [0,+\infty)\}$, whose quadratic variation process $\Bq$---in contrast to the classical Brownian motion---is not deterministic.   The stochastic integral with respect to the $G$-Brownian motion and its quadratic variation were also discussed by Peng \cite{Peng2007}.  Denis et al.~\cite{Denis2011} proves that  the $G$-expectation is in fact  the upper expectation over a collection of mutually singular martingale measures $\mathcal{P}$.
Hu et al.~\cite{HuM2014_1}  showed that there is a unique triple of processes $(Y,Z,K)$ in a proper Banach space satisfying the following scalar-valued BSDE driven by the $G$-Brownian motion $B$:
\begin{equation}\label{GBSDE}
Y_t=\xi+\int_t^Tg(s,Y_s,Z_s)ds+\int_t^Tf(s,Y_s,Z_s)d\Bq_s-\int_t^TZ_sdB_s-\int_t^TdK_s,~~t\in[0,T],
\end{equation}
where $f$ and $g$ are uniformly Lipshchitz in both unknown variables. Hu et al. \cite{Hu2018} proved the existence and uniqueness for adapted solutions to the scalar-valued $z$-quadratic  BSDE~\eqref{GBSDE}  driven by the $G$-Brownian motion  $B$  for a bounded terminal value $\xi$. Very recently, Li, Peng, and Soumana Hima~\cite{Li2017_1} discuss a reflected BSDE driven by the $G$-Brownian motion subject to  a lower obstacle under the uniformly Lipschitz condition, where  a $G$-martingale condition rather than  the conventional Skorohod condition,  is used  to characterize  the unknown bounded variational  process which is introduced into the equation  to keep the first unknown process  stay up the lower obstacle under  the $G$-expectation.  More precisely, they showed that there is a unique triple $(Y,Z,A)$ of processes satisfying the following equation:
\begin{equation}
\begin{cases}
 \displaystyle Y_t=\xi+\int_t^Tg(s,Y_s,Z_s)ds+\int_t^Tf(s,Y_s,Z_s)d\Bq_s\\
 \displaystyle\quad\quad\quad\quad -\int_t^TZ_sdB_s+\int_t^TdA_s,~~t\in[0,T];\\[3mm]
Y_t\geq S_t, 0\leq t\leq T;\quad
\int_0^\cdot (S_s-Y_s)dA_s \text{ is a non-increasing }G\text{-martingale.}
 \label{RGBSDE0}
 \end{cases}
 \end{equation}
  A subsequent study of Li and Peng~\cite{Li2017_2} reported  the following unexpected observation on the upper obstacle  problem  for the reflected BSDE driven by a $G$-Brownian motion: the proof of the uniqueness of solutions in the lower obstacle problem turns out to be difficult to be adapted to the upper obstacle problem. Since the preceding two equations hold $\bb{P}$-$a.s.$ for each $\bb P\in \mathcal{P}$, they  are also associated to second order BSDEs,  which have been discussed by Cheridito et al.~\cite{Cheridito2007}, Soner et al.~\cite{Soner2012},  and Possama\"i and Zhou~\cite{PossZhou2013}. Moreover, Matoussi,  Piozin and Possama\"i~\cite{Matoussi1} and Matoussi, Possama\"i and Zhou~\cite{Matoussi2,Matoussi3} discuss the reflected second order BSDEs. In the context of a $G$-BSDE, the solution is universally discussed in a ``better" space of processes, and  its existence naturally requires more regularity of the coefficients.

As a generalized counterpart of  the classical reflected quadratic BSDEs, the existence and uniqueness result for reflected quadratic BSDEs driven by $G$-Brownian motions still remains to be studied.
 The main objective of this paper is to provide the well-posedness of the reflected $G$-BSDE (\ref{RGBSDE0}) when the generator has a quadratic growth and the terminal value $\xi$ is bounded. As noted in Li, Peng and Soumana Hima~\cite{Li2017_1} and Possama\"{i} and Zhou~\cite{PossZhou2013}, the dominated convergence theorem does not hold under the $G$-framework, and a bounded sequence in $M_G^p(0,T)$ does not necessarily have the weak compactness. These striking differences prevent us from adapting the method of Kobylanski et al.~\cite{Kobylanski2002} to approximate the quadratic generator with Lipshcitz ones and then to prove the solutions of the approximating reflected BSDEs to converge to that of the original reflected quadratic  BSDE. Instead in this paper, we  use a penalty method in the spirit of El Karoui et al.~\cite{El Karoui1997} (for a BSDE in a Wiener space) and Li, Peng and Soumana Hima~\cite{Li2017_1} (for a $G$-BSDE). Since our generator is allowed to grow quadratically in the second unknown variable,  the terminal value $\xi$ is assumed to be bounded for simplicity of exposition, and then the symmetric martingale part of the underlying BSDE is discussed in the BMO space.

As in Hu et al.~\cite{HuM2014_2} and Li, Peng and Soumana Hima~\cite{Li2017_1}, the solution of a forward backward differential equations driven by $G$-Brownian motion ($G$-FBSDEs in short) can be interpreted as a viscosity solution of a PDE. We first prove the existence
of the quadratic $G$-BSDEs in a Markovian setting. We then give the nonlinear Feynman-Kac formula for a fully nonlinear parabolic variational  in equality via the quadratic $G$-BSDEs and the reflected quadratic $G$-BSDEs.

 The paper is organized as follows. Section 2 is dedicated to preliminaries on the $G$-framework, the formulation of reflected $G$-BSDEs, $G$-BMO martingales and $G$-Girsanov Theorem. In Section 3, we introduce some priori estimates for quadratic reflected $G$-BSDEs through the $G$-Girsanov transformation,  which yields the uniqueness in a straightforward way.
 In Section 4, we establish the approximation method via penalization. We state some convergence properties of the solutions to the penalized $G$-BSDEs. In Section 5, we prove our main result and  a comparison theorem.
 Finally, in Section 6, we give a nonlinear Feynmann-Kac formula and address the relation between quadratic $G$-BSDEs and nonlinear parabolic PDEs.

\section{Preliminaries}

\subsection{ Notations and results on $G$-expectation and quadratic $G$-BSDEs}

In this section, we first recall  notations and basic results concerning $G$-expectation, $G$-Brownian motion  and related $G$-stochastic calculus, and quadratic $G$-BSDEs. More details can be found in \cite{HuM2014_1}, \cite{HuM2014_2}, \cite{LiX2011}, \cite{Peng2007}, \cite{Peng2008}, and \cite{Peng2010}.

Let $\Omega$ be a complete separable metric space, and let $\h$ be a linear space of real-valued functions defined on $\Omega$ satisfying $c\in\h$ for each constant $c$ and $|X|\in\h$ if $X\in\h.$
$\h$ is considered as the space of random variables.
\begin{Definition}
(Sublinear expectation space).
A sublinear expectation $\E[\cdot]$ is a functional $\E:\h\to\bb{R}$ satisfying the following properties: for all $X,Y\in\h,$ we have
\begin{enumerate}
\item
Monotonicity: if $X\geq Y$, then $\E[X]\geq\E[Y]$;
\item
Constant preservation: $\E[c]=c,~c\in\bb{R}$;
\item
Sub-additivity: $\E[X+Y]\leq \E[X]+\E[Y]$;
\item
Positive homogeneity: $\E[\lambda X]=\lambda \E[X]$, for all $\lambda\geq 0$.
\end{enumerate}
We call the triple $(\Omega,\h,\E)$ a sublinear expectation space.
\end{Definition}

\begin{Definition}
(Independence).  In a sublinear expectation space $(\Omega,\h,\E)$, a random vector $Y=(Y_1,Y_2,\cdots,Y_n),~Y_i\in\h$ is said to be independent of another random vector $X=(X_1,X_2,\cdots,X_m),~X_i\in\h$, if $\E[\phi(X,Y)]=\E[\E[\phi(x,Y)]|_{x=X}],$
for all $\phi\in\C_lip(\bb{R}^{m+n})$, where $\C_lip(\bb{R}^{n})$ is the space of real continuous functions defined on $\bb R^n$ such that
$$|\phi(x)-\phi(y)|\leq C(1+|x|^k+|y|^k)|x-y|,\quad \forall x,y\in\bb R^n,$$
where $k$ and $C$ depend only on $\phi.$
\end{Definition}

\begin{Definition}
($G$-normal distribution).
We say the random vector $X=(X_1,X_2,\cdots,X_d)$ is $G$-normally distributed, if for any function $\phi\in\C_lip(\bb{R}^d)$, the function $u$ defined by $u(t,x):=\E\big[\phi\big(x+\sqrt tX\big)\big],~(t,x)\in[0,+\infty)\times\bb{R}^d$, is a viscosity of $G$-heat equation:
$$\partial_tu-G\big(D_x^2u\big)=0;~u(0,x)=\phi(x).$$
Here $G$ denotes the function
$$G(A):={1\over 2}\E[\left\langle AX,X\right\rangle]: \bb{S}_d\to \bb R,$$
where $\bb{S}_d$ denotes the collection of $d\times d$ symmetric matrices.
\end{Definition}
The function $G(\cdot)$ is a monotonic, sublinear mapping on $\bb S_d$ ~and
$$G(A)={1\over 2}\E[\left\langle AX,X\right\rangle]\leq {1\over2}|A|\E[|X|^2]:={1\over2}|A|\siup^2$$
implies that there exists a bounded, convex and closed subset $~\Gamma\subseteq \bb S_d^+~$ such that
$$G(A)={1\over2}\sup_{\gamma\in\Gamma}\text{tr}[\gamma A],$$
where $~\bb S_d^+~$ denotes the collection of nonnegative elements in $~\bb S_d$.

In this paper, we only consider a non-degenerate $G$-normal distribution, i.e, there exists some $\sidown>0$ such that $G(A)-G(B)\geq \sidown^2\text{tr}[A-B]$ for any $A\geq B.$

We now fix $\Omega:=C_0([0,\infty);\bb R^d)$, the space of all $\bb R^d$-valued continuous functions $\{\omega_t,~t\in [0,+\infty)\}$ with $\omega_0=0.$ Let $\mathcal{F}=\{\mathcal F_t,~t\in [0,+\infty)\}$ be the nature filtration generated by the canonical process $\{B_t,~t\in [0,+\infty)\}$, i.e., $B_t(\omega)=\omega_t$ for $(t,\omega)\in [0,\infty)\times\Omega.$ Set $\Omega_T:=C_0([0,T];\bb R^d)$. Let us consider the function spaces defined by
$$Lip(\Omega_T):=\big\{\phi(B_{t_1},B_{t_2}-B_{t_1},\cdots,
B_{t_n}-B_{t_{n-1}}):0\leq t_1\leq t_2\leq\cdots\leq t_n\leq T, ~\phi\in\C_lip(\bb R^{d\times n})\big\}$$
for $T>0$, and
$Lip(\Omega)=\bigcup\limits_{n=1}^\infty Lip(\Omega_n)$.
\begin{Definition}
($G$-Brownian motion and $G$-expectation).
On the sublinear expectation space $(\Omega,Lip(\Omega),\E)$, the canonical process $\{B_t,~t\in [0,+\infty)\}$ is called $G$-Brownian motion if the following properties are satisfied:
\begin{enumerate}
\item
$B_0=0$;
\item
For each $t,s>0$, the increment $B_{t+s}-B_t$ is independent of $(B_{t_1},B_{t_2},\cdots,B_{t_n})$, for each $n\in\bb N$ and $0\leq t_1\leq t_2\leq\cdots\leq t_n\leq t;$
\item
$B_{t+s}-B_t$ is $G$-normally distributed.
\end{enumerate}
Moreover, the sublinear expectation $\E[\cdot]$ is called $G$-expectation.
\end{Definition}

\begin{Definition}
(Conditional $G$-expectation).
For the random variable $\xi\in Lip(\Omega_T)$ of the following form:
$$\phi(B_{t_1},B_{t_2}-B_{t_1},\cdots,B_{t_n}-B_{t_{n-1}}),\quad\phi\in\C_lip(\bb R^{d\times n}),$$
the conditional $G$-expectation $\E_{t_i}[\cdot]$, $i=1,2,\cdots,n$, is defined as follows:
$$\E_{t_i}=[\phi(B_{t_1},B_{t_2}-B_{t_1},\cdots,B_{t_n}-B_{t_{n-1}})]
=\tilde\phi(B_{t_1},B_{t_2}-B_{t_1},\cdots,B_{t_i}-B_{t_{i-1}}),$$
where
$$\tilde\phi(x_1,x_2,\cdots,x_i)=\E[\phi(x_1,x_2,\cdots,x_i,B_{t_{i+1}}-B_{t_i},\cdots,B_{t_n}-B_{t_{n-1}})].$$
If $t\in(t_i,t_{i+1})$, the conditional $G$-expectation $\E_t[\xi]$ could be defined by reformulating $\xi$ as
$$\xi=\hat\phi(B_{t_1},B_{t_2}-B_{t_1},\cdots,B_{t}-B_{t_i},B_{t_{i+1}}-B_{t},\cdots,B_{t_n}-B_{t_{n-1}})
,\quad\hat\phi\in\C_lip(\bb R^{d\times (n+1)}).$$
\end{Definition}

For $\xi\in Lip(\Omega_T)$ and $p\geq 1$, we consider the norm $\norm{\xi}_{L_G^p}=\big(\E[|\xi|^p]\big)^{1\over p}$. Denote by $L_G^p(\Omega_T)$ the Banach completion of $Lip(\Omega_T)$ under $\norm{\cdot}_{L_G^p}$. It is easy to check that the conditional $G$-expectation $\E_t[\cdot]: ~Lip(\Omega_T)\to Lip(\Omega_t)$ is a continuous mapping and thus can be extended to $\E_t:~L_G^p(\Omega_T)\to L_G^p(\Omega_t)$.

\begin{Definition}
($G$-martingale).
A process $\{M_t,~t\in [0,T]\}$ is called a $G$-martingale if
\begin{enumerate}[(i)]
\item
   $M_t\in L_G^1(\Omega_t)$, for any $t\in[0,T]$;
\item
$\E_s[M_t]=M_s$, for all $0\leq s\leq t\leq T$.
\end{enumerate}
The process $\{M_t,~t\in [0,T]\}$ is called a symmetric $G$-martingale if $-M$ is also a $G$-martingale.
\end{Definition}

The following representation result of $G$-expectation on $L_G^1(\Omega_T)$, can be found in Denis et  al.~\cite[Propositions 49 and 50, page 157-158]{Denis2011}  and  Hu and Peng~\cite[Theorem 3.5, page 544]{HuM2009}.

\begin{Theorem} There exists a weakly compact set $\mathcal P\subseteq \mathcal{M}_1(\Omega_T)$ (i.e., the set of all probability measures on $(\Omega_T,\mathcal{B}(\Omega_T))$), such that
$$\E[\xi]=\sup_{\bb P\in \mathcal{P}}E^{\bb P}[\xi],~~\forall \xi\in L_G^1(\Omega_T),$$
where $E^{\bb P}[\cdot]$ is the expectation operator with respect to  probability $\bb P$.
Such $\mathcal P$ is called a representative set of  $\E$.
\end{Theorem}

Let $\mathcal P$ be a weakly compact set that represents $\E$. For this $\mathcal P$, we define capacity
$c(A):=\sup_{\bb P\in\mathcal P}\bb P(A),~A\in\mathcal B(\Omega_T)$.

\begin{Definition}
(Quasi-sure).
A set $A\in \mathcal B(\Omega_T)$ is a polar set if $c(A)=0$. A property holds ``quasi-surely" (q.s.) if it holds outside a polar set.
\end{Definition}
In what follows,  two random variables $X$ and $Y$ will not be distinguished  if $X = Y$, q.s.

Soner et al.~\cite[Proposition 3.4, page 272]{Soner2011} give  the following characterization of the conditional $G$-expectation.

\begin{Theorem}
\label{Thm_conditional_E}
For any $\xi\in L_G^1(\Omega_T)$, $t\in[0,T]$ and $\bb P\in \mathcal{P}$,
$$\E_t[\xi]=\operatorname*{ess\sup}_{\bb P^\prime\in \mathcal{P}(t,\bb P)}E^{\bb P^\prime}_t[\xi],\quad \bb P\text{\,-a.s.},$$
where
$$\mathcal{P}(t,\bb P):=\{\bb P^\prime\in \mathcal P: \bb P^\prime=\bb P~\text{on}~\mathcal F_t\}.$$
\end{Theorem}
In view of Theorem \ref{Thm_conditional_E}, it is easy to check the following property for $G$-martingales.
\begin{Proposition}
Assume that $\{M_s,~s\in[0,T]\}$ is a $G$-Martingale and $\{\eta_s,~s\in[0,T]\}$ is a process satisfying $\eta_s\in L_G^1(\Omega_s)$, for any $s\in[0,T]$. Then we have for any $t\in[0,T]$,
$$\E_t[\eta_t+M_T-M_t]=\eta_t.$$
\end{Proposition}

For the terminal value of quadratic $G$-BSDE, we define the space $L_G^\infty(\Omega_T)$ as the completion of $Lip(\Omega_T)$ under the norm
$$\norm{\xi}_{L_G^\infty}:=\inf\{M\geq 0:|\xi|\leq M,~\text{q.s.}\}.$$

For $\xi\in Lip(\Omega_T)$ and $p\geq 1$, define $\norm{\xi}_{p,\mathcal E}=\E[\sup_{t\in[0,T]}\E_t[|\xi|^p]]^{1\over p}$ and denote by $L^p_{\mathcal E}(\Omega_T)$ the completion of $Lip(\Omega_T)$ under $\norm{\cdot}_{p,\mathcal E}$.
Song~\cite[Theorem 3.4, page 293]{Song2011}) gives the following estimate.

\begin{Theorem}
\label{Thm_doob_type_ineq}
For any $\alpha\geq1$ and $\delta>0$, $L^{\alpha+\delta}_{G}(\Omega_T)\subseteq L^\alpha_{\mathcal E}(\Omega_T)$.
More precisely, for any $1<\gamma<\beta:={(\alpha+\delta)/\alpha}$, $\gamma\leq 2$, we have
$$\norm{\xi}_{\alpha,\mathcal E}^\alpha\leq \gamma^*\big\{\norm{\xi}^\alpha_{L_G^{\alpha+\delta}}+
14^{1/ \gamma}C_{\beta/ \gamma}\norm{\xi}^{(\alpha+\beta)/ \gamma}_{L_G^{\alpha+\delta}}\big\},\quad \forall\xi\in Lip(\Omega_T),$$
where $C_{\beta/ \gamma}=\sum_{i=1}^\infty i^{-\beta/\gamma}$, $\gamma^*=\gamma/(\gamma-1)$.
\end{Theorem}

\begin{Remark}
\label{remark_doob_type_ineq}
In view of  \cite[Remark 2.9]{HuM2014_1}, there exists $C_1$ depending only on $\alpha$ and $\delta$ such that
$$\E\Big[\sup_{t\in[0,T]}\E_t[|\xi|^\alpha]\Big]\leq C_1\big\{ \E[|\xi|^{\alpha+\delta}]^{\alpha\over \alpha+\delta}+\E[|\xi|^{\alpha+\delta}]\big\}.$$
\end{Remark}

Let $B$ be a $d$-dimensional $G$-Brownian motion. For each fixed $a\in \bb R^d$, $B_t^a=\left\langle a,B_t \right\rangle$ is a 1-dimensional $G_a$-Brownian motion, where $G_a(\alpha)=G(aa^T)\alpha^++G(-aa^T)\alpha^-.$
The quadratic variation process of $B^a$ is defined by
$$\left\langle B^a \right\rangle_t=\lim_{\mu(\pi_t^N)\to 0}\sum_{j=0}^{N-1}(B_{t_{j+1}^N}^a-B_{t_{j}^N}^a)^2,$$
where $\pi_t^N$, $N=1,2,\cdots$, are refining partitions of $[0,t]$. By Peng \cite{Peng2010}, for all $t,s>0$,
$\left\langle B^a \right\rangle_{t+s}-\left\langle B^a \right\rangle_t\in[-2G(-aa^T)s,2G(aa^T)s]$, q.s. \\
For each fixed $a,\bar a\in\bb R^d$, the mutual variation process of $B^a$ and $B^{\bar a}$ is defined by
$$\left\langle B^a,B^{\bar a} \right\rangle_t={1\over 4}\big[\left\langle B^{a+\bar a} \right\rangle_t-\left\langle B^{a-\bar a} \right\rangle_t\big].$$

Next we discuss the stochastic integrals with respect to the $G$-Brownian motion and its quadratic variation.
\begin{Definition}
Let $M_G^0(0,T)$ be the collection of processes $\eta$ of the following form: for a given partition $\{t_1,r_2,\cdots,t_n\}=\pi_T$ of $[0,T]$,
$$\eta_t(\omega)=\sum_{j=0}^{N-1}\xi_j(\omega)\textbf{1}_{[t_j,t_{j+1})}(t),$$
where $\xi_i\in Lip(\Omega_{t_i})$ for $i=0,1,2,\cdots,N-1$. For $p\geq 1$ and $\eta\in M_G^0(0,T)$, define
$$\norm{\eta}_{H_G^p}:=\E\left[\left(\int_0^T|\eta_s|^2ds\right)^{p\over 2}\right]^{1\over p} \quad \hbox{\rm and }\quad \norm{\eta}_{M_G^p}:=\E\left[\left(\int_0^T|\eta_s|^pds\right)\right]^{1\over p}.$$
Denote by $H_G^p(0,T)$ and $M_G^p(0,T)$ the completion of $M_G^0(0,T)$ under norms $\norm{\cdot}_{H_G^p}$ and $\norm{\cdot}_{M_G^p}$, respectively.
\end{Definition}

For both processes $\eta\in M_G^2(0,T)$ and $\xi\in M_G^1(0,T)$, the $G$-It\^o integrals $\{\int_0^t\eta_sdB_s^i,~0\leq t\leq T\}$ and $\{\int_0^t\xi_sd\langle B^i, B^j\rangle_s,~0\leq t\leq T\}$ are well defined in \cite{LiX2011} and \cite{Peng2010}. Moreover, the following BDG inequality can be found in \cite[Proposition 4.3, Page 295]{Song2011}.
\begin{Proposition}
For $\eta\in H_G^\alpha(0,T)$ with $\alpha\geq 1$ and $p\in(0,\alpha]$, we have
$$\sidown^pc_p\E_t\Big[\Big(\int_t^T|\eta_s|^2ds\Big)^{p\over 2}\Big]\leq \E_t\Big[\sup_{u\in[t,T]}\Big|\int_t^u\eta_sdB_s\Big|^{p}\Big]\leq \siup^pC_p\E_t\Big[\Big(\int_t^T|\eta_s|^2ds\Big)^{p\over 2}\Big],~~\text{q.s.} $$
\end{Proposition}

Denote by $\mathcal{C}_{b,lip}(\bb R^{1+d\times n})\big\}$  the collection of all bounded and Lipschitz functions on $\bb R^{1+d\times n}$. Define
\begin{eqnarray*}
S_G^0(0,T)&:=&\Big\{h(t,B_{t_1\land t},B_{t_2\land t}-B_{t_1\land t},\cdots,
B_{t_n\land t}-B_{t_{n-1}\land t}): \\
&&\quad\quad  ~h\in\mathcal{C}_{b,lip}(\bb R^{1+d\times n}) \hbox{ \rm and } t_1, t_2\cdots ,t_n\in [0,T]\Big\}.
\end{eqnarray*}
For $p\geq 1$ and $\eta\in S_G^0(0,T)$, set $$\norm{\eta}_{S_G^p}:=\E\left[\sup_{t\in[0,T]}|\eta_t|^p\right]^{1\over p}.$$
 Denote by $S_G^p(0,T)$ the completion of $S_G^0(0,T)$ under the norm $\norm{\cdot}_{S_G^p}$.
 The following continuity of  $Y\in S_G^p(0,T)$ for $p>1$ can be found in Li, Peng, and Song~\cite[Lemma 3.7, page 12]{Li2017_3}.

\begin{Lemma}
\label{lemma_Y_continue}
For $Y\in S_G^p(0,T)$ with $p>1$, we have, by setting $Y_s=Y_T$ for $s>T$,
$$F(Y):=\limsup_{\varepsilon\to 0} \E\big[\sup_{t\in[0,T]}\sup_{s\in[t,t+\varepsilon]}|Y_t-Y_s|^p\big]^{1\over p}=0.$$
\end{Lemma}
Similar to $S_G^p(0,T)$, we can define the space $S_G^\infty(0,T)$ as the completion of $S_G^0(0,T)$ under the norm $\norm{\eta}_{S_G^\infty}:=\norm{\sup_{t\in[0,T]}|\eta_t|}_{L_G^\infty}$.

We now introduce some results on quadratic $G$-BSDEs in \cite{Hu2018}. For simplicity, we assume $d=1$ and consider the following type of equation:
\begin{equation}
 Y_t=\xi+\int_t^Tg(s,\omega_{\cdot \land s},Y_s,Z_s)ds+\int_t^Tf(s,\omega_{\cdot\land s},Y_s,Z_s)d\Bq_s-\int_t^TZ_sdB_s-\int_t^TdK_s,~~\text{q.s.},
 \label{QGBSDE}
 \end{equation}
 where the generator $(f,g):[0,T]\times\Omega_T\times\bb{R}^2\to\bb{R}^2$ and the terminal value $\xi$ are supposed to satisfy the following conditions:
 \begin{itemize}
\item[\textbf{(H1)}]$ \int_0^T|f(t,\omega,0,0)|^2dt+\int_0^T|g(t,\omega,0,0)|^2dt +|\xi(\omega)|\leq M_0$, q.s.;
\item[\textbf{(H2)}] The generator $(f,g)$ is uniformly continuous in $(t,\omega)$, i.e. there is a non-decreasing continuous function $w: [0, +\infty)\to [0, +\infty)$ such that $w(0)=0$ and
    $$\sup_{y,z\in \bb{R}}|f(t,\omega,y,z)-f(t^\prime,\omega^\prime,y,z)|\leq w(|t-t^\prime|+\norm{\omega-\omega^\prime}_\infty),$$
        $$\sup_{y,z\in \bb{R}}|g(t,\omega,y,z)-g(t^\prime,\omega^\prime,y,z)|\leq w(|t-t^\prime|+\norm{\omega-\omega^\prime}_\infty);$$
\item[\textbf{(H3)}] There are two positive constants $L_y$ and $L_z$ such that for each $(t,\omega)\in [0,T]\times\Omega,$
    $$|f(t,\omega,y,z)-f(t,\omega,y^\prime,z^\prime)|+|g(t,\omega,y,z)-g(t,\omega,y^\prime,z^\prime)|\leq L_y|y-y^\prime|+L_z(1+|z|+|z^\prime|)|z-z^\prime|;$$
\end{itemize}
\begin{Remark}
In \cite{Hu2018}, the triple $(f, g, \xi)$ is supposed to satisfy the following condition:
 \begin{itemize}
\item[\textbf{(H1')}] For each $t\in [0,T],\quad |f(t,\omega,0,0)|+|g(t,\omega,0,0)|+|\xi(\omega)|\leq M_0$, q.s.
\end{itemize}
The results there still hold if \textbf{(H1')} is replaced with \textbf{(H1)}, by a similar analysis as in \cite{HuM2014_1} and \cite{Hu2018}.
\end{Remark}
\begin{Remark}
\label{remark_f_bound}
Assumption \textbf{(H3)} implies the following
$$|h(t,\omega,y,z)|\leq |h(t,\omega,0,0)|+L_y|y|+L_z(|z|+|z|^2)\leq |h(t,\omega,0,0)|+{1\over2}L_z+L_y|y|+{3\over 2}L_z|z|^2$$
with $h=f, g.$ So $(f, g)$ are linear in $y$ and quadratic in $z$.
\end{Remark}

For simplicity, we denote by $ \mathfrak{G}_G^p(0,T)$ the collection of process $(Y,Z,K)$ such that $(Y,Z)\in S_G^p(0,T)\times H_G^p(0,T)$ and $K$ is a decreasing $G$-martingale with $K_0=0$ and $K_T\in L_G^p(\Omega_T)$.
Hu et al.~\cite[Theorem 5.3, page 22; Eq (3.2) and (3.3), page 13]{Hu2018} give the following theorem.

\begin{Theorem}
\label{Thm_QBSDE}
Assume that $\xi\in L_G^\infty(\Omega_T)$ and the triple $(f, g, \xi)$ satisfies \textbf{(H1)}-\textbf{(H3)}. Then equation (\ref{QGBSDE}) has a unique solution $(Y,Z,K)\in \mathfrak{G}_G^2(0,T)$ such that
$$\norm{Y}_{S_G^\infty}+\norm{Z}_{BMO_G}\leq C(M_0,L_y,L_z),$$ and
$$\E[|K_T|^p]\leq  C(p,M_0,L_y,L_z),\quad \forall p\geq 1,$$
where the norm $\norm{\cdot}_{BMO_G}$ will be defined in Section \ref{section_BMO}.
\end{Theorem}

\subsection{Formulation of the problem}
For simplicity, we consider the $G$-expectation space $(\Omega,L_G^1(\Omega_T),\E)$ for the case of $d=1$
and $\siup^2=\E[B_1^2]\geq-\E[-B_1^2]=\sidown^2>0$. Consider the following equation:

\begin{equation}
\begin{cases}
\displaystyle
 Y_t=\xi+\int_t^Tg(s,\omega_{\cdot \land s},Y_s,Z_s)ds+\int_t^Tf(s,\omega_{\cdot\land s},Y_s,Z_s)d\Bq_s\\[3mm]
 \displaystyle
 \qquad-\int_t^TZ_sdB_s+\int_t^TdA_s,~~\text{q.s.},~t\in[0,T];\\[3mm]
 \displaystyle
Y_t\geq S_t, ~~\text{q.s.} \ t\in [0, T];\\
\text{\rm the process }-\int_0^\cdot (Y_s-S_s)dA_s \text{  \rm is a non-increasing  $G$-martingale on $[0,T]$,}
 \label{RGBSDE1}
 \end{cases}
 \end{equation}
where the generator $(f, g):[0,T]\times\Omega_T\times\bb{R}^2\to\bb{R}^2$ and the terminal value $\xi$ are assumed to satisfy \textbf{(H1)}-\textbf{(H3)}.
Moreover, the obstacle process $\{S_t,~t\in[0,T]\}$ is supposed to satisfy the following conditions:
\begin{itemize}
\item[\textbf{(H4)}]  $S_\cdot\in \bigcap\limits_{\alpha>1} S_G^\alpha(0,T)$ with $S_T\le \xi$, q.s.  Furthermore,  there is a positive constant $N_0$ such that $S_t\leq N_0$, q.s.,  for any $t\in[0,T]$.
\item[\textbf{(H5)}]     $S$ is uniformly continuous in $(t,\omega)$, i.e. there is a non-decreasing continuous function $w: [0, +\infty)\to [0, +\infty)$ with  $w(0)=0$  such that
    $$|S_t(\omega)-S_{t^\prime}(\omega^\prime)|\leq w(|t-t^\prime|+\norm{\omega-\omega^\prime}_\infty).$$
\end{itemize}

\begin{Remark}
\label{Remark_tech_condition}
Like  in \cite{Hu2018},   Assumptions \textbf{(H2)} and \textbf{(H5)} are used to ensure the existence of solutions to our subsequent penalized quadratic $G$-BSDEs.
\end{Remark}

A solution of reflected $G$-BSDEs is defined as follows.

\begin{Definition}
\label{def_solution}
A triple of processes $(Y,Z,A)$ belongs to $\mathcal{S}_G^\alpha(0,T)$ for $\alpha>1$ if $(Y, Z)\in S_G^\alpha(0,T)\times  H_G^\alpha(0,T)$ and $A$ is a continuous nondecreasing process such that  $A_0=0$ and $A_T\in L_G^\alpha(\Omega_T)$. The triple (Y,Z,A) is said to be a solution to the reflected $G$-BSDE~\eqref{RGBSDE1} if ~$(Y,Z,A)\in \mathcal{S}_G^\alpha(0,T)$, and  satisfies~\eqref{RGBSDE1} for $t\in [0,T]$.
\end{Definition}

Our objective is to establish the existence and uniqueness result for the  quadratic $G$-BSDE~\eqref{RGBSDE1}.
For simplicity of exposition, we assume that  $g\equiv 0$ in what follows. Corresponding results still hold for the case of $g\neq 0.$

\subsection{$G$-BMO martingales and $G$-Girsanov Theorem}
\label{section_BMO}
We now introduce some results of $G$-BMO martingale and $G$-Girsanov Theorem in \cite{Hu2018} and \cite{PossZhou2013}.
\begin{Definition}
For $Z\in H_G^2(0,T)$, a symmetric $G$-martingale $\int_0^\cdot Z_sdB_s$ on $[0,T]$ is called a $G$-BMO martingale if
$$\norm{Z}_{BMO_G}^2:=\sup_{\bb{P}\in\mathcal{P}}\norm{Z}_{BMO(\bb{P})}^2
=\sup_{\bb{P}\in\mathcal{P}} \bigg[\sup_{\tau\in\mathcal{T}_0^T}\norm{E_\tau^{\bb{P}}\Big[\int_\tau^T|Z_t|^2d\Bq_t\Big]}_{L^\infty(\bb P)} \bigg]<+\infty,$$
where $\mathcal{T}_0^T$ denotes the totality of all $\mathcal{F}$-stopping times taking values in $[0,T]$ and $\norm{Z}_{BMO(\bb{P})}$ stands for the BMO norm of $\int_0^\cdot Z_sdB_s$ under probability measure $\bb{P}$.
\end{Definition}

Set
$$BMO_G:=\{Z\in H_G^2(0,T):\norm{Z}_{BMO_G}<+\infty\}.$$

In a straightforward manner,  we have the following important norm estimate for a $G$-BMO martingale $\int_0^\cdot Z_sdB_s$.

\begin{Lemma}
\label{energy_ineq} For $Z\in BMO_G$, we have for each ~$t\in[0,T]$,
$$\E_t\Big[\Big(\int_t^T|Z_s|^2d\Bq_s\Big)^{\alpha\over2}\Big]\leq C_\alpha\norm{Z}_{BMO_G}^\alpha,~~\text{q.s.},\quad\forall\alpha\geq1,$$
where $C_\alpha$ is a positive constant depending on $\alpha$.
\end{Lemma}
\begin{proof}
Fix some $(t,\bb{P})\in[0,T]\times\mathcal{P}$. In view of \cite[Corollary 2.1, page 28]{Kazamaki1994}, for each $\bb{P^\prime}\in\mathcal{P}(t,\bb P)$ we have
$$E^{\bb{P^\prime}}_t\Big[\Big(\int_t^T|Z_s|^2d\Bq_s\Big)^{\alpha\over2}\Big]\leq C_\alpha\norm{Z}_{BMO(\bb{P^\prime})}^\alpha\leq C_\alpha\norm{Z}_{BMO_G}^\alpha,~ \bb P^\prime\text{\,-a.s.},$$
where $C_\alpha$ is a positive constant depending only on $\alpha$. In view of the definition of $\mathcal{P}(t,\bb P)$, we have
$$E^{\bb{P^\prime}}_t\Big[\Big(\int_t^T|Z_s|^2d\Bq_s\Big)^{\alpha\over2}\Big]\leq C_\alpha\norm{Z}_{BMO_G}^\alpha,~ \bb P\text{\,-a.s.}$$
In view of Theorem \ref{Thm_conditional_E} and noting that $C_\alpha$ is independent of $\bb{P}^\prime$, we have
$$\E_t\Big[\Big(\int_t^T|Z_s|^2d\Bq_s\Big)^{\alpha\over2}\Big]=\operatorname*{ess\sup}_{\bb P^\prime\in \mathcal{P}(t,\bb P)}E^{\bb{P}^\prime}_t\Big[\Big(\int_0^T|Z_t|^2d\Bq_t\Big)^
{\alpha\over2}\Big]\leq C_\alpha\norm{Z}_{BMO_G}^\alpha,~ \bb P\text{\,-a.s.}$$
Notice that $C_\alpha$ is independent of $\bb{P}$ and we get the lemma.
\end{proof}

Like in the classical stochastic analysis, a $G$-BMO martingale can be used to define  an exponential $G$-martingale.  Hu et al.~\cite[Lemma 3.2, page 11]{Hu2018} give the following lemma.

\begin{Lemma}
For $Z\in BMO_G$, the process
$$\Exp(Z)_t:=\exp\left(\int_0^t Z_s dB_s-{1\over2}\int_0^t|Z_s|^2d\Bq_s \right), \quad t\ge 0$$
is a symmetric $G$-martingale.
\end{Lemma}
Similarly to Possama\"{i} and Zhou \cite{PossZhou2013}, we have the following lemmas.

\begin{Lemma}
\label{reverse holder}
(Reverse H\"{o}lder Inequality) Let $\phi(x)=\big\{1+{1\over x^2}\log{2x-1\over2(x-1)}\big\}^{1\over2}-1$ and $1<q<+\infty.$ If ~$\norm{Z}_{BMO_G}<\phi(q),$ we have
$$\sup_{\bb P\in \mathcal{P}}\sup_{\tau\in\mathcal{T}_0^T}\norm{E^{\bb{P}}_\tau\Big[\Big\{{\Exp(Z)_T\over\Exp(Z)_\tau}\Big\}^q\Big]}_{L^\infty(\bb P)}\leq C_q $$
for a constant $C_q>0$ depending only on $q$.
\end{Lemma}

\begin{proof}
For each $\bb{P}\in\mathcal{P}$, $$\norm{Z}_{BMO(\bb{P})}\leq\norm{Z}_{BMO_G}<\phi(q).$$
Then,  from \cite[Theorem 3.1, page 54]{Kazamaki1994}, we have
$$\sup_{\tau\in\mathcal{T}_0^T}\norm{E^{\bb{P}}_\tau\Big[\Big\{{\Exp(Z)_T\over\Exp(Z)_\tau}\Big\}^q\Big]}_{L^\infty(\bb P)}\leq C_q ,\quad\forall \bb{P}\in\mathcal{P}$$
for a positive constant  $C_q$ which does not  dependent on $\bb{P}$.
\end{proof}
\begin{Lemma}
\label{A_p}
 Let $1<r<+\infty.$ If ~$\norm{Z}_{BMO_G}<{\sqrt{2}\over2}(\sqrt{r}-1),$
$$\sup_{\bb P\in \mathcal{P}}\sup_{\tau\in\mathcal{T}_0^T}\norm{E^{\bb{P}}_\tau\Big[\Big\{{\Exp(Z)_\tau\over\Exp(Z)_T}\Big\}^{1\over r-1}\Big]}_{L^\infty(\bb P)}\leq C_r$$
holds with a constant $C_r>0$ depending only on $r$.
\end{Lemma}
\begin{proof}
Similarly to the proof of Lemma \ref{reverse holder}, the
desired result is an immediate consequence of \cite[Theorem 2.4, page 33]{Kazamaki1994} for all $\bb{P}\in\mathcal{P}$.
\end{proof}
\begin{Remark}
\label{remark_Ap_reverse_holder}
Assume $\norm{Z}_{BMO_G}<\phi(q)$ for some $q\in (1,+\infty)$. Fix some $(t,\bb{P})\in[0,T]\times\mathcal{P}$. In view Lemma \ref{reverse holder}, we have for each $\bb P^\prime\in \mathcal{P}(t,\bb P)$,
$$E_t^{\bb P^\prime}\Big[\Big\{{\Exp(Z)_T\over \Exp(Z)_t}\Big\}^q\Big]\leq C_q ,\quad \bb P^\prime\text{\,-a.s.}.$$
In view of the the definition of $\mathcal{P}(t,\bb P)$, we have
$$E_t^{\bb P^\prime}\Big[\Big\{{\Exp(Z)_T\over \Exp(Z)_t}\Big\}^q\Big]\leq C_q ,\quad \bb P\text{\,-a.s.}$$
Thus in view of Theorem \ref{Thm_conditional_E}, we get
$$\E_t\Big[\Big\{{\Exp(Z)_T\over \Exp(Z)_t}\Big\}^q\Big]=\operatorname*{ess\sup}_{\bb P^\prime\in \mathcal{P}(t,\bb P)}E_t^{\bb P^\prime}\Big[\Big\{{\Exp(Z)_T\over \Exp(Z)_t}\Big\}^q\Big]\leq C_q ,\quad \bb P\text{\,-a.s.}$$
Noting that $C_q$ is independent of ~$\bb P$, we have the following reverse H\"{o}lder inequality,
$$\E_t\Big[\Big\{{\Exp(Z)_T\over \Exp(Z)_t}\Big\}^q\Big]\leq C_q,~~\text{q.s.}$$
Similarly, in view of Theorem \ref{Thm_conditional_E} and Lemma \ref{A_p}, we have
$$\E_t\Big[\Big\{{\Exp(Z)_t\over\Exp(Z)_T}\Big\}^{1\over r-1}\Big]\leq C_r,~~\text{q.s.},$$
if $\norm{Z}_{BMO_G}<{\sqrt{2}\over2}(\sqrt{r}-1)$ for some $r\in (1,+\infty)$.
\end{Remark}
\begin{Remark}
The reverse H\"{o}lder inequality in Remark \ref{remark_Ap_reverse_holder} is used in the proof of Hu et al. \cite[Lemma 3.4]{Hu2018}. We give a proof here for convenience of the reader.
\end{Remark}
\begin{Remark}
\label{remark_Ap_Rp}
Suppose there exist $\{Z^n\}_{n\in\bb{N}}\subseteq H_G^2(0,T)$ such that $\norm{Z^n}_{BMO_G}\leq M$ for all $n\in\bb{N}$. Taking $t=0$ in Reamrk \ref{remark_Ap_reverse_holder}, we can know that there exist $q>1$ and $r>1$ which are depending only on $M$ such that:
$$\E[\Exp(Z^n)_T^q]\leq C_q\quad\text{and}\quad\E\Big[\Big\{\Exp(Z^n)_T\Big\}^{1\over 1-r}\Big]\leq C_r.$$
\end{Remark}
With the exponential martingale, we can generalize the Girsanov theorem. In \cite{Hu2018}, we know that we can define a new $G$-expectation $\Et[\cdot]$ with $\Exp(Z)$ satisfying
\begin{equation}
\Et[X]=\sup_{\bb{P}\in\mathcal{P}}E^\bb{P}[\Exp(Z)_TX]=\E[\Exp(Z)_TX],\quad \forall X\in L^p_G(\Omega_T),
\end{equation}
where $p>{q\over q-1}$ and $q$ is the order in the reverse H\"{o}lder inequality for $\Exp(Z)$.
Moreover, the conditional expectation $\Et_t[\cdot]$ is well-defined following the procedure introduced in \cite{Hu2018} and \cite{Xu2011}. And we have
\begin{equation}
\Et_t[X]=\E_t\Big[{\Exp(Z)_T\over\Exp(Z)_t}X\Big],~~\text{q.s.},~~\forall X\in L^p_G(\Omega_T).
\end{equation}

The following two lemmas give the Girsanov theorem in the $G$-framework, and can be found in Hu et al.~\cite{Hu2018}.

\begin{Lemma}
Suppose that $Z\in BMO_G$. We define a new $G$-expectation $\Et[\cdot]$ by $\Exp(Z)$. Then the process $B-\int Zd\Bq$ is a $G$-Brownian motion under $\Et[\cdot]$.
\end{Lemma}

\begin{Lemma}
\label{lemma_Gisr_K}
Suppose that $Z\in BMO_G$. We define a new $G$-expectation $\Et[\cdot]$ by $\Exp(Z)$. Suppose that $K$ is a decreasing $G$-martingale such that $K_0=0$ and for some $p>{q\over q-1}, K_t\in L_G^p(\Omega_t), 0\leq t\leq T,$ where $q$ is the order in the reverse H\"{o}lder inequality for $\Exp(Z)$. Then $K$ is a decreasing $G$-martingale under $\Et[\cdot]$.
\end{Lemma}

\section{A priori estimates for solutions of reflected quadratic $G$-BSDEs}

With $G$-BMO martingale and $G$-Girsanov Theorem, we have the following comparison theorem for quadratic $G$-BSDEs.

\begin{Theorem}
	\label{Thm_compare}
	Let the triplet $(\xi^i,f^i,g^i)$ satisfy \textbf{(H1)}-\textbf{(H3)} for $i=1,2$. Let $(Y^i,Z^i,K^i)\in \mathfrak{G}_G^2(0,T)$  be the solution to the following $G$-BSDE:
\begin{eqnarray*}
 Y_t^i&=&\xi^i+\int_t^Tg^i(s,Y_s^i,Z_s^i)ds+\int_t^Tf^i(s,Y_s^i,Z_s^i)d\Bq_s+\int_t^T dV_s^i\\
	&&-\int_t^TZ_s^idB_s-\int_t^T dK_s^i,~~\text{q.s.},~t\in[0,T],
\end{eqnarray*}
	where $V^i$ is a continuous finite variation process, for $i=1,2$. Assume that $$(Y^i,Z^i,K^i_T,V^i)\in S_G^\infty(0,T)\times BMO_G\times \bigcap_{p\geq 1} L_G^p(\Omega_T)\times \bigcap_{p\geq 1} S_G^p(0,T),$$ and $K^i$ is a decreasing $G$-martingale.
	If $\xi^1\geq \xi^2$, $g^1\geq g^2$, $f^1\geq f^2$,~~\text{q.s.} and $V^1-V^2$ is an increasing process, then we have $Y_t^1\geq Y_t^2$,~~\text{q.s.},~ for any $t\in[0,T]$.
\end{Theorem}
\begin{proof}
	Without loss of generality, we assume that $g^1=g^2=0$.
	
Define $\hat{\xi}:=\xi^1-\xi^2$  and for $t\in [0,T]$,
$$\hat{Y_t}:=Y_t^1-Y_t^2,\quad \hat{Z_t}:=Z_t^1-Z_t^2,\quad \hat K_t:=K^1_t-K^2_t, \quad \hat V_t:=V^1_t-V^2_t,$$
and $\hat{f_t}:=f^1(t,Y_t^2,Z_t^2)-f^2(t,Y_t^2,Z_t^2).$
Like  in the proof of \cite[Proposition 3.5]{Hu2018}, 	we use the method of  linearization  to write
$$\hat Y_t=\hat\xi+\int_t^T(\hat f_s+\hat m_s^\varepsilon+\hat a_s^\varepsilon\hat Y_s+\hat b_s^\varepsilon\hat Z_s)d\Bq_s-\int_t^T \hat Z_s dB_s-\int_t^T d\hat K_s+\int_t^Td\hat V_s,~~\text{q.s.},$$
where  for
	$0\leq s\leq T,$
	\begin{eqnarray*}
		\hat a_s^\varepsilon&:=&[1-l(\hat Y_s)]
		{f^1(s,Y_s^1,Z_s^1)-f^1(s,Y_s^2,Z_s^1)\over
			\hat Y_s}\textbf{1}_{\{|\hat Y_s|>0\}}, \\
		\hat b_s^\varepsilon&:=&[1-l(\hat Z_s)]
		{f^1(s,Y_s^2,Z_s^1)-f^1(s,Y_s^2,Z_s^2)\over
			|\hat Z_s|^2}\hat Z_s\textbf{1}_{\{|\hat Z_s|>0\}}, \\
		\hat m_s^\varepsilon&:=& l(\hat Y_s)
		[f^1(s,Y_s^1,Z_s^1)-f^1(s,Y_s^2,Z_s^1)]+l(\hat Z_s)[f^1(s,Y_s^2,Z_s^1)-f^1(s,Y_s^2,Z_s^2)]
	\end{eqnarray*}
	for a  scalar Lipschitz continuous function $l$ such that $\textbf{1}_{[-\varepsilon,\varepsilon]}(x)\leq l(x)\leq \textbf{1}_{[-2\varepsilon,2\varepsilon]}(x)$ with $x\in (-\infty,+\infty)$.
	We also have
	\begin{eqnarray*}
		|\hat a_s^\varepsilon|&\leq& L_y,\quad |\hat b_s^\varepsilon|\leq L_z(1+|Z_s^1|+|Z_s^2|),\\
		|\hat m_s^\varepsilon|&\leq& 2\varepsilon(L_y+L_z(1+ 2\varepsilon+2|Z_s^1|)).
	\end{eqnarray*}
	Define $\tilde{B}_t:=B_t-\int_0^t\hat b_s^\varepsilon d\Bq_s$ for $t\in [0,T]$.
	In view of \cite[Lemma 3.6]{Hu2018}, we know that $\hat b^\varepsilon\in BMO_G$. Therefore,  we can define a new $G$-expectation $\Et[\cdot]$ by $\Exp(\hat b^\varepsilon)$, such that $\tilde{B}$ is a $G$-Brownian motion under $\Et[\cdot]$. Then the last $G$-BSDE reads
	$$\hat Y_t=\hat\xi+\int_t^T(\hat f_s+\hat m_s^\varepsilon+\hat a_s^\varepsilon\hat Y_s)d\Bq_s-\int_t^T \hat Z_s d\tilde{B}_s-\int_t^T d\hat K_s+\int_t^Td\hat V_s,~~\text{q.s.}$$
	
	Applying It\^o's formula to $e^{\int_0^t\hat a_s^\varepsilon d\Bq_s}\hat Y_t$,~ we have
	\begin{eqnarray*}
		&&e^{\int_0^t\hat a_s^\varepsilon d\Bq_s}\hat Y_t\\
		&=& e^{\int_0^T\hat a_s^\varepsilon d\Bq_s}\hat\xi+\int_t^T e^{\int_0^s\hat a_u^\varepsilon d\Bq_u} \hat f_s d\Bq_s+\int_t^T e^{\int_0^s\hat a_u^\varepsilon d\Bq_u} \hat m_s^\varepsilon d\Bq_s\\
		&&-\int_t^T e^{\int_0^s\hat a_u^\varepsilon d\Bq_u} \hat Z_s d\tilde{B}_s- \int_t^T e^{\int_0^s\hat a_u^\varepsilon d\Bq_u} d\hat K_s+\int_t^T e^{\int_0^s\hat a_u^\varepsilon d\Bq_u} d\hat V_s\\
		&\geq& \int_t^T e^{\int_0^s\hat a_u^\varepsilon d\Bq_u} \hat m_s^\varepsilon d\Bq_s-\int_t^T e^{\int_0^s\hat a_u^\varepsilon d\Bq_u} \hat Z_s d\tilde{B}_s+\int_t^T e^{\int_0^s\hat a_u^\varepsilon d\Bq_u} dK_s^2,~~\text{q.s.}
	\end{eqnarray*}
	So we have
	$$-e^{\int_0^t\hat a_s^\varepsilon d\Bq_s}\hat Y_t +\int_t^T e^{\int_0^s\hat a_u^\varepsilon d\Bq_u} dK_s^2
	\leq -\int_t^T e^{\int_0^s\hat a_u^\varepsilon d\Bq_u} \hat m_s^\varepsilon d\Bq_s+\int_t^T e^{\int_0^s\hat a_u^\varepsilon d\Bq_u} \hat Z_s d\tilde{B}_s,~~\text{q.s.}$$
	In view of  Hu et al. \cite[Lemma 3.4]{HuM2014_1} and  Lemma \ref{lemma_Gisr_K}, we know $\int_0^\cdot e^{\int_0^s\hat a_u^\varepsilon d\Bq_u} dK_s^2$ is a decreasing $G$-martingale under both $\E[\cdot]$ and $\Et[\cdot]$.
	Taking conditional $G$-expectation on both sides, we have
	$$-e^{\int_0^t\hat a_s^\varepsilon d\Bq_s}\hat Y_t\leq \Et_t\Big[-\int_t^T e^{\int_0^s\hat a_u^\varepsilon d\Bq_u} \hat m_s^\varepsilon d\Bq_s\Big],~~\text{q.s.}$$
	Since $|\hat a_s^\varepsilon|\leq L_y$, we have
	$$\hat Y_t\geq -e^{2L_y\Bq_T}\Et_t\Big[\int_t^T  |\hat m_s^\varepsilon| d\Bq_s\Big],~~\text{q.s.}$$
	Finally, it remains to prove the limit
	$$\lim_{\varepsilon\to0}\Et_t\Big[\int_t^T |\hat m_s^\varepsilon|d\Bq_s\Big]=0,~~\text{q.s.}$$
	Let $\phi(x)=\big\{1+{1\over x^2}\log{2x\over2(x-1)}\big\}^{1\over2}-1$. We know there exist $p>1$ independent of $\varepsilon$ such that
	$$\norm{\hat b_s^\varepsilon}_{BMO_G}\leq\norm{L_z(1+|Z_1|+|Z_2|)}_{BMO_G}< \phi(p^\prime),$$
	where $p^\prime={p\over p-1}$.
	Then according to Lemma \ref{reverse holder}, for $X\in L_G^p(\Omega_T)$, we have
	$$\Et_t[X]=\E_t\Big[{\Exp(\hat b^\varepsilon)_T\over \Exp(\hat b^\varepsilon)_t}X\Big]\leq \E_t\Big[\Big({\Exp(\hat b^\varepsilon)_T\over \Exp(\hat b^\varepsilon)_t}\Big)^{p^\prime}\Big]^{1\over p^\prime}\E_t[|X|^p]^{1\over p}\leq C_p\E_t[|X|^p]^{1\over p},~~\text{q.s.}$$
	In view of Lemma \ref{energy_ineq}, we have
	\begin{eqnarray*}
		\Et_t\Big[\int_t^T |\hat m_s^\varepsilon|d\Bq_s\Big]&\leq&
		2\varepsilon\siup^2T(L_y+L_z+2L_z\varepsilon)
		+4\varepsilon\Et_t\Big[\int_t^T|Z_s^1|d\Bq_s\Big]\\
		&\leq&
		2\varepsilon\siup^2T(L_y+L_z+2L_z\varepsilon)
		+4\varepsilon\siup^2T\Et_t\Big[\int_t^T|Z_s^1|^2d\Bq_s\Big]^{1\over 2}\\
		&\leq&2\varepsilon\siup^2T(L_y+L_z+2L_z\varepsilon)
		+4\varepsilon C_p\siup^2T\E_t\Big[\Big(\int_t^T|Z_s^1|^2d\Bq_s\Big)^p\Big]^{1\over 2p}\\
		&\leq&2\varepsilon\siup^2T(L_y+L_z+2L_z\varepsilon)
		+4\varepsilon C_pC_p^{\prime\prime}\norm{Z^1}_{BMO_G},~~\text{q.s.}
	\end{eqnarray*}
	So we get $\lim_{\varepsilon\to0}\Et_t\Big[\int_t^T |\hat m_s^\varepsilon|d\Bq_s\Big]=0,~~\text{q.s.}$
\end{proof}

Consider the following type of BSDE:
 \begin{equation}
\begin{cases}
\displaystyle
 Y_t=\xi+\int_t^Tf(s,\omega_{\cdot\land s},Y_s,Z_s)d\Bq_s-\int_t^TZ_sdB_s+\int_t^TdA_s,~~\text{q.s.},~t\in[0,T];\\
 \displaystyle
Y_t\geq S_t,~~\text{q.s.},\quad t\in[0,T];\quad
\int_0^\cdot(S_s-Y_s)dA_s \text{ is a non-increasing }G\text{-martingale,}
 \label{RGBSDE2}
 \end{cases}
 \end{equation}
 with $A$ being a continuous nondecreasing process and  $A_0=0$.
\begin{Proposition}
\label{Propose_Z_bound}
Let $f$ satisfy \textbf{(H1)} and \textbf{(H3)}. Assume that $(Y,Z,A)$ solves
$$ Y_t=\xi+\int_t^Tf(s,Y_s,Z_s)d\Bq_s-\int_t^TZ_sdB_s+\int_t^TdA_s,~~\text{q.s.},\quad t\in[0,T],$$
where $$(Y,Z)\in S_G^\infty(0,T)\times H_G^2(0,T),$$
and $A$ is a continuous nondecreasing process with $A_0=0$.

Then there exist constant $C_1:=C_1(\norm{Y}_{S_G^\infty},T,L_z,L_y,M_0,\siup)$
such that $$\norm{Z}_{BMO_G}\leq C_1,$$
and constant $C_2:=C_2(\norm{Y}_{S_G^\infty},T,L_z,L_y,M_0,\siup,\alpha)$
for any $\alpha\geq 1$, such that $$\E[|A_T|^\alpha]\leq C_2.$$
\end{Proposition}
\begin{proof}
For each $\bb P\in\mathcal{P}$, we know
$$ Y_t=\xi+\int_t^Tf(s,Y_s,Z_s)d\Bq_s-\int_t^TZ_sdB_s+\int_t^TdA_s,~~\bb P\text{\,-a.s.},\quad t\in[0,T].$$
Then, for some $a>0~$, applying It\^o's formula under $\bb P$ to $e^{-aY_t}$,~ we have for each $\tau\in\mathcal{T}_0^T$,
\begin{eqnarray*}
&&{a^2\over 2}\int_\tau^Te^{-aY_s}Z_s^2d\Bq_s\\
&=&e^{-a\xi}-e^{-aY_\tau}-\int_\tau^Tae^{-aY_s}f(s,Y_s,Z_s)d\Bq_s+\int_\tau^Tae^{-aY_s}Z_sdB_s\\
&&-\int_\tau^Tae^{-aY_s}dA_s,~~\bb P\text{\,-a.s.}
\end{eqnarray*}
Since $A~$ is a continuous nondecreasing process, noting $a>0$ and Remark \ref{remark_f_bound}, we have
\begin{eqnarray*}
&&{a^2\over 2}\int_\tau^T e^{-aY_s}|Z_s|^2d\Bq_s\\
&&\leq e^{-a\xi}-e^{-aY_\tau}-\int_\tau^Tae^{-aY_s}f(s,Y_s,Z_s)d\Bq_s+\int_\tau^Tae^{-aY_s}Z_sdB_s\\
&&\leq e^{-a\xi}-e^{-aY_\tau}+\int_\tau^Tae^{-aY_s}(|f(s,0,0)|+{1\over 2}L_z+L_y|Y_s|)d\Bq_s\\
&&\quad+{3aL_z\over2}\int_\tau^Te^{-aY_s}|Z_s|^2d\Bq_s+\int_\tau^Tae^{-aY_s}Z_sdB_s,~~\bb P\text{\,-a.s.}
\end{eqnarray*}
Taking $a=4L_z$, noting $Y\in S_G^\infty(0,T)$ and taking conditional expectations under $\bb{P}$ on both sides, we have
\begin{eqnarray*}
&&2L_z^2E^{\bb{P}}_{\tau}\Big[\int_\tau^T e^{-aY_s}|Z_s|^2d\Bq_s\Big]\\
&&\leq E^{\bb{P}}_{\tau}\Big[e^{-a\xi}-e^{-aY_\tau}+\int_\tau^Tae^{-aY_s}(|f(s,0,0)|+{1\over 2}L_z+L_y|Y_s|)d\Bq_s\Big]\\
&&\leq 2e^{4L_z\norm{Y}_{S_G^\infty}}+4L_z\siup^2\Big(\sqrt{T} E^{\bb{P}}_{\tau}\Big[\Big(\int_0^T|f(s,0,0)|^2ds\Big)^{1\over 2}\Big]
    +{1\over 2}L_zT+L_yT\norm{Y}_{S_G^\infty}\Big)e^{4L_z\norm{Y}_{S_G^\infty}}\\
&&\leq 2e^{4L_z\norm{Y}_{S_G^\infty}}+4L_z\siup^2\Big(\sqrt{TM_0}+{1\over 2}L_zT+L_yT\norm{Y}_{S_G^\infty}\Big)e^{4L_z\norm{Y}_{S_G^\infty}},~~\bb P\text{\,-a.s.}
\end{eqnarray*}
Then with the arbitrariness of $\tau$, we obtain for all $\bb{P}\in\mathcal{P}$,
$$\norm{Z}_{BMO(\bb{P})}^2\leq {1\over L_z^2}e^{8L_z\norm{Y}_{S_G^\infty}}+{2\over L_z}\siup^2\Big(\sqrt{TM_0}
+{1\over 2}L_zT+L_yT\norm{Y}_{S_G^\infty}\Big)e^{8L_z\norm{Y}_{S_G^\infty}}.$$
Finally, with the arbitrariness of $\bb{P}$, we get
$$\norm{Z}_{BMO_G}^2\leq {1\over L_z^2}e^{8L_z\norm{Y}_{S_G^\infty}}+{2\over L_z}\siup^2\Big(\sqrt{TM_0}
+{1\over 2}L_zT+L_yT\norm{Y}_{S_G^\infty}\Big)e^{8L_z\norm{Y}_{S_G^\infty}}.$$
Now we get the estimate for $Z$. We have
$$A_T=Y_0-\xi-\int_0^Tf(s,Y_s,Z_s)d\Bq_s+\int_0^TZ_sdB_s,~~\text{q.s.}$$
In view of BDG inequality and Remark \ref{remark_f_bound}, we have for each $\alpha\geq1$
\begin{eqnarray*}
\E[A_T^\alpha]&\leq& C_\alpha\E[|Y_0|^\alpha+|\xi|^\alpha]+C_\alpha\E\Big[\Big(\int_0^T|f(s,Y_s,Z_s)|d\Bq_s\Big)^\alpha\Big]\\
&&+C_\alpha\E\Big[\Big(\int_0^T|Z_t|^2d\Bq_t\Big)^{\alpha\over2}\Big]\\
&\leq& 2C_\alpha\norm{Y}_{S_G^\infty}^\alpha+\tilde C_\alpha\E\Big[\Big(\int_0^T|f(s,0,0)|d\Bq_s\Big)^\alpha
+\Big(\int_0^T{1\over 2}L_z+L_y|Y_s|d\Bq_s\Big)^\alpha\Big]\\
&&+{3\tilde C_\alpha L_z\over2}\E\Big[\Big(\int_0^T|Z_t|^2d\Bq_t\Big)^{\alpha}\Big]+C_\alpha\E\Big[\Big(\int_0^T|Z_t|^2d\Bq_t\Big)^{\alpha\over2}\Big]\\
&\leq& 2C_\alpha\norm{Y}_{S_G^\infty}^\alpha+\tilde C_\alpha\siup^{2\alpha}\Big\{\E\Big[\Big(T\int_0^T|f(s,0,0)|^2ds\Big)^{\alpha\over 2}\Big]
+\Big({1\over 2}L_zT+L_yT\norm{Y}_{S_G^\infty}\Big)^\alpha\Big\}\\
&&+{3\tilde C_\alpha L_z\over2}\E\Big[\Big(\int_0^T|Z_t|^2d\Bq_t\Big)^{\alpha}\Big]+C_\alpha\E\Big[\Big(\int_0^T|Z_t|^2d\Bq_t\Big)^{\alpha\over2}\Big].
\end{eqnarray*}
In view of Lemma \ref{energy_ineq}, we have
\begin{eqnarray*}
\E[A_T^\alpha]&\leq& 2C_\alpha\norm{Y}_{S_G^\infty}^\alpha+\tilde C_\alpha\siup^{2\alpha}\Big\{(TM_0)^{\alpha\over 2}
+\Big({1\over 2}L_zT+L_yT\norm{Y}_{S_G^\infty}\Big)^\alpha\Big\}\\
&&+{3\tilde C_\alpha\bar C_{2\alpha} L_z\over2}\norm{Z}_{BMO_G}^{2\alpha}+C_\alpha
\bar C_\alpha\norm{Z}_{BMO_G}^{\alpha}.
\end{eqnarray*}
Substituting the estimate for $Z$, we get the estimate for $A$.
\end{proof}
\begin{Proposition}
\label{Propose_Y_bound}
Let $(\xi,f,S)$ satisfy \textbf{(H1)}, \textbf{(H3)} and \textbf{(H4)}. Assume that the triplet
$(Y,Z,A)\in\mathcal{S}_G^{2p}(0,T)$ with some $p>1$, is a solution to the reflected $G$-BSDE with data $(\xi,f,S)$. Moreover,
 we suppose
$$\norm{L_z(1+|Z|)}_{BMO_G}< \phi(q):=\bigg\{1+{1\over q^2}\log{2q-1\over2(q-1)}\bigg\}^{1\over2}-1,$$
with $q$ satisfying $p>{q\over q-1}$.\\
Then there exists a constant $~C:=C(T,L_z,L_y,\siup,N_0)~$ such that
$$\norm{Y_t}_{L_G^\infty}\leq C\Bigg(1+\norm{\xi}_{L_G^\infty}+\norm{\int_0^T|f(s,0,0)|^2ds}_{L_G^\infty}^{1\over 2}\Bigg),\quad \forall t\in[0,T].$$
\end{Proposition}
\begin{proof}
For some $r>0~$, applying It\^o's formula to $e^{rt}|Y_t-N_0|^2$,~ we have for each $t\in[0,T]$,
\begin{eqnarray*}
&&e^{rt}|Y_t-N_0|^2+r\int_t^T e^{rs}|Y_s-N_0|^2ds+\int_t^T e^{rs}|Z_s|^2d\Bq_s\\
&=&e^{rT}|\xi-N_0|^2+\int_t^T 2e^{rs}(Y_s-N_0)f(s,Y_s,Z_s)d\Bq_s\\
&&-\int_t^T 2e^{rs}(Y_s-N_0)Z_sdB_s+\int_t^T 2e^{rs}(Y_s-N_0)dA_s,~~\text{q.s.}
\end{eqnarray*}
We have
$$f(s,Y_s,Z_s)=f(s,0,0)+m_s^\varepsilon+a_s^\varepsilon Y_s+b_s^\varepsilon Z_s,$$
 where
\begin{eqnarray*}
a_s^\varepsilon&:=&[1-l(Y_s)]{f(s,Y_s,Z_s)-f(s,0,Z_s)\over Y_s}\textbf{1}_{\{|Y_s|>0\}},\\
b_s^\varepsilon&:=&[1-l(Z_s)]{f(s,0,Z_s)-f(s,0,0)\over |Z_s|^2}Z_s\textbf{1}_{\{|Z_s|>0\}},\\
m_s^\varepsilon&:=& l(Y_s)[f(s,Y_s,Z_s)-f(s,0,Z_s)]+l(Z_s)[f(s,0,Z_s)-f(s,0,0)]
\end{eqnarray*}
with $s\in [0, T]$ and the function $l$ being Lipschitz continuous such that $\textbf{1}_{[-\varepsilon,\varepsilon]}(x)\leq l(x)\leq \textbf{1}_{[-2\varepsilon,2\varepsilon]}(x)$ for $x\in (-\infty, +\infty)$. Moreover,
\begin{eqnarray*}
|a_s^\varepsilon|&\leq& L_y,\quad |b_s^\varepsilon|\leq L_z(1+|Z_s|),\quad
|m_s^\varepsilon|\leq2\varepsilon(L_y+L_z(1+ 2\varepsilon)).
\end{eqnarray*}
In view of \cite[Lemma 3.6]{Hu2018}, we know that $b^\varepsilon\in BMO_G$. Set $\tilde{B}_t:=B_t-\int_0^tb_s^\varepsilon d\Bq_s$ for $t\in [0,T]$. Thus we can define a new $G$-expectation $\Et[\cdot]$ by $\Exp(b_s^\varepsilon)$, such that $\tilde{B}$ is a $G$-Brownian motion under $\Et[\cdot]$.
Then we have for each $t\in[0,T]$,
\begin{eqnarray*}
&&e^{rt}|Y_t-N_0|^2+r\int_t^T e^{rs}|Y_s-N_0|^2ds+\int_t^T e^{rs}|Z_s|^2d\Bq_s\\
&\leq&e^{rT}|\xi-N_0|^2+\int_t^T 2e^{rs}(Y_s-N_0)(f(s,0,0)+m_s^\varepsilon+a_s^\varepsilon Y_s+b_s^\varepsilon Z_s)d\Bq_s\\
&&-\int_t^T 2e^{rs}(Y_s-N_0)Z_sdB_s+\int_t^T 2e^{rs}(Y_s-N_0)dA_s\\
&\leq&e^{rT}|\xi-N_0|^2+(1+2L_y)\int_t^T e^{rs}|Y_s-N_0|^2d\Bq_s-\int_t^T 2e^{rs}(Y_s-N_0)Z_sd\tilde{B}_s\\
&&+\int_t^T e^{rs}\big(f(s,0,0)+|m_s^\varepsilon|+N_0L_y\big)^2d\Bq_s+\int_t^T 2e^{rs}(Y_s-S_s)dA_s,~~\text{q.s.}
\end{eqnarray*}

Setting $r>\siup^2(1+2L_y)$ and taking conditional expectations on both sides, we have
\begin{eqnarray*}
&&e^{rt}|Y_t-N_0|^2+\Et_t\Big[-\int_t^T 2e^{rs}(Y_s-S_s)dA_s\Big]\\
&\leq& \Et_t[e^{rT}|\xi-N_0|^2]+\Et_t\Big[\int_t^T e^{rs}\big(f(s,0,0)+|m_s^\varepsilon|+N_0L_y\big)^2d\Bq_s\Big],~~\text{q.s.}
\end{eqnarray*}
From (\ref{RGBSDE2}), we know that $\{-\int_0^t(Y_s-S_s)dA_s\}_{t\in[0,T]}$ is a non-increasing G-martingale under $\E[\cdot]$.
Moreover,
$$\E\Big[\Big(\int_0^t(Y_s-S_s)dA_s\Big)^p\Big]\leq\E\Big[\sup_{s\in[0,T]}|Y_s-S_s|^p\Big(\int_0^TdA_s\Big)^p\Big]
\leq\E\big[\sup_{s\in[0,T]}|Y_s-S_s|^{2p}\big]^{1\over 2}\E[|A_T|^{2p}]^{1\over 2}.$$
Note  $p>{q\over q-1}$ and
$$\norm{b_s^\varepsilon}_{BMO_G}\leq\norm{L_z(1+|Z|)}_{BMO_G}< \phi(q).$$
In view of Lemma \ref{lemma_Gisr_K}, we know that $\{-\int_0^t(Y_s-S_s)dA_s\}_{t\in[0,T]}$ is a non-increasing G-martingale under $\Et[\cdot]$.
Then for each $t\in[0,T]$,
\begin{eqnarray*}
e^{rt}|Y_t-N_0|^2&\leq& \Et_t[e^{rT}|\xi-N_0|^2]+\Et_t\Big[\int_t^T e^{rs}\big(f(s,0,0)+|m_s^\varepsilon|+N_0L_y\big)^2d\Bq_s\Big]\\
&\leq& 2e^{rT}(\norm{\xi}_{L_G^\infty}^2+N_0^2)+2e^{rT}\siup^2\Bigg\{\big(2\varepsilon(L_y+L_z(1+ 2\varepsilon))+N_0L_y\big)^2T\\
&&+\norm{\int_0^T|f(t,0,0)|^2dt}_{L_G^\infty}\Bigg\},~~\text{q.s.}
\end{eqnarray*}
Let $\varepsilon\to 0$, we have
$$e^{rt}|Y_t-N_0|^2\leq 2e^{rT}(\norm{\xi}_{L_G^\infty}^2+N_0^2)+2e^{rT}\siup^2\Bigg(N_0^2L_y^2T+\norm{\int_0^T|f(t,0,0)|^2dt}_{L_G^\infty}\Bigg),~~\text{q.s.}$$
So we get the estimate for $Y$.
\end{proof}
\begin{Proposition}
\label{Propose_stable}
Let $(\xi^1,f^1,S^1)$ and $(\xi^2,f^2,S^2)$ be two sets of data, each one satisfying \textbf{(H1)}, \textbf{(H3)} and \textbf{(H4)}. Assume that the triplet $(Y^i,Z^i,A^i)\in\mathcal{S}_G^{2p}(0,T)$ with some $p>1$, is a solution of the reflected $G$-BSDE with data $(\xi^i,f^i,S^i),~i=1,2$. Moreover, we suppose
 $$\norm{L_z(1+|Z^1|+|Z^2|)}_{BMO_G}< \phi(q):=\left\{1+{1\over q^2}\log{2q-1\over2(q-1)}\right\}^{1\over2}-1.$$
with $q$ satisfying $p>{q\over q-1}$.
Then there exists a constant $~C_1:=C_1(q,T,L_z,L_y,N_0)~$ such that for each $t \in [0,T]$,
\begin{eqnarray*}
|Y_t^1-Y_t^2|^2&\leq& C_1\norm{\xi^1-\xi^2}_{L_G^\infty}^2
+C_1`\E_t\left[\left(\int_t^T |\hat\lambda_s|^2d\Bq_s\right)^p\right]^{1\over p}\\
&&+C_1\E_t\left[\sup_{s\in [t,T]}|S_t^1-S_t^2|^{2p}\right]^{1\over 2p}
\E_t\left[|A_T^1-A_t^1|^{2p}+|A_T^2-A_t^2|^{2p}\right]^{1\over 2p},~~\text{q.s.},
\end{eqnarray*}
where $$\hat\lambda_s:=f^1(s,Y_s^2,Z_s^2)-f^2(s,Y_s^2,Z_s^2).$$
Moreover, there exists a constant $~C_2:=C_2(T,L_z,L_y,\siup,N_0,M_0)~$ such that
$$\E_t\left[\int_t^T|Z_s^1-Z_s^2|^2d\Bq_s\right]\leq C_2
\norm{\hat Y}_{S_G^\infty}
\left(1+\E_t[|A_T^1-A_t^1|+|A_T^2-A_t^2|]\right),~~\text{q.s.},\quad\forall t \in [0,T].$$
\end{Proposition}
\begin{proof}
First, with Proposition \ref{Propose_Z_bound} and \ref{Propose_Y_bound}, we know that there exists a constant $C:=C(T,L_z,L_y,\siup,N_0,M_0)$ such that
\begin{equation}
\label{eq_Y_Z_bound}
\sum_{i=1}^2(\norm{Y^i}_{S_G^\infty}+\norm{Z^i}_{BMO_G})\leq C.
\end{equation}
Define
$$\hat{Y_t}:=Y_t^1-Y_t^2,~\hat{Z_t}:=Z_t^1-Z_t^2,~\hat{S_t}:=S_t^1-S_t^2,~\hat{\xi}:=\xi^1-\xi^2.$$
With the condition of $f^1$ and $f^2$, we see that
$\hat\lambda\in H_G^{2p}(0,T).$
As in the proof of Proposition \ref{Propose_Y_bound}, define  for $0\leq s\leq T,$
\begin{eqnarray*}
\hat a_s^\varepsilon&:=&[1-l(\hat Y_s)]
{f^1(s,Y_s^1,Z_s^1)-f^1(s,Y_s^2,Z_s^1)\over
\hat Y_s}\textbf{1}_{\{|\hat Y_s|>0\}},\\
\hat b_s^\varepsilon&:=&[1-l(\hat Z_s)]
{f^1(s,Y_s^2,Z_s^1)-f^1(s,Y_s^2,Z_s^2)\over
|\hat Z_s|^2}\hat Z_s\textbf{1}_{\{|\hat Z_s|>0\}}, \\
\hat m_s^\varepsilon&:=& l(\hat Y_s)
[f^1(s,Y_s^1,Z_s^1)-f^1(s,Y_s^2,Z_s^1)]+l(\hat Z_s)[f^1(s,Y_s^2,Z_s^1)-f^1(s,Y_s^2,Z_s^2)],
\end{eqnarray*}
where $l$ is a Lipschitz continuous function such that $\textbf{1}_{[-\varepsilon,\varepsilon]}(x)\leq l(x)\leq \textbf{1}_{[-2\varepsilon,2\varepsilon]}(x)$.   Also define $\hat A: =A^1-A^2$.  We have
$$\hat Y_t=\hat\xi+\int_t^T(\hat\lambda_s+\hat m_s^\varepsilon+\hat a_s^\varepsilon\hat Y_s+\hat b_s^\varepsilon\hat Z_s)d\Bq_s-\int_t^T \hat Z_s dB_s+\int_t^Td\hat A_s,~~\text{q.s.},~t\in[0,T],$$
and for each $s\in[0,T]$,
\begin{eqnarray*}
|\hat a_s^\varepsilon|&\leq& L_y,\quad |\hat b_s^\varepsilon|\leq L_z(1+|Z_s^1|+|Z_s^2|),\\
|\hat m_s^\varepsilon|&\leq& 2\varepsilon(L_y+L_z(1+ 2\varepsilon+2|Z_s^1|)).
\end{eqnarray*}
Then we have
$$\hat Y_t=\hat\xi+\int_t^T(\hat\lambda_s+\hat m_s^\varepsilon+\hat a_s^\varepsilon\hat Y_s)d\Bq_s-\int_t^T \hat Z_s d\tilde{B}_s+\int_t^Td\hat A_s,~~\text{q.s.},~t\in[0,T],$$
where $d\tilde{B}_s=dB_s-\hat b_s^\varepsilon d\Bq_s.$
Similarly to the proof of Proposition \ref{Propose_Y_bound}, we can define a new $G$-expectation $\Et[\cdot]$ by $\Exp(\hat b^\varepsilon)$, such that $\tilde{B}$ is a $G$-Brownian Motion under $\Et[\cdot]$.

For some $r>0~$, applying It\^o's formula to $e^{rt}|\hat Y_t|^2$, we have for each $t\in[0,T]$,
\begin{eqnarray*}
&&e^{rt}|\hat Y_t|^2+r\int_t^T e^{rs}|\hat Y_s|^2ds+\int_t^Te^{rs}|\hat Z_s|^2d\Bq_s\\
&=&e^{rT}|\hat\xi|^2+\int_t^T2e^{rs}\hat Y_s(\hat\lambda_s+\hat m_s^\varepsilon+\hat a_s^\varepsilon\hat Y_s)d\Bq_s-\int_t^T2e^{rs}\hat Y_sd\tilde B_s+\int_t^T2e^{rs}\hat Y_sd\hat A_s \\
&\leq& e^{rT}|\hat\xi|^2+\int_t^T e^{rs}(|\hat\lambda_s|^2+|\hat m_s^\varepsilon|^2)d\Bq_s+(2+2L_y)\int_t^T e^{rs}|\hat Y_s|^2d\Bq_s-\int_t^T2e^{rs}\hat Y_sd\tilde B_s \\
&&+\int_t^T2e^{rs}\hat S_sd\hat A_s+\int_t^T2e^{rs}(\hat Y_s-\hat S_s)d\hat A_s\\
&\leq& e^{rT}|\hat\xi|^2+\int_t^T e^{rs}(|\hat\lambda_s|^2+|\hat m_s^\varepsilon|^2)d\Bq_s+(2+2L_y)\int_t^T e^{rs}|\hat Y_s|^2d\Bq_s-\int_t^T2e^{rs}\hat Y_sd\tilde B_s \\
&&+\int_t^T2e^{rs}\hat S_sd\hat A_s+\int_t^T2e^{rs}(Y_s^1-S_s^1)dA_s^1
+\int_t^T2e^{rs}(Y_s^2-S_s^2)dA_s^2,~~\text{q.s.}
\end{eqnarray*}
In view of a similar argument as in the proof of Proposition \ref{Propose_Y_bound}, we see that $\int_0^\cdot (S_s^i-Y_s^i)\, dA_s^i$ is a non-increasing $G$-martingale on $[0,T]$ under $\Et[\cdot]$ for $i=1,2$.

Setting $r>\siup^2(2+2L_y)$ and taking conditional expectations on both sides, we have
\begin{eqnarray*}
e^{rt}|\hat Y_t|^2&\leq& \Et_t[e^{rT}|\hat\xi|^2]+\Et_t\Big[\int_t^T e^{rs}(|\hat\lambda_s|^2+|\hat m_s^\varepsilon|^2)d\Bq_s+\int_t^T2e^{rs}\hat S_sd\hat A_s\Big]\\
&\leq& e^{rT}\bigg\{\norm{\hat\xi}_{L_G^\infty}^2+\Et_t\Big[\int_t^T (|\hat\lambda_s|^2+|\hat m_s^\varepsilon|^2)d\Bq_s\Big]
+2\Et_t\Big[\sup_{s\in[t,T]}|\hat S_s||\hat A_T-\hat A_t|\Big]\bigg\},~~\text{q.s.}
\end{eqnarray*}
Note that $\norm{\hat b_s^\varepsilon}_{BMO_G}\leq\norm{L_z(1+|Z_1|+|Z_2|)}_{BMO_G}< \phi(q)<\phi(p^\prime)$ where $p^\prime={p\over p-1}$. Then according to Lemma \ref{reverse holder}, $\forall X\in L_G^p(\Omega_T)$, we have
for each $t\in[0,T]$,
$$\Et_t[X]=\E_t\Big[{\Exp(\hat b^\varepsilon)_T\over \Exp(\hat b^\varepsilon)_t}X\Big]\leq \E_t\Big[\Big({\Exp(\hat b^\varepsilon)_T\over \Exp(\hat b^\varepsilon)_t}\Big)^{p^\prime}\Big]^{1\over p^\prime}\E_t[X^p]^{1\over p}\leq C_p\E_t[X^p]^{1\over p},~~\text{q.s.}$$
So by the H\"{o}lder inequality, we have for each $t\in[0,T]$,
\begin{eqnarray*}
\Et_t\Big[\sup_{s\in[t,T]}|\hat S_s||\hat A_T-\hat A_t|\Big]&\leq&
C_p\E_t\Big[\sup_{s\in[t,T]}|\hat S_s|^p|\hat A_T-\hat A_t|^p\Big]^{1\over p}\\
&\leq& C_pC_p^\prime \E_t\Big[\sup_{s\in[t,T]}|\hat S_s|^p\Big]^{1\over 2p}\E_t\Big[|A_T^1-A_t^1|^{2p}+|A_T^2-A_t^2|^{2p}\Big]^{1\over 2p},~~\text{q.s.}
\end{eqnarray*}
 Then there exists a constant $~C_1:=C_1(q,T,L_z,L_y,N_0)~$ such that for each $t\in[0,T]$,
 \begin{eqnarray*}
|\hat{Y_t}^2|&\leq& C_1\bigg\{\norm{\hat\xi}_{L_G^\infty}^2
+\E_t\Big[\Big(\int_t^T |\hat\lambda_s|^2d\Bq_s\Big)^p\Big]^{1\over p}
+\Et_t\Big[\int_t^T |\hat m_s^\varepsilon|^2d\Bq_s\Big]\\
&&+\E_t\Big[\sup_{s\in [t,T]}|\hat{S_t}|^{2p}\Big]^{1\over 2p}
\E_t\Big[|A_T^1-A_t^1|^{2p}+|A_T^2-A_t^2|^{2p}\Big]^{1\over 2p}\bigg\},~~\text{q.s.}
\end{eqnarray*}
Finally, we just need to prove
$$\lim_{\varepsilon\to0}\Et_t\Big[\int_t^T |\hat m_s^\varepsilon|^2d\Bq_s\Big]=0,~~\text{q.s.}$$
In view of Lemma \ref{energy_ineq}, we have
 \begin{eqnarray*}
 \Et_t\Big[\int_t^T |\hat m_s^\varepsilon|^2d\Bq_s\Big]&\leq&
 8\varepsilon^2\siup^2T(L_y+L_z+2L_z\varepsilon)^2
 +32\varepsilon^2\Et_t\Big[\int_t^T|Z_s^1|^2d\Bq_s\Big]\\
 &\leq&8\varepsilon^2\siup^2T(L_y+L_z+2L_z\varepsilon)^2
 +32\varepsilon^2C_p\E_t\Big[\Big(\int_t^T|Z_s^1|^2d\Bq_s\Big)^p\Big]^{1\over p}\\
  &\leq&8\varepsilon^2\siup^2T(L_y+L_z+2L_z\varepsilon)^2
 +32\varepsilon^2C_pC_p^{\prime\prime}\norm{Z^1}_{BMO_G}^2,~~\text{q.s.}\\
 \end{eqnarray*}
 So we get $\lim_{\varepsilon\to0}\Et_t\Big[\int_t^T |\hat m_s^\varepsilon|^2d\Bq_s\Big]=0.$ And we get the estimate for $\hat Y.$

 Then we consider the estimate for $\hat Z.$ Applying It\^o's formula to $e^{rt}|\hat Y_t|^2$,~ we have for each $t \in [0,T]$,
 \begin{eqnarray*}
&&|\hat Y_t|^2+\int_t^T|\hat Z_s|^2d\Bq_s\\
&=&|\hat\xi|^2+\int_t^T2\hat Y_s(f^1(s,Y_s^1,Z_s^1)-f^2(s,Y_s^2,Z_s^2))d\Bq_s-\int_t^T2\hat Y_sdB_s+\int_t^T2\hat Y_sd\hat A_s ,~~\text{q.s.}
\end{eqnarray*}
Taking conditional expectations on both sides, we have
  \begin{eqnarray*}
&&\E_t\Big[\int_t^T|\hat Z_s|^2d\Bq_s\Big]\\
&\leq& \norm{\hat Y}_{S_G^\infty}\Big(M_0+2\sum_{i=1}^2\E_t\Big[\int_t^T|f^i(s,Y_s^i,Z_s^i)|d\Bq_s\Big]+
\E_t[|\hat A_T-\hat A_t|]\Big),~~\text{q.s.}
\end{eqnarray*}
Note that $\forall i=1,2,$
\begin{eqnarray*}
&&\E_t\Big[\int_t^T|f^i(s,Y_s^i,Z_s^i)|d\Bq_s\Big]\\
&\leq& \E_t\Big[\int_t^T\Big(|f^i(s,0,0)|+{L_z\over2}+L_y|Y_s^i|+{3L_z\over2}|Z_s^i|^2\Big)d\Bq_s\Big]\\
&\leq& \Big(\sqrt{M_0T}+{L_zT\over2}\Big)\siup^2+L_y\siup^2T\norm{Y^i}_{S_G^\infty}+{3L_z\over2}\norm{Z^i}_{BMO_G}^2,~~\text{q.s.}
\end{eqnarray*}
With (\ref{eq_Y_Z_bound}), we get the estimate for $\hat Z.$
 \end{proof}
 \begin{Remark}
 \label{Remark_uniqueness}
 The uniqueness for solutions to the reflected quadratic $G$-BSDE is an immediate consequence of  Proposition \ref{Propose_stable}.
 \end{Remark}

\section{Penalized $G$-BSDEs and their limit}
Similar to \cite{Li2017_1} and \cite{Li2017_2}, we use a penalized method. In this section, we first prove some convergence properties of solutions to the penalized $G$-BSDEs.
For $(f,\xi,S)$ satisfying (\textbf{H1})-(\textbf{H5}) and $n\in\bb N$, we  consider the following penalized $G$-BSDE:
\begin{equation}
\label{Penal_GBSDE1}
Y_t^n=\xi+\int_t^Tf(s,Y_s^n,Z_s^n)d\Bq_s+n\int_t^T(Y_s^n-S_s)^-ds-\int_t^TZ_s^ndB_s-\int_t^TdK_s^n,~~\text{q.s.},~t\in[0,T].
\end{equation}
Define  $L_t^n:=n\int_0^t(Y_s^n-S_s)^-ds$ for $t\in [0,T]$.  The  penalized $G$-BSDE reads:
\begin{equation}
\label{Penal_GBSDE2}
Y_t^n=\xi+\int_t^Tf(s,Y_s^n,Z_s^n)d\Bq_s-\int_t^TZ_s^ndB_s-\int_t^TdK_s^n+\int_t^TdL_s^n,~~\text{q.s.},~t\in[0,T].
\end{equation}
From Theorem \ref{Thm_QBSDE}, the penalized BSDE~(\ref{Penal_GBSDE1}) or (\ref{Penal_GBSDE2}) has a unique solution $(Y^n,Z^n,K^n)\in \mathfrak{G}_G^2(0,T)$ such that
$$\norm{Y^n}_{S_G^\infty}+\norm{Z^n}_{BMO_G}\leq C(M_0,L_y,L_z,n),$$ and
$$\E[|K_T^n|^p]\leq  C(p,M_0,L_y,L_z,n),\quad \forall p\geq 1.$$
Both estimates  depend on $n$. In fact,  $(Y^n,Z^n,K^n,L^n)$ is uniformly bounded in $n$.

\begin{Lemma}
\label{lemma_Y^n_bound}
There exists two positive constants $C$ and $C_p$ which are independent of $n$, such that
$$\norm{Y^n}_{S_G^\infty}+\norm{Z^n}_{BMO_G}\leq C,$$and
$$\E[|K_T^n|^p]+\E[|L_T^n|^p]\leq  C_p,\quad \forall p\geq 1.$$
\end{Lemma}

\begin{proof}
First we consider the estimate for $Y^n$.
The proof is very similar to that of Proposition \ref{Propose_Y_bound}.

For some $r>0~$, applying It\^o's formula to $e^{rt}|Y_t-N_0|^2$,~ we have for each $t\in[0,T]$,
\begin{eqnarray*}
&&e^{rt}|Y_t^n-N_0|^2+r\int_t^T e^{rs}|Y_s^n-N_0|^2ds+\int_t^T e^{rs}|Z_s^n|^2d\Bq_s\\
&=&e^{rT}|\xi-N_0|^2+\int_t^T 2e^{rs}(Y_s^n-N_0)f(s,Y_s^n,Z_s^n)d\Bq_s-\int_t^T 2e^{rs}(Y_s^n-N_0)Z_s^ndB_s\\
&&+\int_t^T 2e^{rs}(Y_s^n-N_0)d(L_s^n-K_s^n),~~\text{q.s.}
\end{eqnarray*}
Noting that
\begin{eqnarray*}
&&\int_t^T e^{rs}(Y_s^n-N_0)dL_s^n\\
&=&n\int_t^T e^{rs}(Y_s^n-N_0)(Y_s^n-S_s)^-ds
\leq n\int_t^T e^{rs}(Y_s^n-S_s)(Y_s^n-S_s)^-ds\leq 0,~~\text{q.s.},
\end{eqnarray*}
we have
\begin{eqnarray*}
&&e^{rt}|Y_t^n-N_0|^2+r\int_t^T e^{rs}|Y_s^n-N_0|^2ds\\
&\leq&e^{rT}|\xi-N_0|^2+\int_t^T 2e^{rs}(Y_s^n-N_0)f(s,Y_s^n,Z_s^n)d\Bq_s-\int_t^T 2e^{rs}(Y_s^n-N_0)Z_s^ndB_s\\
&&-\int_t^T 2e^{rs}(Y_s^n-N_0)dK_s^n,~~\text{q.s.}
\end{eqnarray*}
Similar to the proof of Proposition \ref{Propose_Y_bound}, we have for each $s\in[0,T]$,
$$f(s,Y_s^n,Z_s^n)=f(s,0,0)+m_s^{n,\varepsilon}+a_s^{n,\varepsilon} Y_s^n+b_s^{n,\varepsilon} Z_s^n,$$
where
\begin{eqnarray*}
|a_s^{n,\varepsilon}|&\leq& L_y,~~|b_s^{n,\varepsilon}|\leq L_z(1+|Z_s^n|),\\
|m_s^{n,\varepsilon}|&\leq& 2\varepsilon(L_y+L_z(1+ 2\varepsilon)).
\end{eqnarray*}
So we have
\begin{eqnarray*}
&&e^{rt}|Y_t^n-N_0|^2+r\int_t^T e^{rs}|Y_s^n-N_0|^2ds+\int_t^T 2e^{rs}(Y_s^n-N_0)^+dK_s^n\\
&\leq&e^{rt}|Y_t^n-N_0|^2+r\int_t^T e^{rs}|Y_s^n-N_0|^2ds+\int_t^T 2e^{rs}(Y_s^n-N_0)dK_s^n\\
&\leq&e^{rT}|\xi-N_0|^2+\int_t^T 2e^{rs}(Y_s^n-N_0)(f(s,0,0)+m_s^{n,\varepsilon}+a_s^{n,\varepsilon} Y_s^n+b_s^{n,\varepsilon} Z_s^n)d\Bq_s\\
&&-\int_t^T 2e^{rs}(Y_s^n-N_0)Z_s^ndB_s\\
&\leq&e^{rT}|\xi-N_0|^2+(1+2L_y)\int_t^T e^{rs}|Y_s^n-N_0|^2d\Bq_s\\
&&+\int_t^T e^{rs}\big(f(s,0,0)+|m_s^{n,\varepsilon}|+N_0L_y\big)^2d\Bq_s-\int_t^T 2e^{rs}(Y_s^n-N_0)Z_s^nd\tilde{B}_s,~~\text{q.s.},
\end{eqnarray*}
where $d\tilde{B}_s^{n,\varepsilon}=dB_s-b_s^{n,\varepsilon} d\Bq_s$.
In view of \cite[Lemma 3.6]{Hu2018}, we know that $b^{n,\varepsilon}\in BMO_G$. Thus we can define a new $G$-expectation $\Et^{n,\varepsilon}[\cdot]$ by $\Exp(b_s^{n,\varepsilon})$, such that $\tilde{B}^{n,\varepsilon}$ is a $G$-Brownian motion under $\Et^{n,\varepsilon}[\cdot]$.

In view of  Hu et al. \cite[Lemma 3.4]{HuM2014_1} and  Lemma \ref{lemma_Gisr_K}, we know that the process
 $$\int_0^\cdot 2e^{rs}(Y_s^n-N_0)^+dK_s^n$$
 is a decreasing $G$-martingale under both $\E[\cdot]$ and $\Et^{n,\varepsilon}[\cdot]$.
Setting $r>\siup^2(1+2L_y)$ and taking conditional expectations in the last inequality, we have for each $t\in[0,T]$,
$$e^{rt}|Y_t^n-N_0|^2\leq \Et_t^{n,\varepsilon}[e^{rT}|\xi-N_0|^2]+\Et^{n,\varepsilon}_t\Big[\int_t^T e^{rs}\big(f(s,0,0)+|m_s^{n,\varepsilon}|+N_0L_y\big)^2d\Bq_s\Big],~~\text{q.s.}$$
Then
\begin{eqnarray*}
e^{rt}|Y_t^n-N_0|^2&\leq& \Et^{n,\varepsilon}_t[e^{rT}|\xi-N_0|^2]+\Et^{n,\varepsilon}_t\Big[\int_t^T e^{rs}\big(f(s,0,0)+|m_s^{n,\varepsilon}+N_0L_y\big)^2d\Bq_s\Big]\\
&\leq& 2e^{rT}(\norm{\xi}_{L_G^\infty}^2+N_0^2)+2e^{rT}\siup^2\Bigg\{\norm{\int_0^T |f(s,0,0)|^2ds }_{L_G^\infty}\\
&&+\big(2\varepsilon(L_y+L_z+ 2L_z\varepsilon)+N_0L_y\big)^2T\Bigg\},~~\text{q.s.}
\end{eqnarray*}
Setting $\varepsilon\to 0$, we have
$$e^{rt}|Y_t^n-N_0|^2\leq 2e^{rT}(\norm{\xi}_{L_G^\infty}^2+N_0^2)+2e^{rT}\siup^2\Bigg\{\norm{\int_0^T |f(s,0,0)|^2ds }_{L_G^\infty}
+N_0^2L_y^2T\Bigg\},~~\text{q.s.}$$
So we know there exists a constant $C^\prime$ independent of $n$ such that $\norm{Y^n}_{S_G^\infty}\leq C^\prime$.

Then by Proposition \ref{Propose_Z_bound}, we know that there exist two constants $C^{\prime\prime}$ and $C_p^\prime$ which are independent of $n$, such that
$$\norm{Z^n}_{BMO_G}\leq C^{\prime\prime},$$and
$$\E[|L_T^n-K_T^n|^p]\leq  C_p^\prime,\quad \forall p\geq 1.$$
We have $$\norm{Y^n}_{S_G^\infty}+\norm{Z^n}_{BMO_G}\leq C,$$
with $C=2(C^\prime+C^{\prime\prime})$, and
$$\E[|L_T^n|^p]+\E[|K_T^n|^p]\leq2\E[|L_T^n-K_T^n|^p]\leq  C_p,$$
with $C_p=2C_p^\prime$.
\end{proof}

The following lemma plays a key role in the proof of the convergence of $\{Y^n\}$.  It gives the convergence of $(Y^n-S)^-$ in $S_G^\alpha(0,T)$.

\begin{Lemma}
\label{lemma_Y^n-S convergence}
For each $\alpha>1$, we have
$$\lim_{n\to\infty}\E\big[\sup_{t\in[0,T]}|(Y^n_t-S_t)^-|^\alpha\big]=0.$$
\end{Lemma}


\begin{proof} The lemma has been proved by
Li, Peng and Soumana Hima \cite[Lemma 4.3]{Li2017_1}  when the generator $(f,g)$ is uniformly Lipschitz continuous.  Their arguments can be adapted to our general case.

First, we sketch the main ideas.
Under our $z$-quadratic generator, we will still use the method of linearization. By the $G$-Girsanov theorem, we can we rewrite the $G$-BSDE (\ref{Penal_GBSDE1}) so that the generator is independent of $z$ under a new $G$-expectation $\Et[\cdot]$. Similarly as in  \cite[Lemma 4.3]{Li2017_1},  the following holds true:
$$\lim_{n\to\infty}\Et\big[\sup_{t\in[0,T]}|(Y^n_t-S_t)^-|^\alpha\big]=0,\quad \forall \alpha>1.$$
Then from Lemmas~\ref{A_p} and \ref{reverse holder}, we see that $\Et[\cdot]$ can be replaced with $\E[\cdot]$ in the last limit, which completes the proof.

Now we begin our proof. Similar to the proof of Proposition \ref{Propose_Y_bound} and Lemma \ref{lemma_Y^n_bound}, we first rewrite the $G$-BSDE (\ref{Penal_GBSDE1}) by linearization into the form:
$$f(s,Y_s^n,Z_s^n)=f(s,0,0)+m_s^{n,\varepsilon}+a_s^{n,\varepsilon} Y_s^n+b_s^{n,\varepsilon} Z_s^n,~~s\in[0,T],$$
with
\begin{eqnarray*}
|a_s^{n,\varepsilon}|\leq L_y,\quad |b_s^{n,\varepsilon}|\leq L_z(1+|Z_s^n|),\quad \hbox{ \rm and } \quad
|m_s^{n,\varepsilon}|\leq 2\varepsilon(L_y+L_z(1+ 2\varepsilon)),~~s\in[0,T].
\end{eqnarray*}
So the $G$-BSDE (\ref{Penal_GBSDE1}) reads
$$Y_t^n=\xi+\int_t^Tf(s,0,0)+m_s^{n,\varepsilon}+a_s^{n,\varepsilon} Y_s^nd\Bq_s+n\int_t^T(Y_s^n-S_s)^-ds-\int_t^TZ_s^nd\tilde B^{n,\varepsilon}_s-\int_t^TdK_s^n,~~\text{q.s.},$$
where $d\tilde{B}^{n,\varepsilon}_s=dB_s-b_s^{n,\varepsilon} d\Bq_s$.
In view of \cite[Lemma 3.6]{Hu2018}, we know that $b^{n,\varepsilon}\in BMO_G$.
Thus we can define a new $G$-expectation $\Et^{n,\varepsilon}[\cdot]$ by $\Exp(b^{n,\varepsilon})$, such that $\tilde{B}^{n,\varepsilon}$ is a $G$-Brownian motion under $\Et^{n,\varepsilon}[\cdot]$.

We now prove
\begin{equation}\label{penalty lim}
\lim_{n\to\infty}\Et^{n,\varepsilon}\left[\sup_{t\in[0,T]}|(Y^n_t-S_t)^-|^\alpha\right]=0.
\end{equation}
Set
\begin{eqnarray*}
y_t^n&=&\xi+\int_t^Tf(s,0,0)+m_s^{n,\varepsilon}+a_s^{n,\varepsilon} Y_s^nd\Bq_s+n\int_t^T(S_s-y_s^n)ds\\
&&-\int_t^Tz_s^nd\tilde B^{n,\varepsilon}_s-\int_t^Tdk_s^n,~~\text{q.s.},\quad t\in[0,T].
\end{eqnarray*}
Then we have for each $t\in[0,T]$,
$$y_t^n=e^{nt}\Et_t^{n,\varepsilon}\left[e^{-nT}\xi+\int_t^Tne^{-ns}ds+\int_t^Te^{-ns}(f(s,0,0)+m_s^{n,\varepsilon}
+a_s^{n,\varepsilon} Y_s^n)d\Bq_s\right],~~\text{q.s.}$$
In view of \cite[Theorem 3.6]{HuM2014_2}, we have for each $t\in[0,T]$,
$$Y_t^n-S_t\geq y_t^n-S_t=\Et_t^{n,\varepsilon}\left[\tilde S_t^n +\int_t^Te^{n(t-s)}(f(s,0,0)+m_s^{n,\varepsilon}
+a_s^{n,\varepsilon} Y_s^n)d\Bq_s\right],~~\text{q.s.},$$
where $$\tilde S_t^n:=e^{n(t-T)}(\xi-S_t)+\int_t^Tne^{n(t-s)}(S_s-S_t)ds, \quad t\in [0,T].$$
It follows that
$$(Y_t^n-S_t)^-\leq (y_t^n-S_t)^-\leq\Et_t^{n,\varepsilon}\left[|\tilde S_t^n| +\left|\int_t^Te^{n(t-s)}(f(s,0,0)+m_s^{n,\varepsilon}
+a_s^{n,\varepsilon} Y_s^n)d\Bq_s\right|\right],~~\text{q.s.}$$
We have for any $\alpha>1$,
\begin{eqnarray*}
&&\Et^{n,\varepsilon}\left[\sup_{t\in[0,T]}\left|\int_t^Te^{{n(t-s)}}(f(s,0,0)+m_s^{n,\varepsilon}
+a_s^{n,\varepsilon} Y_s^n)d\Bq_s\right|^\alpha\right]\\
&\leq&\siup^{2\alpha}\Et^{n,\varepsilon}\bigg[\left(\int_0^T(f(s,0,0)+m_s^{n,\varepsilon}
+a_s^{n,\varepsilon} Y_s^n)^2ds\right)^{\alpha\over 2}\sup_{t\in[0,T]}\left(\int_t^Te^{2n(t-s)}ds\right)^{\alpha\over 2}\bigg]\\
&\leq&\left({1-e^{-2nT}\over n}\right)^{\alpha\over 2}\siup^{2\alpha}
\Bigg\{\norm{\int_0^T|f(s,0,0)|^2ds}_{L_G^\infty}\\
&&\quad\quad\quad+T\big[L_y\norm{Y^n}_{S_G^\infty}+2\varepsilon (L_y+L_z+ 2L_z\varepsilon)\big]^2\Bigg\}^{\alpha\over 2}.
\end{eqnarray*}
In view of Lemma \ref{lemma_Y^n_bound}, we have
\begin{equation}
\label{eq_lim_f}
\lim_{n\to\infty}\Et^{n,\varepsilon}\left[\sup_{t\in[0,T]}\left|\int_t^Te^{n(t-s)}(f(s,0,0)+m_s^{n,\varepsilon}
+a_s^{n,\varepsilon} Y_s^n)d\Bq_s\right|^\alpha\right]=0.
\end{equation}
For $\epsilon>0$, it is straightforward to show for each $t\in[0,T]$,
\begin{eqnarray*}
|\tilde S_t^n|&=&\left|e^{n(t-T)}(\xi-S_t)+\int_{t+\epsilon}^Tne^{n(t-s)}(S_s-S_t)\, ds
+\int_t^{t+\epsilon}ne^{n(t-s)}(S_s-S_t)ds\right|\\
&\leq& e^{n(t-T)}|\xi-S_t|+e^{-n\epsilon}\sup_{s\in[t+\epsilon,T]}|S_t-S_s|
+\sup_{s\in[t,t+\epsilon]}|S_t-S_s|,~~\text{q.s.}
\end{eqnarray*}
For $\delta\in (0,T)$, we have
\begin{eqnarray}
&&\sup_{t\in[0,T-\delta]}|\tilde S_t^n |\\ &\leq&e^{-n\delta}\sup_{t\in[0,T-\delta]}|\xi-S_t|+e^{-n\epsilon}\sup_{t\in[0,T-\delta]}
\sup_{s\in[t+\epsilon,T]}|S_t-S_s|\nonumber\\[3mm]
&&\quad\quad+\sup_{t\in[0,T-\delta]}\sup_{s\in[t,t+\epsilon]}|S_t-S_s|\\[3mm]
&\leq&e^{-n\delta}(\sup_{t\in[0,T]}|S_t|+|\xi|)+2e^{-n\epsilon}\sup_{t\in[0,T]}|S_t|
+\sup_{t\in[0,T]}\sup_{s\in[t,t+\epsilon]}|S_t-S_s|,~~\text{q.s.}
\label{ineq_St_n}
\end{eqnarray}
Define the function
$$\phi(x):=\left(1+{1\over x^2}\log{2x-1\over2(x-1)}\right)^{1\over2}-1, \quad x\in (1,\infty).$$
 In view of Lemma \ref{lemma_Y^n_bound}, we can choose $p>1$ independent of $n$ and $\varepsilon$, such that
$$\norm{b^{n,\varepsilon}}_{BMO_G}\leq L_z(1+\norm{Z^n}_{BMO_G})<\phi(p).$$
Set $q={p\over p-1}$.
Then in view of
Lemma \ref{reverse holder}, we have for each $\alpha>1$ and $X\in L_G^q(\Omega_T)$,
\begin{equation}
\label{eq_change_measure1}
\Et^{n,\varepsilon}[X]
=\E\big[\Exp(b^{n,\varepsilon})_TX\big]\leq
\E\big[\Exp(b^{n,\varepsilon})_T^p\big]^{1\over p}\E[|X|^q]^{1\over q}
\leq C_p\E[|X|^q]^{1\over q},
\end{equation}
where $C_p$ depends only on $p$.

In view of Assumption {\bf (H4)} on $S$, we know
$$\E[\sup_{t\in[0,T]}|S_t|^\alpha]<+\infty,\quad \forall \alpha>1.$$
So we have for all $\alpha>1$,
$$\Et^{n,\varepsilon}\Big[\sup_{t\in[0,T]}|S_t|^\alpha\Big]\leq
C_p\E\Big[\sup_{t\in[0,T]}|S_t|^{\alpha q}\Big]^{1\over q},$$
and
$$\Et^{n,\varepsilon}\Big[\sup_{t\in[0,T]}\sup_{s\in[t,t+\epsilon]}|S_t-S_s|^\alpha\Big]\leq
C_p\E\Big[\sup_{t\in[0,T]}\sup_{s\in[t,t+\epsilon]}|S_t-S_s|^{\alpha q}\Big]^{1\over q}.$$
From (\ref{ineq_St_n}), we know
\begin{equation}
\label{eq_lim_St_n}
\limsup_{n\to\infty}\Et^{n,\varepsilon}\Big[\sup_{t\in[0,T-\delta]}|\tilde S_t^n|^\alpha\Big]\leq
C_\alpha C_p\E\Big[\sup_{t\in[0,T]}\sup_{s\in[t,t+\epsilon]}|S_t-S_s|^{\alpha q}\Big]^{1\over q}.
\end{equation}
Then,  in view of (\ref{eq_lim_f}), (\ref{eq_lim_St_n}),  and Remark \ref{remark_doob_type_ineq}, we have
\begin{eqnarray*}
&&\limsup_{n\to\infty}\Et^{n,\varepsilon}\Big[\sup_{t\in[0,T-\delta]}|(Y_t^n-S_t)^-|^\alpha\Big]\\
&\leq & \limsup_{n\to\infty}\Et^{n,\varepsilon}\Big[\sup_{t\in[0,T-\delta]}
\Et^{n,\varepsilon}_t\Big[|\tilde S_t^n| +\Big|\int_t^Te^{n(t-s)}(f(s,0,0)+m_s^{n,\varepsilon}
+a_s^{n,\varepsilon} Y_s^n)d\Bq_s\Big|\Big]^\alpha\Big]\\
&\leq & C\limsup_{n\to\infty}\Et^{n,\varepsilon}\Big[\sup_{t\in[0,T-\delta]}
\Et^{n,\varepsilon}_t\Big[\sup_{u\in[0,T-\delta]}|\tilde S_u^n|^\alpha \Big]\Big]\\
&&+C\limsup_{n\to\infty}\Et^{n,\varepsilon}\Big[\sup_{t\in[0,T-\delta]}
\Et^{n,\varepsilon}_t\Big[\sup_{u\in[0,T-\delta]}\Big|\int_u^Te^{n(t-s)}(f(s,0,0)+m_s^{n,\varepsilon}
+a_s^{n,\varepsilon} Y_s^n)d\Bq_s\Big|^\alpha \Big]\Big]\\
&=& C\limsup_{n\to\infty}\Et^{n,\varepsilon}\Big[\sup_{t\in[0,T-\delta]}
\Et^{n,\varepsilon}_t\Big[\sup_{u\in[0,T-\delta]}|\tilde S_u^n|^\alpha \Big]\Big]\\
&\leq& C^\prime\limsup_{n\to\infty}\bigg\{\Et^{n,\varepsilon}\Big[\sup_{u\in[0,T-\delta]}|\tilde S_u^n|^{2\alpha} \Big]+\Et^{n,\varepsilon}\Big[\sup_{u\in[0,T-\delta]}|\tilde S_u^n|^{2\alpha} \Big]^{1\over 2}\bigg\}\\
&\leq& C^{\prime\prime}\bigg\{\E\Big[\sup_{t\in[0,T]}\sup_{s\in[t,t+\epsilon]}|S_t-S_s|^{2\alpha q}\Big]^{1\over q}+\E\Big[\sup_{t\in[0,T]}\sup_{s\in[t,t+\epsilon]}|S_t-S_s|^{2\alpha q}\Big]^{1\over 2q} \bigg\},
\end{eqnarray*}
where $C^{\prime\prime}$ is independent of $n$, $\delta$ and $\epsilon$.
Therefore,  in view of Lemma \ref{lemma_Y_continue}, setting $\epsilon\to 0$,  we have
$$\limsup_{n\to\infty}\Et^{n,\varepsilon}\Big[\sup_{t\in[0,T-\delta]}|(Y_t^n-S_t)^-|^\alpha\Big]=0.$$
In view of Theorem \ref{Thm_compare}, we get $Y_t^n\geq Y_t^1$ and then obtain
\begin{eqnarray*}
&&\limsup_{n\to\infty}\Et^{n,\varepsilon}\Big[\sup_{t\in[0,T]}|(Y_t^n-S_t)^-|^\alpha\Big]\\
&\leq& \limsup_{n\to\infty}\Et^{n,\varepsilon}\Big[\sup_{t\in[0,T-\delta]}|(Y_t^n-S_t)^-|^\alpha\Big]+
\limsup_{n\to\infty}\Et^{n,\varepsilon}\Big[\sup_{t\in[T-\delta,T]}|(Y_t^n-S_t)^-|^\alpha\Big]\\
&\leq&\limsup_{n\to\infty}\Et^{n,\varepsilon}\Big[\sup_{t\in[T-\delta,T]}|(Y_t^1-S_t)^-|^\alpha\Big].
\end{eqnarray*}
By Lemma \ref{lemma_Y_continue} again and noting that $(Y_T^1-S_T)^-=0$, we obtain
$$\lim_{\delta\to 0}\E\Big[\sup_{t\in[T-\delta,T]}|(Y_t^1-S_t)^-|^\alpha\Big]=0,~~\forall \alpha>1.$$
Finally, with (\ref{eq_change_measure1}), we derive that
$$\limsup_{n\to\infty}\Et^{n,\varepsilon}\Big[\sup_{t\in[T-\delta,T]}|(Y_t^1-S_t)^-|^\alpha\Big]\leq C_p\E\Big[\sup_{t\in[T-\delta,T]}|(Y_t^1-S_t)^-|^{q\alpha}\Big]^{1\over q}.$$
Let $\delta\to 0$ and we know
$$\limsup_{n\to\infty}\Et^{n,\varepsilon}\Big[\sup_{t\in[T-\delta,T]}|(Y_t^1-S_t)^-|^\alpha\Big]=0.$$
Therefore,  we have~\eqref{penalty lim}.

Next we want to change the $G$-expectation in the last equality. Actually, in view of Lemma \ref{A_p} and Remark \ref{remark_Ap_Rp}, there exists $r>1$ which is independent of $n$ and $\varepsilon$, such that
$$\E\left[\left\{\Exp(b^{n,\varepsilon})_T\right\}^{1\over 1-r}\right]\leq C_r,$$
for some positive constant $C_r$ which depends only on $r$. Thus, for each $\alpha>1$, we have
\begin{eqnarray*}
\E\Big[\sup_{t\in[0,T]}|(Y_t^n-S_t)^-|^\alpha\Big]&=&
\E\Big[\Exp(b^{n,\varepsilon})_T^{1\over r}\Exp(b^{n,\varepsilon})_T^{-{1\over r}}\sup_{t\in[0,T]}|(Y_t^n-S_t)^-|^\alpha\Big]\\
&\leq& \E\Big[\Exp(b^{n,\varepsilon})_T\sup_{t\in[0,T]}|(Y_t^n-S_t)^-|^{\alpha r}\Big]^{1\over r}
\E\Big[\Big\{\Exp(b^{n,\varepsilon})_T\Big\}^{1\over 1-r}\Big]^{r-1\over r}\\
&\leq& C_r^{r-1\over r}\Et^{n,\varepsilon}\Big[\sup_{t\in[0,T]}|(Y_t^n-S_t)^-|^{\alpha r}\Big]^{1\over r}.
\end{eqnarray*}
So $$\limsup_{n\to\infty}\E\Big[\sup_{t\in[0,T]}|(Y_t^n-S_t)^-|^\alpha\Big]=0.$$
\end{proof}

Now we show the convergence of the sequence $\{Y^n\}_{n=1}^\infty$.
\begin{Lemma}
\label{lemma_Y^n convergence}
The sequence $\{Y^n\}_{n=1}^\infty$ is a Cauchy sequence in $S_G^\alpha(0,T)$ for any $\alpha\geq 2$.
\end{Lemma}

\begin{proof}
For $m,n\in\bb N$ and each $t\in[0,T]$, set
$$\hat Y^{n,m}_t=Y^n_t-Y^m_t,\quad\hat Z^{n,m}_t=Z^n_t-Z^m_t,\quad \hat K^{n,m}_t=K^n_t-K^m_t,\quad\hat L^{n,m}_t=L^n_t-L^m_t.$$
We  use the method of linearization. Similar to the proof of Proposition \ref{Propose_Y_bound} and Lemma \ref{lemma_Y^n_bound}, $\forall \varepsilon>0$, we write for each $s\in[0,T]$,
$$f(s,Y_s^n,Z_s^n)-f(s,Y_s^m,Z_s^m)=m_s^{n,m,\varepsilon}+a_s^{n,m,\varepsilon}
\hat Y_s^{n,m}+b_s^{n,m,\varepsilon}\hat Z_s^{n,m}$$
with
\begin{eqnarray*}
| a_s^{n,m,\varepsilon}|&\leq& L_y,~~| b_s^{n,m,\varepsilon}|\leq L_z(1+|Z_s^n|+|Z_s^m|),\\
| m_s^{n,m,\varepsilon}|&\leq& 2\varepsilon(L_y+L_z(1+ 2\varepsilon+2|Z_s^n|)).
\end{eqnarray*}
So we have for each $t\in[0,T]$,
\begin{eqnarray*}
\hat Y_t^{n,m}&=&\int_t^T m_s^{n,m,\varepsilon}+a_s^{n,m,\varepsilon}
\hat Y_s^{n,m}+b_s^{n,m,\varepsilon}\hat Z_s^{n,m} d\Bq_s\\
&&-\int_t^T \hat Z_s^{n,m}dB_s
-\int_t^Td\hat K^{n,m}_s+\int_t^Td\hat L^{n,m}_s\\[3mm]
&=&\int_t^T m_s^{n,m,\varepsilon}+a_s^{n,m,\varepsilon}
\hat Y_s^{n,m} d\Bq_s-\int_t^T \hat Z_s^{n,m}d\tilde B_s^{n,m,\varepsilon}
-\int_t^Td\hat K^{n,m}_s+\int_t^Td\hat L^{n,m}_s,~~\text{q.s.},
\end{eqnarray*}
where $d\tilde B_s^{n,m,\varepsilon}=dB_s-b_s^{n,m,\varepsilon}\hat Z_s^{n,m} d\Bq_s.$
In view of \cite[Lemma 3.6]{Hu2018}, we know that $b^{n,m,\varepsilon}\in BMO_G$
and we  define a new $G$-expectation $\Et^{n,m,\varepsilon}[\cdot]$ by $\Exp(b^{n,m,\varepsilon})$, such that $\tilde{B}^{n,m,\varepsilon}$ is a $G$-Brownian motion under $\Et^{n,m,\varepsilon}[\cdot]$.

For all $\alpha\geq 2$, by applying It\^o's formula to $|\hat Y_t^{n,m}|^\alpha e^{rt}$, we get for each $t\in[0,T]$,
\begin{eqnarray*}
&&|\hat Y_t^{n,m}|^\alpha e^{rt}+\int_t^T re^{rs}|\hat Y_s^{n,m}|^\alpha ds+{1\over2} \alpha(\alpha-1)\int_t^T e^{rs}|\hat Y_s^{n,m}|^{\alpha-2}|\hat Z_s^{n,m}|^2d\Bq_s\\
&=&\int_t^T \alpha e^{rs}|\hat Y_s^{n,m}|^{\alpha-2}\hat Y_s^{n,m}m_s^{n,m,\varepsilon}d\Bq_s+\int_t^T \alpha e^{rs}|\hat Y_s^{n,m}|^{\alpha}a_s^{n,m,\varepsilon} d\Bq_s\\
&&- \int_t^T \alpha e^{rs}|\hat Y_s^{n,m}|^{\alpha-2}\hat Y_s^{n,m}\hat Z_s^{n,m}d\tilde B_s^{n,m,\varepsilon}-\int_t^T \alpha e^{rs}|\hat Y_s^{n,m}|^{\alpha-2}\hat Y_s^{n,m}d\hat K^{n,m}_s\\
&&+\int_t^T \alpha e^{rs}|\hat Y_s^{n,m}|^{\alpha-2}\hat Y_s^{n,m}d\hat L^{n,m}_s,~~\text{q.s.}
\end{eqnarray*}
Let $r>L_y\alpha\siup^2$. Noting that $| a_s^{n,m,\varepsilon}|\leq L_y$, we get
\begin{eqnarray*}
|\hat Y_t^{n,m}|^\alpha e^{rt}
&\leq&\int_t^T \alpha e^{rs}|\hat Y_s^{n,m}|^{\alpha-2}\hat Y_s^{n,m}m_s^{n,m,\varepsilon}d\Bq_s\\
&&- \int_t^T \alpha e^{rs}|\hat Y_s^{n,m}|^{\alpha-2}\hat Y_s^{n,m}\hat Z_s^{n,m}
d\tilde B_s^{n,m,\varepsilon}\\
&&-\int_t^T \alpha e^{rs}|\hat Y_s^{n,m}|^{\alpha-2}\hat Y_s^{n,m}d\hat K^{n,m}_s
+\int_t^T \alpha e^{rs}|\hat Y_s^{n,m}|^{\alpha-2}\hat Y_s^{n,m}d\hat L^{n,m}_s,~~\text{q.s.}
\end{eqnarray*}
It is easy to check that
\begin{eqnarray*}
&&\int_t^T \alpha e^{rs}|\hat Y_s^{n,m}|^{\alpha-2}\hat Y_s^{n,m}d\hat L^{n,m}_s\\
&=&-\int_t^T \alpha e^{rs}|\hat Y_s^{n,m}|^{\alpha-2} (Y_s^{n}-S_s)d L^{m}_s
-\int_t^T \alpha e^{rs}|\hat Y_s^{n,m}|^{\alpha-2} (Y_s^{m}-S_s)dL^{n}_s\\
&&+\int_t^T \alpha e^{rs}|\hat Y_s^{n,m}|^{\alpha-2} (Y_s^{n}-S_s)d L^{n}_s
+\int_t^T \alpha e^{rs}|\hat Y_s^{n,m}|^{\alpha-2} (Y_s^{m}-S_s)dL^{m}_s\\
&\leq&-\int_t^T \alpha e^{rs}|\hat Y_s^{n,m}|^{\alpha-2} (Y_s^{n}-S_s)d L^{m}_s
-\int_t^T \alpha e^{rs}|\hat Y_s^{n,m}|^{\alpha-2} (Y_s^{m}-S_s)dL^{n}_s,~~\text{q.s.}\\
\end{eqnarray*}
Noting that
\begin{eqnarray*}
\int_t^T \alpha e^{rs}|\hat Y_s^{n,m}|^{\alpha-2}\hat Y_s^{n,m}d\hat K^{n,m}_s&\geq&
\int_t^T \alpha e^{rs}|\hat Y_s^{n,m}|^{\alpha-2}(\hat Y_s^{n,m})^+d K^m_s\\
&& \quad\quad
+\int_t^T \alpha e^{rs}|\hat Y_s^{n,m}|^{\alpha-2}(\hat Y_s^{n,m})^-d K^n_s,~~\text{q.s.},
\end{eqnarray*}
we have
\begin{eqnarray*}
|\hat Y_t^{n,m}|^\alpha e^{rt}+M_T-M_t
&\leq&\int_t^T \alpha e^{rs}|\hat Y_s^{n,m}|^{\alpha-2}\hat Y_s^{n,m}m_s^{n,m,\varepsilon}d\Bq_s\\
&&-\int_t^T \alpha e^{rs}|\hat Y_s^{n,m}|^{\alpha-2} (Y_s^{n}-S_s)d L^{m}_s\\
&&-\int_t^T \alpha e^{rs}|\hat Y_s^{n,m}|^{\alpha-2} (Y_s^{m}-S_s)dL^{n}_s,~~\text{q.s.},
\end{eqnarray*}
where
$$M_t:=\int_0^t \alpha e^{rs}|\hat Y_s^{n,m}|^{\alpha-2}[(\hat Y_s^{n,m})^+d K^m_s
+(\hat Y_s^{n,m})^-d K^n_s+\hat Y_s^{n,m}\hat Z_s^{n,m}d\tilde B_s^{n,m,\varepsilon}].
$$

In view of \cite[Lemma 3.3]{HuM2014_1} and \cite[Lemma 3.4]{Hu2018}, we conclude that $M$ is a $G$-martingale under $\Et^{n,m,\varepsilon}[\cdot]$.
Thus we obtain
\begin{eqnarray}
\nonumber&&|\hat Y_t^{n,m}|^\alpha e^{rt}
-\Et_t^{n,m,\varepsilon}\left[\int_t^T \alpha e^{rs}|\hat Y_s^{n,m}|^{\alpha-2}\hat Y_s^{n,m}m_s^{n,m,\varepsilon}d\Bq_s\right]\\
&\leq&\Et_t^{n,m,\varepsilon}\biggl[-\int_t^T \alpha e^{rs}|\hat Y_s^{n,m}|^{\alpha-2} (Y_s^{n}-S_s)\, d L^{m}_s\nonumber\\
&&\quad\quad\quad-\int_t^T \alpha e^{rs}|\hat Y_s^{n,m}|^{\alpha-2} (Y_s^{m}-S_s)dL^{n}_s\biggr],~~\text{q.s.}
\label{ineq_Y_hat_m_n1}
\end{eqnarray}
Noting the following estimate
\begin{eqnarray*}
&&\Et_t^{n,m,\varepsilon}\left[-\int_t^T \alpha e^{rs}|\hat Y_s^{n,m}|^{\alpha-2} (Y_s^{m}-S_s)dL^{n}_s\right]\\
&=& \Et_t^{n,m,\varepsilon}\left[-\int_t^T n\alpha e^{rs}|\hat Y_s^{n,m}|^{\alpha-2} (Y_s^{m}-S_s)(Y_s^{n}-S_s)^-ds\right]\\
&\leq& \alpha e^{rT}\Et_t^{n,m,\varepsilon}\left[\int_t^T n| (Y_s^{n}-S_s)-(Y_s^m-S_s)|^{\alpha-2} (Y_s^{m}-S_s)^-(Y_s^{n}-S_s)^-ds\right]\\
&\leq& C\Et_t^{n,m,\varepsilon}\left[\int_t^T n| (Y_s^{n}-S_s)^-|^{\alpha-1} (Y_s^{m}-S_s)^-ds\right]\\
&&+C\Et_t^{n,m,\varepsilon}\left[\int_t^T n| (Y_s^{m}-S_s)^-|^{\alpha-1} (Y_s^{n}-S_s)^-ds\right],~~\text{q.s.},
\end{eqnarray*}
where $C$ is independent of $n$, $m$ and $\varepsilon$, we deduce from (\ref{ineq_Y_hat_m_n1}) that
\begin{eqnarray}
\nonumber&&\Et^{n,m,\varepsilon}\left[\sup_{t\in[0,T]}|\hat Y_t^{n,m}|^\alpha\right]
-\Et^{n,m,\varepsilon}\left[\sup_{t\in[0,T]}\Et_t^{n,m,\varepsilon}\left[\int_0^T \alpha e^{rs}|\hat Y_s^{n,m}|^{\alpha-1}|m_s^{n,m,\varepsilon}|d\Bq_s\right]\right]\\\nonumber
&\leq&C\Et^{n,m,\varepsilon}\left[\sup_{t\in[0,T]}\Et_t^{n,m,\varepsilon}\left[\int_0^T (m+n)| (Y_s^{n}-S_s)^-|^{\alpha-1} (Y_s^{m}-S_s)^-ds\right]\right]\\
&&+C\Et^{n,m,\varepsilon}\left[\sup_{t\in[0,T]}\Et_t^{n,m,\varepsilon}\left[\int_0^T (m+n)| (Y_s^{m}-S_s)^-|^{\alpha-1} (Y_s^{n}-S_s)^-ds\right]\right].
\label{ineq_Y_hat_m_n2}
\end{eqnarray}
Recall that
$$\phi(x)=\left(1+{1\over x^2}\log{2x-1\over2(x-1)}\right)^{1\over2}-1, \quad x>1.$$
 In view of Lemma \ref{lemma_Y^n_bound}, we can choose $p>1$ independent of $n$, $m$ and $\varepsilon$, such that
$$\norm{b^{n,m,\varepsilon}}_{BMO_G}\leq L_z(1+\norm{Z^n}_{BMO_G}+\norm{Z^m}_{BMO_G})<\phi(p).$$
Set $q={p\over p-1}$.
Then in view of
Lemma \ref{reverse holder}, we have for each $\alpha>1$ and $X\in L_G^q(\Omega_T)$,
\begin{equation}
\label{eq_change_measure2}
\Et_t^{n,m,\varepsilon}[X]
=\E_t\Big[{\Exp(b^{n,m,\varepsilon})_T\over \Exp(b^{n,m,\varepsilon})_t}X\Big]\leq
\E_t\Big[\Big({\Exp(b^{n,m,\varepsilon})_T\over \Exp(b^{n,m,\varepsilon})_t}\Big)^p\Big]^{1\over p}\E_t[|X|^q]^{1\over q}
\leq C_p\E_t[|X|^q]^{1\over q},~~\text{q.s.},
\end{equation}
where $C_p$ depends only on $p$.

Then we have for some $\beta>1$,
\begin{eqnarray}
\nonumber&&\Et^{n,m,\varepsilon}\left[\left(\int_t^T n|(Y_s^{n}-S_s)^-|^{\alpha-1} (Y_s^{m}-S_s)^-ds\right)^\beta\right]\\
\nonumber&\leq&\Et^{n,m,\varepsilon}\left[\sup_{s\in[0,T]}
\left\{|(Y_s^{n}-S_s)^-|^{(\alpha-2)\beta}|(Y_s^{m}-S_s)^-|^\beta\right\}
\left(\int_t^T n(Y_s^{n}-S_s)^- ds\right)^\beta\right]\\
\nonumber&\leq&\Et^{n,m,\varepsilon}\left[\sup_{s\in[0,T]}|(Y_s^{n}-S_s)^-|^{4(\alpha-2)\beta}\right]^{1\over 4}
\Et^{n,m,\varepsilon}\left[\sup_{s\in[0,T]}|(Y_s^{m}-S_s)^-|^{4\beta}\right]^{1\over 4}\\
\nonumber&&\times\Et^{n,m,\varepsilon}\left[\left(\int_t^T n(Y_s^{n}-S_s)^- ds\right)^{2\beta}\right]^{1\over 2}\\
\nonumber&\leq&C_p^3\E\left[\sup_{s\in[0,T]}(Y_s^{n}-S_s)^-|^{4(\alpha-2)\beta q}\right]^{1\over 4q}
\E\left[\sup_{s\in[0,T]}(Y_s^{m}-S_s)^-|^{4\beta q}\right]^{1\over 4q}\\
&&\times \E\left[\left(\int_t^T n(Y_s^{n}-S_s)^- ds\right)^{2\beta q}\right]^{1\over 2q}\nonumber\\
&\leq& C_1\E\left[\sup_{s\in[0,T]}(Y_s^{m}-S_s)^-|^{4\beta q}\right]^{1\over 4q},
\label{ineq_Y_hat_m_n3}
\end{eqnarray}
and
\begin{eqnarray}
\nonumber&&\Et^{n,m,\varepsilon}\left[\left(\int_t^T m| (Y_s^{n}-S_s)^-|^{\alpha-1} (Y_s^{m}-S_s)^-ds\right)^\beta\right]\\
\nonumber&\leq& \Et^{n,m,\varepsilon}\left[\sup_{s\in[0,T]}|(Y_s^{n}-S_s)^-|^{(\alpha-1)\beta}
\left(\int_t^T m(Y_s^{m}-S_s)^- ds\right)^\beta\right]\\
\nonumber&\leq& \Et^{n,m,\varepsilon}\left[\sup_{s\in[0,T]}|(Y_s^{n}-S_s)^-|^{2(\alpha-1)\beta}\right]^{1\over 2}
\Et^{n,m,\varepsilon}\left[\left(\int_t^T m(Y_s^{m}-S_s)^- ds\right)^{2\beta}\right]^{1\over 2}\\
\nonumber&\leq& C_p^2 \E\left[\sup_{s\in[0,T]}|(Y_s^{n}-S_s)^-|^{2(\alpha-1)\beta q}\right]^{1\over 2q}
\E\left[\left(\int_t^T m(Y_s^{m}-S_s)^- ds\right)^{2\beta q}\right]^{1\over 2q}\\
&\leq& C_2\E\left[\sup_{s\in[0,T]}|(Y_s^{n}-S_s)^-|^{2(\alpha-1)\beta q}\right]^{1\over 2q}.
\label{ineq_Y_hat_m_n4}
\end{eqnarray}
Moreover, in view of the assumption on $S$ and Lemma \ref{lemma_Y^n_bound}, we know $C_1$ and $C_2$ are independent of $n$, $m$ and $\varepsilon$.
From Remark \ref{remark_doob_type_ineq}, there exists a constant $C^{\prime}$ independent of $n$, $m$ and $\varepsilon$, such that
\begin{eqnarray*}
\Et^{n,m,\varepsilon}\Big[\sup_{t\in[0,T]}\Et_t^{n,m,\varepsilon}[|X|]\Big]
\leq C^\prime ( \Et^{n,m,\varepsilon}[|X|^\beta]^{1\over \beta}+ \Et^{n,m,\varepsilon}[|X|^\beta]) .
\end{eqnarray*}
Then with (\ref{ineq_Y_hat_m_n2}), (\ref{ineq_Y_hat_m_n3}) and (\ref{ineq_Y_hat_m_n4}), we have
\begin{eqnarray}
\nonumber&&\Et^{n,m,\varepsilon}\Big[\sup_{t\in[0,T]}|\hat Y_t^{n,m}|^\alpha\Big]
-C^\prime\Et^{n,m,\varepsilon}\Big[\Big(\int_0^T \alpha e^{rs}|\hat Y_s^{n,m}|^{\alpha-1}|m_s^{n,m,\varepsilon}|d\Bq_s\Big)^\beta\Big]^{1\over \beta}\\
\nonumber&&-C^\prime\Et^{n,m,\varepsilon}\Big[\Big(\int_0^T \alpha e^{rs}|\hat Y_s^{n,m}|^{\alpha-1}|m_s^{n,m,\varepsilon}|d\Bq_s\Big)^\beta\Big]\\
\nonumber&\leq&CC^\prime\bigg\{\Et^{n,m,\varepsilon}\Big[\Big(\int_0^T (m+n)| (Y_s^{n}-S_s)^-|^{\alpha-1} (Y_s^{m}-S_s)^-ds\Big)^\beta\Big]^{1\over \beta}\\
\nonumber&&+\Et^{n,m,\varepsilon}\Big[\Big(\int_0^T (m+n)| (Y_s^{n}-S_s)^-|^{\alpha-1} (Y_s^{m}-S_s)^-ds\Big)^\beta\Big]\\
\nonumber&&+\Et^{n,m,\varepsilon}\Big[\Big(\int_0^T (m+n)| (Y_s^{m}-S_s)^-|^{\alpha-1} (Y_s^{n}-S_s)^-ds\Big)^\beta\Big]^{1\over \beta}\\
\nonumber&&+\Et^{n,m,\varepsilon}\Big[\Big(\int_0^T (m+n)| (Y_s^{m}-S_s)^-|^{\alpha-1} (Y_s^{n}-S_s)^-ds\Big)^\beta\Big]\bigg\}\\
\nonumber&\leq& \bar C\sum_{j=m,n}\bigg(\E\Big[\sup_{s\in[0,T]}(Y_s^{j}-S_s)^-|^{4\beta q}\Big]^{1\over 4q}+
\E\Big[\sup_{s\in[0,T]}|(Y_s^{j}-S_s)^-|^{2(\alpha-1)\beta q}\Big]^{1\over 2q}\\
&&+\E\Big[\sup_{s\in[0,T]}(Y_s^{j}-S_s)^-|^{4\beta q}\Big]^{1\over 4q\beta}+
\E\Big[\sup_{s\in[0,T]}|(Y_s^{j}-S_s)^-|^{2(\alpha-1)\beta q}\Big]^{1\over 2q\beta} \ \bigg),
\label{ineq_Y_hat_m_n5}
\end{eqnarray}
where $\bar C$ is independent of $n$, $m$ and $\varepsilon$.

In view of Lemma \ref{A_p} and Remark \ref{remark_Ap_Rp}, there is $r>1$ which is independent of $n$, $m$ and $\varepsilon$, such that
$$\E\Big[\Big\{\Exp(b^{n,m,\varepsilon})_T\Big\}^{1\over 1-r}\Big]\leq C_r,$$
where $C_r$ depends only on $r$.
Thus, for each $\alpha^\prime\geq 2$, we have
\begin{eqnarray*}
\E\Big[\sup_{t\in[0,T]}|\hat Y_t^{n,m}|^{\alpha^\prime}\Big]&=&
\E\Big[\Exp(b^{n,m,\varepsilon})_T^{1\over r}\Exp(b^{n,m,\varepsilon})_T^{-{1\over r}}\sup_{t\in[0,T]}|\hat Y_t^{n,m}|^{\alpha^\prime}\Big]\\
&\leq& \E\Big[\Exp(b^{n,m,\varepsilon})_T\sup_{t\in[0,T]}|\hat Y_t^{n,m}|^{\alpha^\prime r}\Big]^{1\over r}
\E\Big[\Big\{\Exp(b^{n,m,\varepsilon})_T\Big\}^{1\over 1-r}\Big]^{r-1\over r}\\
&\leq& C_r^{r-1\over r}\Et^{n,m,\varepsilon}\Big[\sup_{t\in[0,T]}|\hat Y_t^{n,m}|^{\alpha^\prime r}\Big]^{1\over r}.
\end{eqnarray*}
Setting $\alpha=\alpha^\prime r>2$ in (\ref{ineq_Y_hat_m_n5}), we have
\begin{eqnarray}
\nonumber&&\E\Big[\sup_{t\in[0,T]}|\hat Y_t^{n,m}|^{\alpha^\prime}\Big]^r
-C_r^{r-1}C^\prime\Et^{n,m,\varepsilon}\Big[\Big(\int_0^T \alpha e^{rs}|\hat Y_s^{n,m}|^{\alpha-1}|m_s^{n,m,\varepsilon}|d\Bq_s\Big)^\beta\Big]^{1\over \beta}\\
\nonumber&&-C_r^{r-1}C^\prime\Et^{n,m,\varepsilon}\Big[\Big(\int_0^T \alpha e^{rs}|\hat Y_s^{n,m}|^{\alpha-1}|m_s^{n,m,\varepsilon}|d\Bq_s\Big)^\beta\Big]\\
\nonumber&\leq& C_r^{r-1}\bar C\sum_{j=m,n}\bigg(\E\Big[\sup_{s\in[0,T]}(Y_s^{j}-S_s)^-|^{4\beta q}\Big]^{1\over 4q}+
\E\Big[\sup_{s\in[0,T]}|(Y_s^{j}-S_s)^-|^{2(\alpha-1)\beta q}\Big]^{1\over 2q}\\
&&+\E\Big[\sup_{s\in[0,T]}(Y_s^{j}-S_s)^-|^{4\beta q}\Big]^{1\over 4q\beta}+
\E\Big[\sup_{s\in[0,T]}|(Y_s^{j}-S_s)^-|^{2(\alpha-1)\beta q}\Big]^{1\over 2q\beta} \ \bigg).
\label{ineq_Y_hat_m_n6}
\end{eqnarray}
On the other hand,
\begin{eqnarray*}
&&\Et^{n,m,\varepsilon}\Big[\Big(\int_0^T \alpha e^{rs}|\hat Y_s^{n,m}|^{\alpha-1}|m_s^{n,m,\varepsilon}|d\Bq_s\Big)^\beta\Big]\\
&\leq& C_p \Et\Big[\Big(\int_0^T \alpha e^{rs}|\hat Y_s^{n,m}|^{\alpha-1}|m_s^{n,m,\varepsilon}|d\Bq_s\Big)^{\beta q}\Big]^{1\over q}\\
&\leq& 2\varepsilon C_p \alpha e^{rT}\norm{\hat Y^{n,m}}_{S_G^\infty}^{\beta(\alpha-1)}\E\Big[\Big(\int_0^T (L_y+L_z(1+ 2\varepsilon+2|Z_s^n|))d\Bq_s\Big)^{\beta q}\Big]^{1\over q}\\
&&\longrightarrow 0,~~as~~\varepsilon\to 0.
\end{eqnarray*}
Let $\varepsilon\to 0$ in (\ref{ineq_Y_hat_m_n6}). Then in view of Lemma \ref{lemma_Y^n-S convergence}, we conclude that $(Y^n)_{n=1}^\infty$ is a Cauchy sequence in $S_G^{\alpha\prime}(0,T).$
\end{proof}

\section{Existence and uniqueness result on reflected quadratic $G$-BSDEs}

Our main result in the paper is stated as follows.

\begin{Theorem}
\label{Thm_existence_uniqueness}
Let the triple $(\xi, f, S)$ satisfy \textbf{(H1)}-\textbf{(H5)}. Then,  the reflected $G$-BSDE (\ref{RGBSDE2}) has a unique solution (Y,Z,A) such that $(Y, Z)\in S_G^{\infty}(0,T)\times BMO_G$ and
$A\in \bigcap_{\alpha\geq 2} S_G^\alpha(0,T)$.
\end{Theorem}

\begin{proof} The uniqueness of the solution is referred to Remark~\ref{Remark_uniqueness}. We now prove the existence.

Recalling the penalized $G$-BSDE (\ref{Penal_GBSDE2}), for $m,n\in\bb N$ and each $t\in[0,T]$, define
$$\hat Y^{n,m}_t:=Y^n_t-Y^m_t,\quad \hat Z^{n,m}_t:=Z^n_t-Z^m_t,\quad \hat K^{n,m}_t:=K^n_t-K^m_t,\quad\hat L^{n,m}_t:=L^n_t-L^m_t,$$
and
$$\hat f^{n,m}_t:=f(t,Y^n_t,Z^n_t)-f(t,Y^m_t,Z^m_t).$$
In view of Lemma \ref{lemma_Y^n convergence}, there exists $Y\in S_G^\alpha(0,T)$ satisfying $$\lim_{n\to\infty}\E[\sup_{t\in[0,T]}|Y_t-Y_t^n|^\alpha]=0, \quad\forall \alpha\geq2.$$
Note that there is $L:=L(L_y,L_z)$ such that for each $t\in[0,T]$,
$$|\hat f^{n,m}_t|\leq L_y|\hat Y^{n,m}_t|+L_z(1+|Z_t^m|+|Z_t^n|)|\hat Z^{n,m}_t|
\leq L(1+|\hat Y^{n,m}_t|+|Z_t^m|^2+|Z_t^n|^2).$$
Applying It\^o's formula to $|\hat Y^{n,m}_t|^2$, we get for each $t\in[0,T]$,
\begin{eqnarray*}
&&|\hat Y^{n,m}_t|^2+\int_t^T |\hat Z^{n,m}_s|^2 d\Bq_s\\
&=&2\int_t^T\hat Y^{n,m}_s\hat f^{n,m}_s d\Bq_s-2\int_t^T\hat Y^{n,m}_sd\hat K^{n,m}_s
+2\int_t^T\hat Y^{n,m}_sd\hat L^{n,m}_s-2\int_t^T\hat Y^{n,m}_s\hat Z^{n,m}_sdB_s\\
&\leq& 2L\int_t^T\hat Y^{n,m}_s(1+|\hat Y^{n,m}_s|+|Z_s^m|^2+|Z_s^n|^2) d\Bq_s-2\int_t^T\hat Y^{n,m}_sd\hat K^{n,m}_s\\
&&+2\int_t^T\hat Y^{n,m}_sd\hat L^{n,m}_s-2\int_t^T\hat Y^{n,m}_s\hat Z^{n,m}_sdB_s,~~\text{q.s.}
\end{eqnarray*}
Setting $t=0$, we have
\begin{eqnarray*}
\int_0^T |\hat Z^{n,m}_s|^2 d\Bq_s
&\leq& 2L\siup^2T\sup_{s\in[0,T]}|\hat Y^{n,m}_s|^2+2L\sup_{s\in[0,T]}|\hat Y^{n,m}_s|\int_0^T(1+|Z_s^m|^2+|Z_s^n|^2) d\Bq_s\\
&&+2\sup_{s\in[0,T]}|\hat Y^{n,m}_s|\sum_{j=m,n}(|K_T^j|+|L_T^j|)
-2\int_0^T\hat Y^{n,m}_s\hat Z^{n,m}_sdB_s,~~\text{q.s.}
\end{eqnarray*}
With the B-D-G inequality and H\"older's inequality, we have
\begin{eqnarray*}
&&\E\Big[\Big(\int_0^T |\hat Z^{n,m}_s|^2 d\Bq_s\Big)^{\alpha\over 2}\Big]\\
&\leq& C_\alpha\bigg\{2L\siup^2T\E\Big[\sup_{s\in[0,T]}|\hat Y^{n,m}_s|^\alpha\Big]+2L\E\Big[\sup_{s\in[0,T]}|\hat Y^{n,m}_s|^{\alpha\over 2}\Big(\int_0^T(1+|Z_s^m|^2+|Z_s^n|^2) d\Bq_s\Big)^{\alpha\over 2}\Big]\\
&&+2\E\Big[\sup_{s\in[0,T]}|\hat Y^{n,m}_s|^{\alpha\over 2}
\Big(\sum_{j=m,n}(|K_T^j|+|L_T^j|)\Big)^{\alpha\over 2}\Big]
+2\E\Big[\Big(\int_0^T|\hat Y^{n,m}_s\hat Z^{n,m}_s|^2ds\Big)^{\alpha\over4}\Big]\bigg\}\\
&\leq& C_\alpha^\prime \E\Big[\sup_{s\in[0,T]}|\hat Y^{n,m}_s|^\alpha\Big]
+C_\alpha^\prime\E\Big[\sup_{s\in[0,T]}|\hat Y^{n,m}_s|^{\alpha}\Big]^{1\over 2}
\bigg\{\E\Big[\Big(\int_0^T(1+|Z_s^m|^2+|Z_s^n|^2) d\Bq_s\Big)^\alpha\Big]^{1\over 2}\\
&&+\E\Big[\sum_{j=m,n}(|K_T^j|^\alpha+|L_T^j|^\alpha)\Big]^{1\over 2}
+\E\Big[\Big(\int_0^T| Z^{n}_s|^2+|Z^m_s|^2ds\Big)^{\alpha\over2}\Big]^{1\over 2}\bigg\}.
\end{eqnarray*}
In view of Lemmas~\ref{lemma_Y^n_bound} and \ref{energy_ineq}, there exists a constant $C_1$ independent of $m$ and $n$, such that
$$
\E\Big[\Big(\int_0^T |\hat Z^{n,m}_s|^2 d\Bq_s\Big)^{\alpha\over 2}\Big]
\leq C_1 \E\Big[\sup_{s\in[0,T]}|\hat Y^{n,m}_s|^\alpha\Big]
+C_1\E\Big[\sup_{s\in[0,T]}|\hat Y^{n,m}_s|^{\alpha}\Big]^{1\over 2}.
$$
In view of Lemma \ref{lemma_Y^n convergence}, we know that $\{Z^n\}_{n=1}^\infty$ is a Cauchy sequence in $H_G^{\alpha}(0,T)$ for $\alpha\geq2$. Thus there exists $Z\in H_G^\alpha(0,T)$ satisfying
$$\lim_{n\to\infty}\E\Big[\Big(\int_0^T|Z_s-Z_s^n|ds\Big)^{\alpha\over 2}\Big]=0, \quad\forall \alpha\geq2.$$
Now set $A^n:=L^n-K^n$. It is easy to check that $(A_t^n)_{t\in[0,T]}$ is a nondecreasing process and
$$A_t^n-A_t^m=\hat Y^{n,m}_0-\hat Y^{n,m}_t-\int_0^t\hat f^{n,m}_s d\Bq_s+\int_0^t \hat Z^{n,m}_s dB_s,~~\text{q.s.}$$
So we get
\begin{eqnarray}
\nonumber&&\E\Big[\sup_{t\in[0,T]}|A_t^n-A_t^m|^\alpha\Big]\\
&\leq& C_2\bigg\{\E\Big[\sup_{t\in[0,T]}|\hat Y^{n,m}_t|^\alpha\Big]
+\E\Big[\Big(\int_0^T|\hat f^{n,m}_s| d\Bq_s\Big)^\alpha\Big]
+\E\Big[\Big(\int_0^T |\hat Z^{n,m}_s|^2 ds\Big)^{\alpha\over 2}\Big]\bigg\}.
\label{ineq_A_m_n}
\end{eqnarray}
From the assumption on $f$, we have
\begin{eqnarray*}
&&\E\Big[\Big(\int_0^T\hat f^{n,m}_s d\Bq_s\Big)^\alpha\Big]\\
&\leq & \E\Big[\Big(\int_0^T\ L_y|\hat Y^{n,m}_s|+L_z(1+|Z_s^m|+|Z_s^n|)|\hat Z^{n,m}_t| d\Bq_s\Big)^\alpha\Big]\\
&\leq& C_3\bigg\{ \E\Big[\sup_{t\in[0,T]}|\hat Y^{n,m}_t|^\alpha\Big]
+\E\Big[\Big(\int_0^T (1+|Z_s^m|+|Z_s^n|)^2d\Bq_s\Big)^{\alpha\over 2}
\Big(\int_0^T |\hat Z^{n,m}_t|^2 d\Bq_s\Big)^{\alpha\over 2}\Big]\bigg\}\\
&\leq& C_3\bigg\{ \E\Big[\sup_{t\in[0,T]}|\hat Y^{n,m}_t|^\alpha\Big]
+\E\Big[\Big(\int_0^T (1+|Z_s^m|+|Z_s^n|)^2d\Bq_s\Big)^{\alpha}\Big]^{1\over 2}\\
&&\times\E\Big[\Big(\int_0^T |\hat Z^{n,m}_t|^2 d\Bq_s\Big)^{\alpha}\Big]^{1\over 2}\bigg\}.
\end{eqnarray*}
Then in view of Lemmas~\ref{lemma_Y^n_bound} and~\ref{energy_ineq} and inequality~(\ref{ineq_A_m_n}), we know that $\{A^n\}_{n=1}^\infty$ is a Cauchy sequence in $S_G^{\alpha}(0,T)$ for each $\alpha\geq2$. There exists a nondecreasing process $(A_t)_{t\in[0,T]}$ such that
$$\lim_{n\to\infty}\E\left[\sup_{t\in [0,T]}|A_t-A_t^n|^\alpha\right]=0.$$

Now, we prove $Y\in S_G^\infty(0,T)$. In view of Lemma \ref{lemma_Y^n_bound}, we know there exist a constant $C>0$ such that $\norm{Y^n}_{S_G^\infty}\leq C$. Recall that
$$\E[X]=\sup_{\bb P\in\mathcal P}E^{\bb P}[X], \quad\forall X\in L_G^1(\Omega_T).$$
From
$$\lim_{n\to\infty}\E\left[\sup_{t\in[0,T]}|Y_t-Y_t^n|^2\right]=0,$$
we see that for each $\bb P\in\mathcal P$,  $\{\sup_{t\in[0,T]}|Y_t^n|, n=1,2,\ldots\}$ converges in probability $\bb P$ to
$\sup_{t\in[0,T]}|Y_t|$. Then, there exits a sub-sequence of $(\sup_{t\in[0,T]}|Y_t^n|)$ such that  ${\bb P}$-a.s.,
$$\lim_{k\to \infty}\sup_{t\in[0,T]}|Y_t^{n_k}|=  \sup_{t\in[0,T]}|Y_t|.$$
Since  $\sup_{t\in[0,T]}|Y_t^{n_k}|\leq C$ for a positive constant $C$  independent of $\bb P$, we have $\sup_{t\in[0,T]}|Y_t|\leq C$ $\bb P\text{\,-a.s.}$ for each $\bb P\in \mathcal P$, and then $\sup_{t\in[0,T]}|Y_t|\leq C$, q.s., which yields the inequality $\norm{Y}_{S_G^\infty}\leq C$. In view of Proposition \ref{Propose_Z_bound}, we have $Z\in BMO_G$.

From Lemma \ref{lemma_Y^n-S convergence}, we have $Y_t\geq S_t$ for  $t\in[0,T]$.
We claim that $\int_0^\cdot (S_s-Y_s)\, dA_s$  is a non-increasing $G$-martingale on $[0,T]$. Set
$\tilde K_t^n=\int_0^t (Y_s-S_s)dK_s^n$. Since $Y_t\geq S_t$  for $t\in[0,T]$ and $K^n$ is a decreasing $G$-martingale, then $\tilde K^n$ is a decreasing $G$-martingale.

We have
\begin{eqnarray*}
&&\sup_{t\in[0,T]}\Big|-\int_0^t(Y_s-S_s)dA_s-\tilde K_t^n\Big|\\
&\leq& \sup_{t\in[0,T]}\biggl\{ \Big|-\int_0^t(Y_s-S_s)dA_s+\int_0^t(Y_s-S_s)dA_s^n\Big|
+\Big|\int_0^t(Y_s^n-Y_s)dA_s^n\Big|\\
&&+\Big|\int_0^t(Y_s^n-Y_s)dK_s^n\Big|+\Big|\int_0^t-(Y_s^n-S_s)dL_s^n\Big|\biggr\}\\
&\leq& \sup_{t\in[0,T]}\Big\{ \Big|\int_0^t(\tilde Y_s^m-\tilde S_s^m)d(A_s^n-A_s)\Big|
+\Big|\int_0^t\{Y_s-S_s-(\tilde Y_s^m-\tilde S_s^m)\}d(A_s^n-A_s)\Big|\Big\}\\
&&+\sup_{t\in[0,T]}|Y_t-Y_t^n|[|A_T^n|+|K_T^n|]+\sup_{t\in[0,T]}(Y_s^n-S_s)^-|L_T^n|,~~\text{q.s.},
\end{eqnarray*}
with
$$\tilde Y_t^m:=\sum_{i=0}^{m-1}Y_{t_i^m}\textbf{1}_{[t_i^m,t_{i+1}^m)}(t), \quad \tilde S_t^m:=\sum_{i=0}^{m-1}S_{t_i^m}\textbf{1}_{[t_i^m,t_{i+1}^m)}(t)$$
and
$$t_i^m:={iT\over m}, \quad i=0,1,\cdots,m.$$

In view of Lemma \ref{lemma_Y^n_bound} and identically as in the proof of \cite[Theorem 5.1]{Li2017_1}, we have
$$\lim_{n\to\infty} \E\Big[\sup_{t\in[0,T]}\Big|-\int_0^t(Y_s-S_s)dA_s-\tilde K_t^n\Big|\Big]=0,$$
which implies that $\int_0^\cdot (S_s-Y_s)\, dA_s$ is a non-increasing $G$-martingale on $[0,T]$.
\end{proof}

In an identical way, we have the following theorem.

\begin{Theorem}
\label{Thm_existence_uniqueness_1}
Suppose that $\xi$, $f$, $g$, and $S$ satisfy \textbf{(H1)}-\textbf{(H5)}. Then the reflected $G$-BSDE (\ref{RGBSDE1}) has a unique solution (Y,Z,A) such that $(Y, Z)\in L_G^{\infty}[0,T]\times BMO_G$ and
$A\in \bigcap_{\alpha\geq 2} S_G^\alpha[0,T]$.
\end{Theorem}

We have the following comparison theorem for reflected quadratic $G$-BSDEs.

\begin{Theorem}
\label{Thm_compare_RBSDE}
Let the set $(\xi^i, f^i, g^i, S^i)$ satisfy \textbf{(H1)}-\textbf{(H5)},  and $(Y^i,Z^i,A^i)\in \mathcal{S}_G^2(0,T)$ be the solution to the following reflected $G$-BSDE:
\begin{equation*}
\begin{cases}
\displaystyle
 Y_t^i=\xi^i+\int_t^Tg^i(s,\omega_{\cdot \land s},Y_s^i,Z_s^i)ds+\int_t^Tf^i(s,\omega_{\cdot\land s},Y_s^i,Z_s^i)d\Bq_s-\int_t^TZ_s^idB_s+\int_t^TdA_s^i,~~\text{q.s.};\\
 \displaystyle
Y_t^i\geq S_t^i,~~\text{q.s.}, \quad 0\leq t\leq T; \quad \int_0^\cdot (S_s^i-Y_s^i)\, dA_s^i \text{ is a non-increasing } G\text{-martingale},
 \end{cases}
 \end{equation*}
with $i=1,2$. Assume that $(Y^i, Z^i)\in S_G^\infty(0,T)\times BMO_G$ and $A^i_T\in \bigcap_{p\geq 1} L_G^p(\Omega_T)$ for  $i=1,2$.
If $\xi^1\geq \xi^2$, $g^1\geq g^2$, $f^1\geq f^2$, and $S^1\geq S^2$,~~\text{q.s.}, then $Y_t^1\geq Y_t^2,~~\text{q.s.}$~ for any $t\in [0,T]$.
\end{Theorem}

\begin{proof}
The proof is identical to that of \cite[Theorem 5.3]{Li2017_1}.

We first consider the following $G$-BSDE
\begin{eqnarray*}
 y_t^n&=&\xi^2+\int_t^Tg^2(s,y_s^n,z_s^n)ds+\int_t^Tf^2(s,y_s^n,z_s^n)d\Bq_s+\int_t^T n(y_s^n-S_s^2)^-ds\\
&& -\int_t^Tz_s^ndB_s-\int_t^TdK_s^n,\quad \forall t\in[0,T],
\end{eqnarray*}
for $n=1,2,\cdots$.
As before, we have
$$\lim_{n\to\infty}\E\Big[\sup_{t\in[0,T]}|Y_t^2-y_t^n|^\alpha\Big]=0,\quad\forall \alpha\geq 2.$$
Noting that $Y_t^1\geq S_t^1$, we can rewrite the equation for $(Y^1,Z^1,A^1)$ as
\begin{eqnarray*}
Y_t^1&=&\xi^1+\int_t^Tg^1(s,Y_s^1,Z_s^1)ds+\int_t^Tf^1(s,Y_s^1,Z_s^1)d\Bq_s\\
&&+\int_t^T n(Y_s^1-S_s^1)^-ds-\int_t^TZ_s^1dB_s+\int_t^TdA_s^1,~~\text{q.s.},~t\in[0,T].
\end{eqnarray*}
Using Theorem \ref{Thm_compare},  we have  $Y_t^1\geq y_t^n,~~\text{q.s.}$~ for all $n\in\bb N$. Letting $n\to \infty$, we conclude that
$Y_t^1\geq Y_t^2,~~\text{q.s.}$
\end{proof}

\section{Relation between quadratic $G$-BSDEs and nonlinear parabolic PDEs }

Consider the following PDE:
\begin{equation}
\begin{cases}
\label{PDE_thm}
\partial_t u+F(D_x^2u,D_xu,u,x,t)=0,~(t,x)\in[0,T)\times \bb R^n;\\
u(T,x)=\phi(x),~x\in\bb R^n,
\end{cases}
\end{equation}
where
\begin{eqnarray*}
F(A,p,y,x,t)
&:=&G\left(\sigma^\mathrm{T}(t,x) A \sigma(t,x)+2f(t,x,y,\sigma^\mathrm{T}(t,x)p)+2h^\mathrm{T}(t,x)p\right)\\
&&+b^\mathrm{T}(t,x)p+g(t,x,y,\sigma^\mathrm{T}(t,x)p),
\end{eqnarray*}
for each $(A,p,y,x,t)\in \bb S_n\times \bb R^n\times \bb R\times\bb R^n\times[0,T] $.

We shall give a nonlinear Feynman-Kac formula for the fully nonlinear PDE (\ref{PDE_thm})  when the functions $f$ and $g$ are quadratic in the last argument. Similar to Li, Peng and Soumana Hima \cite[Section 6]{Li2017_1}, we give the relationship between solutions of the obstacle problem for nonlinear parabolic PDEs and the related reflected quadratic $G$-BSDEs.

In what follows, we consider the $G$-expectation space $(\Omega,L_G^1(\Omega_T),\E)$ for the case of $d=1$ and $\siup^2=\E[B_1^2]\geq-\E[-B_1^2]=\sidown^2>0$.

\subsection{Nonlinear Feynman-Kac formula}

Our main assumptions of this section are formulated as follows.

For deterministic functions $b,h,\sigma: [0,T]\times\bb R^n\to\bb R^n$, $\phi:\bb R^n\to\bb R$, and $f,g:[0,T]\times\bb R^n\times \bb R\times\bb R\to\bb R$, we make the following assumptions.

 \begin{itemize}
\item[\textbf{(A1)}]The functions $b,h,\sigma,f,g$ are uniformly continuous in $t$, i.e. there is a non-decreasing continuous function $w: [0, +\infty)\to [0, +\infty)$ such that $w(0)=0$ and
    $$\sup_{x,y,z\in \bb{R}}|l_1(t,x,y,z)-l_1(t^\prime,x,y,z)|\leq w(|t-t^\prime|),\quad l_1=f,g,$$
        $$\sup_{x\in\bb{R}}|l_2(t,x)-l_2(t^\prime,x)|\leq w(|t-t^\prime|),\quad l_2=b,h,\sigma;$$
\item[\textbf{(A2)}] There exist a positive integer $m$ and a constant $L>0$ such that for each
$(t,x,x^\prime,y,y^\prime,z,z^\prime)\in[0,T]\times\bb R^n\times\bb R^n\times \bb R\times\bb R\times\bb R\times\bb R$,
\begin{eqnarray*}
 &&|b(t,x)-b(t,x^\prime)|+|h(t,x)-h(t,x^\prime)|+|\sigma(t,x)-\sigma(t,x^\prime)|\leq L|x-x^\prime|,\\
 &&|\phi(x)-\phi(x^\prime)|\leq L(1+|x|^m+|x^\prime|^m)|x-x^\prime|,\\
 &&|f(t,x,y,z)-f(t,x^\prime,y^\prime,z^\prime)|+|g(t,x,y,z)-g(t,x^\prime,y^\prime,z^\prime)|\\
 &&\qquad\leq L[(1+|x|^m+|x^\prime|^m)|x-x^\prime|+|y-y^\prime|+(1+|z|+|z^\prime|)|z-z^\prime|];
\end{eqnarray*}
\item[\textbf{(A3)}] There is a positive constant $M_0$ such that
$$\int_0^T\sup_{x\in \bb{R}^n}\left[|f(t,x,0,0)|^2+|g(t,x,0,0)|^2\right]\, dt+\sup_{x\in \bb{R}^n}|\phi(x)|\leq M_0;$$
\item[\textbf{(A4)}]
There are two constants  $\varepsilon>0$ and $K>0$ such that for each $(t,x)\in[0,T]\times \bb R^n$,
 $$\varepsilon I\leq \sigma\sigma^\mathrm{T}(t,x)\leq KI.$$
\end{itemize}

\begin{Remark}
\label{remark_sigma}
Assumption \textbf{(A4)} implies  that $\sigma$ is bounded on $[0,T]\times \bb R^n$.
\end{Remark}

For each $(t,\xi)\in[0,T]\times \bigcap\limits_{p\geq2}L_G^p(\Omega_t;\bb R^n)$, we consider the following of $G$-SDE:
\begin{equation}
X_s=\xi+\int_t^sb(u, X_u)du+\int_t^sh(u, X_u)d\Bq_u+\int_t^s\sigma(u, X_u)dB_u,~~\text{q.s.},~s\in[t,T].
\label{SDE0}
\end{equation}
Denote by  $X^{t,\xi}$ the solution to $G$-SDE (\ref{SDE0}).  Then, we have
\begin{Proposition}
\label{propose_estimate_sde}
(See \cite[Chapter \uppercase\expandafter{\romannumeral5}]{Peng2010}) Let $\xi,\xi^\prime\in L_G^p(\Omega_t;\bb R^n)$ with $p\geq 2$.
Then we have, for each $\delta\in[0,T-t]$,
\begin{eqnarray*}
&&\E_t \left[\sup_{s\in[t,t+\delta]}\left|X_{s}^{t,\xi}-X_{s}^{t,\xi^\prime}\right|^p\right]\leq C|\xi-\xi^\prime|^p,\\
&&\E_t \left[\sup_{s\in[t,t+\delta]}\left|X_{s}^{t,\xi}\right|^p\right]\leq C(1+|\xi|^p),\\
&&\E_t \left[\sup_{s\in[t,t+\delta]}\left|X_{s}^{t,\xi}-\xi\right|^p\right]\leq C(1+|\xi|^p)\delta^{p\over2},
\end{eqnarray*}
where the constant $C$ depends on $L,G,p,n$ and $T$.
\end{Proposition}

\begin{Proposition}
\label{Propose_stable_sde}
Let the triplet $(b^i,h^i,\sigma^i)$ satisfy \textbf{(A1)}-\textbf{(A2)} for $i=1,2$. For each $(t,\xi)\in[0,T]\times L_G^p(\Omega_t;\bb R^n)$, $p\geq2$, let $X^{t,\xi,i}$  be the solution to the following $G$-SDE:
	$$ X_s^{t,\xi,i}=\xi+\int_t^sb^i(s, X_u^{t,\xi,i})du+\int_t^sh^i(u, X_u^{t,\xi,i})d\Bq_u+\int_t^s\sigma^i(u, X_u^{t,\xi,i})dB_u,
~~\text{q.s.},~ s\in[t,T].$$
Then for each $\delta\in[0,T-t]$, there exist a constant $C$ depends only on $L,G,p$ and $T$, such that
\begin{eqnarray*}
&&\E_t \left[\left|X_{t+\delta}^{t,\xi,1}-X_{t+\delta}^{t,\xi,2}\right|^p\right]\\
&\leq& C\left(\E_t\left[\left(\int_t^{t+\delta}|\hat b_u|du\right)^{p} \right]+\E_t\left[\left(\int_t^{t+\delta}|\hat h_u|du\right)^{p} \right]
+\E_t\left[\left(\int_t^{t+\delta}|\hat\sigma_u|^2du\right)^{p\over2} \right] \right),~~\text{q.s.},
\end{eqnarray*}
where for each $u\in[t,T]$,
$$\hat l_u:=l^1(u,X_u^{t,\xi,2})-l^2(u,X_u^{t,\xi,2}), \quad l=b,h,\sigma.$$
\end{Proposition}
\begin{proof}
For simplicity, we assume $h^i=0$. Then we have
\begin{eqnarray*}
X_{t+\delta}^{t,\xi,1}-X_{t+\delta}^{t,\xi,2}
&=&\int_t^{t+\delta}\hat b_u+b^1\left(t,X_{t+\delta}^{t,\xi,1}\right)-b^1\left(t,X_{t+\delta}^{t,\xi,2}\right)du\\
&&+\int_t^{t+\delta}\hat \sigma_u+\sigma^1\left(t,X_{t+\delta}^{t,\xi,1}\right)-\sigma^1\left(t,X_{t+\delta}^{t,\xi,2}\right)dB_u.
\end{eqnarray*}
In view of BDG inequality, we have
\begin{eqnarray*}
&&\E_t \left[\left|X_{t+\delta}^{t,\xi,1}-X_{t+\delta}^{t,\xi,2}\right|^p\right]\\
&\leq & C_1\left( \E_t\left[\int_t^{t+\delta}\left|X_{u}^{t,\xi,1}-X_{u}^{t,\xi,2}\right|^p du \right]+\E_t\left[\left(\int_t^{t+\delta}|\hat b_u|du\right)^{p} \right]
+\E_t\left[\left(\int_t^{t+\delta}|\hat\sigma_u|^2du\right)^{p\over2} \right] \right)\\
&\leq & C_1\left( \int_t^{t+\delta}\E_t\left[\left|X_{u}^{t,\xi,1}-X_{u}^{t,\xi,2}\right|^p du \right]+\E_t\left[\left(\int_t^{t+\delta}|\hat b_u|du\right)^{p} \right]
+\E_t\left[\left(\int_t^{t+\delta}|\hat\sigma_u|^2du\right)^{p\over2} \right] \right).\\
\end{eqnarray*}
By the Gronwall's inequality, we obtain
$$
\E_t \left[\left|X_{t+\delta}^{t,\xi,1}-X_{t+\delta}^{t,\xi,2}\right|^p\right]
\leq C_1e^{C_1T}\left(\E_t\left[\left(\int_t^{t+\delta}|\hat b_u|du\right)^{p} \right]
+\E_t\left[\left(\int_t^{t+\delta}|\hat\sigma_u|^2du\right)^{p\over2} \right] \right).
$$
\end{proof}
We now consider the following $G$-BSDE.
\begin{eqnarray}
\nonumber Y_s&=&\phi(X_T^{t,\xi})+\int_s^Tg(u, X_u^{t,\xi},Y_u,Z_u)du+\int_s^Tf(u, X_u^{t,\xi},Y_u,Z_u)d\Bq_u\\
 &&-\int_s^TZ_udB_u-\int_s^TdK_u,\quad \text{q.s.},~s\in[t,T].
 \label{BSDE0}
\end{eqnarray}
We should point out that Theorem 5.3 in Hu et al. \cite{Hu2018} can not be used to our case directly. Because it is hard to check the assumption \textbf{(H2)} directly in our Markovian case. We now give the the existence of the solution to $G$-BSDE (\ref{BSDE0}) in the spirit of the method in  Hu et al. \cite{HuM2014_1} and Hu et al. \cite{Hu2018}.

Without loss of generality, we assume $h=0$, $g=0$ and $t=0$. For each $x_0\in\bb R^n$, we consider the following  forward and backward differential equations in the $G$-framework ($G$-FBSDE )
\begin{flalign}
&X_t=x_0+\int_0^tb(u, X_u)du+\int_0^t\sigma(u, X_u)dB_u,~~\text{q.s.},~t\in[0,T],
\label{SDE1}\\
 &Y_t=\phi(X_T)+\int_t^Tf(u, X_u,Y_u,Z_u)d\Bq_u-\int_t^TZ_udB_u-\int_t^TdK_u,~~\text{q.s.},~ t\in[0,T],
 \label{BSDE1}
\end{flalign}
where $b,\sigma,f$ and $\phi$ satisfy \textbf{(A1)}-\textbf{(A4)}.

First, we introduce the following fully nonlinear PDE on $[0,T]$:
\begin{equation}
\begin{cases}\displaystyle
\label{PDE1}
\partial_t u+G\left(\sigma^\mathrm{T}(t,x) D_x^2u \sigma(t,x)+2f(t,x,u,\sigma^\mathrm{T}(t,x)D_xu)\right)+b^\mathrm{T}(t,x)D_xu=0,\\
\qquad\qquad\qquad\qquad\qquad\qquad\qquad \qquad\qquad\qquad\qquad\qquad
(t,x)\in[0,T)\times\bb R^n;\\
u(T,x)=\phi(x),~~x\in\bb R^n.
\end{cases}
\end{equation}
We make the following assumptions on  the coefficients of the PDE (\ref{PDE1}).
 \begin{itemize}
\item[\textbf{(A5)}]The function $f(t,x,y,z)$ is continuously differentiable in $(x,y,z)$, differentiable in $t$, and twice differentiable in $(x,y,z)$, where  the first-order time derivative of $f$  and the second-order derivatives of $f$ in $(x,y,z)$ are bounded on the set $[0,T]\times\bb R^n\times
    [-M_y,M_y]\times[-M_z,M_z]$, for any $M_y, M_z>0$.

\item[\textbf{(A6)}]Both functions $b$ and $\sigma$ are differentiable in $t$ and twice differentiable in $x$, where the first-order time  derivative of $(b, \sigma)$  and the second-order spatial derivatives of $(b, \sigma)$  are bounded on the set $[0,T]\times\bb R^n$.

\item[\textbf{(A7)}]The functions $b$ is bounded on the set $[0,T]\times\bb R^n$. The function $f$ is bounded on the set $[0,T]\times\bb R^n\times\bb R\times\bb R$.

\item[\textbf{(A8)}]There exists a constant $L>0$ such that for each
$(t,y,z)\in[0,T]\times\bb R\times\bb R$,
$$|\phi(x)-\phi(x^\prime)|+|f(t,x,y,z)-f(t,x^\prime,y,z)|\leq L|x-x^\prime|,\quad \forall x,x^\prime\in\bb R^n.$$
\end{itemize}

Note that Peng \cite[Appendix C]{Peng2010} used Krylov \cite[Theorem 6.4.3]{Krylov1987} to prove that  there is a classical solution to PDE (\ref{PDE1}) when $b=0$, $f=0$, and $\sigma=1$. In a similar way,
we prove that there is  a classical solution to PDE (\ref{PDE1}) and further that  $u(t,\cdot)$ is uniformly Lipschitz continuous.

\begin{Proposition}
\label{Propose_lip_pde}
Assume $b,\sigma,f$ and $\phi$ satisfy \textbf{(A1)}-\textbf{(A8)}. Then the PDE (\ref{PDE1}) admits a classical solution $u\in C([0,T]\times\bb R^n)$ bounded by $M:=M(M_0,L)$, and there exists a constant $\alpha\in(0,1)$ such that for each $k\in(0,T)$,
$$\norm{u}_{C^{1+\alpha /2,2+\alpha}([0,T-k]\times\bb R^n)}<\infty.$$
Moreover, there exists a constant $C>0$ such that for all $t\in[0,T]$,
$$|u(t,x)-u(t,x^\prime)|\leq C|x-x^\prime|,\quad \forall x,x^\prime\in\bb R^n. $$
\end{Proposition}
\begin{proof}
First, we introduce the truncation function. For each integer $N$, let $\rho_N:\bb R\to\bb R$ be a a smooth modification of the projection on
$[-N,N]$ such that $|\rho_N|\leq N$, $|\rho_N^\prime|\leq 1$ and $\rho_N(z)=z$ when $|z|\leq N-1$.
We consider the following PDE.
\begin{equation}
\begin{cases}
\label{PDE1_truncation}
\partial_t u+G\left(\sigma^\mathrm{T}(t,x) D_x^2u \sigma(t,x)+2f^N(t,x,u,\sigma^\mathrm{T}(t,x)D_xu)\right)+b^\mathrm{T}(t,x)D_xu=0,\\
\qquad\qquad\qquad\qquad\qquad\qquad\qquad \qquad\qquad\qquad\qquad\qquad
(t,x)\in[0,T)\times\bb R^n;\\
u(T,x)=\phi(x),~~x\in \bb R^n,
\end{cases}
\end{equation}
where $f^N$ is defined as
$$f^N(t,x,y,z):=f(t,x,y,\rho_N(z)),~~\forall (t,x,y,z)\in[0,T]\times\bb R^n\times\bb R\times\bb R.$$
 It is easy to check that $f^N$ is uniformly Lipschitz in $z$.

Considering  the PDE for the quantity $e^{(L\siup^2+1)(t-T)}u(t,x)$ as in Peng \cite[Appendix C]{Peng2010}, in view of  Krylov~\cite[Theorem 6.4.3]{Krylov1987}, we can prove that the PDE (\ref{PDE1_truncation}) admits a classical solution
$u^N\in C([0,T]\times\bb R^n)$ dominated by a constant $M:=M(M_0,L)$, such that   for some constant $\alpha\in(0,1)$,  the related restriction of  $u^N$ belong to $C^{1+\alpha /2,2+\alpha}([0, k]\times\bb R^n)$ with any $k\in(0,T)$.

We now rewrite  PDE (\ref{PDE1_truncation}) into a HJB equation, and then estimate the gradient $D_xu^N$.
Since
$$G(a)={1\over2}(\siup^2a^+-\sidown^2a^-)=\sup_{v\in[\sidown,\siup]}{1\over2}v^2a,$$
 the PDE (\ref{PDE1_truncation}) is the following  HJB equation:
\begin{equation}
\begin{cases}
\label{PDE_HJB} \displaystyle
\partial_t u+\sup_{v\in[\sidown,\siup]}H^N(t,x,u(t,x),D_xu(t,x),D^2_xu(t,x),v)=0,\\[3mm]
u(T,x)=\phi(x).
\end{cases}
\end{equation}
where  the Hamiltonian $H^N$ is defined as follows: for $(t,x,y,p,A,v)\in [0,T]\times\bb R^n\times\bb R\times\bb R^n\times\bb S_n\times[\sidown,\siup]$,
$$H^N(t,x,y,p,A,v):={1\over2}\hat\sigma^\mathrm{T}(t,x,v) A \hat\sigma(t,x,v)+F^N(t,x,y,\hat\sigma^\mathrm{T}(t,x)p,v)+b^\mathrm{T}(t,x)p$$
with
$$\hat\sigma(t,x,v):=v\sigma(t,x) \quad \hbox{\rm and } \quad F^N(t,x,y,z,v):=v^2f^N(t,x,y,{z/v}) \hbox{ \rm for }  z\in\bb R.$$
This shows that  $u^N$ is in fact a value function of a control problem.

Let $(\Omega,\mathcal{F},\bb P)$ be the classical Wiener space. Let $W$ be a one-dimension standard Brownian motion under Probability $\bb P$.
For each $(t,x)\in[0,T]\times \bb R^n$, we consider the FBSDE:
 \begin{eqnarray*}
\mathcal{X}_s^{t,x,v}&=&x+\int_t^sb(r, \mathcal X_r^{t,x,v})\, dr+\int_t^s\hat\sigma(r,  \mathcal X_r^{t,x,v},v_r)\, dW_r,~~\bb P\text{\,-a.s.},~s\in[t,T],\\
 \mathcal  Y_s^{t,x,v,N}&=&\phi( \mathcal X_T^{t,x,v})+\int_s^TF^N(r,  \mathcal X_r^{t,x,v}, \mathcal Y_r^{t,x,v,N}, \mathcal Z_r^{t,x,v,N},v_r)\, dr\\
&& -\int_s^T \mathcal Z_r^{t,x,v,N}dW_r,~~\bb P\text{\,-a.s.},~s\in[t,T],
\end{eqnarray*}
Let $\mathcal F^W$ be the filtration generated by $W$ and augmented by all $\bb P$-null sets. Let $\mathcal V$ be the set of
all $\mathcal F^W_t$-progressively measurable processes valued in $[\sidown,\siup]$.
In view of \cite[Theorem 4.2]{Peng1992} or \cite[Theorem 4.2, Theorem 5.3]{Buckdahn2008} and noting that $u^N$ is a viscosity solution of the PDE (\ref{PDE_HJB}), we have
$$u^N(t,x)=\sup_{v\in\mathcal V} \mathcal Y_t^{t,x,v,N}.$$
Note that for each $z\in \bb R$,
$$|\rho_N(z)|\leq z,\quad \text{and} \quad |\rho_N^\prime(z)|\leq 1.$$
We have for each $(t,x,y,v)\in[0,T]\times \bb R^n\times\bb R\times [\sidown,\siup]$,
\begin{eqnarray*}
&&|F^N(t,x,y,z,v)-F^N(t,x,y^\prime,z^\prime,v)|\\
&=&\left|v^2f^N\left(t,x,y,{z\over v}\right)-v^2f^N\left(t,x,y^\prime,{z^\prime\over v}\right)\right|\\
&\leq& Lv^2\left(1+\left|\rho_N\left(z\over v\right)\right|+\left|\rho_N\left(z^\prime\over v\right)\right|\right)
\left|\rho_N\left(z\over v\right)-\rho_N\left(z^\prime\over v\right)\right|\\
&\leq& L(\siup+|z|+|z^\prime|)|z-z^\prime|.
\end{eqnarray*}
In view of \cite[Lemma 1]{Briand2006}, there exists a constant $C_1$ independent of $v$ and $N$ such that for each $(t,x)\in[0,T]\times \bb R^n$, $$\sup_{s\in[t,T]}| \mathcal Y_s^{t,x,v,N}|\leq C_1.$$
In view of \cite[Proposition 2.1]{Briand2013}, there exists a constant $C_2$ independent of $v$ and $N$ such that for each $(t,x)\in[0,T]\times \bb R^n$,
$$\norm{\int_t^\cdot  \mathcal Z_s^{t,x,v,N}dW_s}_{BMO(\bb{P})}\leq C_2.$$
By a similar stability result as in \cite[Theorem 5.1]{Ankirchner2007}, there exist a constant $C_3$ and some $p>2$ which are independent of $v$ and $N$ such that for each $(t,x,x^\prime)\in[0,T]\times \bb R^n\times \bb R^n$,
\begin{eqnarray*}
|\mathcal Y_t^{t,x,v,N}-\mathcal Y_t^{t,x^\prime,v,N}|\leq C_3\left\{E^{\bb P}\left[|\phi(\mathcal X_T^{t,x,v})-\phi(\mathcal X_T^{t,x^\prime,v})|^p\right]^{1\over p}
+E^{\bb P}\left[\left(\int_t^T |\delta F^N_s|ds\right)^p\right]^{1\over p}\right\},
\end{eqnarray*}
 where
 $$\delta F^N_s:=F^N(s, \mathcal X_s^{t,x,v},\mathcal Y_s^{t,x,v,N},\mathcal Z_s^{t,x,v,N},v_s)
 -F^N(s, \mathcal X_s^{t,x^\prime,v},\mathcal Y_s^{t,x,v,N},\mathcal Z_s^{t,x,v,N},v_s).$$
Thus we get,
\begin{eqnarray*}
|\mathcal Y_t^{t,x,v,N}-\mathcal Y_t^{t,x^\prime,v,N}|\leq LC_3\left\{E^{\bb P}\left[|\mathcal X_T^{t,x,v}-\mathcal X_T^{t,x^\prime,v}|^p\right]^{1\over p}
+E^{\bb P}\left[\left(\int_t^T |\mathcal X_s^{t,x,v}-\mathcal X_s^{t,x^\prime,v}|ds\right)^p\right]^{1\over p}\right\}.
\end{eqnarray*}
By inequality (3.3) in \cite{Buckdahn2008}, there exists a constant $C_4$ independent of $v$ and $N$ such that for each $(t,x,x^\prime)\in[0,T]\times \bb R^n\times \bb R^n$,
$$E^{\bb P}\left[\sup_{s\in[t,T]}|\mathcal X_s^{t,x,v}-\mathcal X_s^{t,x^\prime,v}|^p\right]\leq C_4|x-x^\prime|^p.$$
Thus there exists a constant $C$ independent of $v$ and $N$ such that for each $(t,x,x^\prime)\in[0,T]\times \bb R^n\times \bb R^n$,
$$|\mathcal Y_t^{t,x,v,N}-\mathcal Y_t^{t,x^\prime,v,N}|\leq C|x-x^\prime|,$$
which means $|u^N(t,x)-u^N(t,x^\prime)|\leq C|x-x^\prime|$.
So we get
$$\sup_{(t,x)\in[0,T]\times \bb R^n}|D_xu^N(t,x)|\leq C.$$
In view of Remark \ref{remark_sigma}, we know $\sigma$ is bounded.
Let $$N>1+C\cdot\max_{(t,x)\in[0,T]\times \bb R^n}|\sigma(t,x)|.$$ It is easy to check
that for each $(t,x)\in[0,T]\times\bb R^n$,
$$f^N(t,x,u^N,\sigma^\mathrm{T}(t,x)D_xu^N)=f(t,x,u^N,\sigma^\mathrm{T}(t,x)D_xu^N).$$
Set $u:=u^N$ and we know $u$ is the solution of PDE \eqref{PDE1}.
\end{proof}

For each $t\in[0,T]$, we set $Y_t:=u(t,X_t)$, $Z_t:=\sigma^\mathrm{T}(t,X_t)D_xu(t,X_t)$ and
\begin{eqnarray*}
K_t&=&\int_0^t{1\over2}D_x^2u(s,X_s)+f\left(s,X_s,u(s,X_s),\sigma^\mathrm{T}(s,X_s)D_xu(s,X_s)\right)d\Bq_s\\
&&-\int_0^tG\left(D_x^2u(s,X_s)+f\left(s,X_s,u(s,X_s),\sigma^\mathrm{T}(s,X_s)D_xu(s,X_s)\right)\right)ds.
\end{eqnarray*}
For any $0<k<T$, applying It\^o's formula to $u(t,X_t)$ for $t\in[0,k]$, we get
\begin{equation}
\label{BSDE_k}
Y_t=u(k,X_{k})+\int_t^{k} f(s,X_s,Y_s,Z_s)d\Bq_s-\int_t^{k}Z_sdB_s-\int_t^{k}dK_s,~~\text{q.s.}
\end{equation}

Similarly to \cite[inequality (4.3)]{HuM2014_1}, in view of Proposition \ref{Propose_lip_pde} and (\ref{BSDE_k}), we could obtain that there exist a constant $C>0$, for $t_1,t_2\in[0,T]$ and $x_1,x_2\in\bb R^n$,
$$|u(t_1,x_1)-u(t_2,x_2)|\leq C\left(\left(1+|x_1|\right) \sqrt{|t_1-t_2|} +|x_1-x_2|\right).$$
Then we can deduce that $(Y,Z,K)\in \bigcap\limits_{p\geq2}\mathfrak{G}_G^p(0,T)$ and $(Y,Z,K)$ is a solution to (\ref{BSDE1}). So we have the following lemma.
\begin{Lemma}
Assume $b,\sigma,f$, and $\phi$ satisfy assumptions \textbf{(A1)}-\textbf{(A8)}. Then $G$-BSDE (\ref{BSDE1}) has a solution $(Y,Z,K)\in\bigcap\limits_{p\geq2}\mathfrak{G}_G^p(0,T)$.
\end{Lemma}

As an immediate consequence of both proofs of \cite[Proposition 3.5]{Hu2018} and  Proposition \ref{Propose_stable}, we have the following  stability property for quadratic $G$-BSDEs.

\begin{Proposition}
\label{Propose_stable_pde}
Let the triplet $(\xi^i,f^i,g^i)$ satisfy assumption \textbf{(H1)} and \textbf{(H3)} for $i=1,2$. Let $(Y^i,Z^i,K^i)\in \mathfrak{G}_G^2(0,T)$  be the solution to the following $G$-BSDE:
	$$ Y_t^i=\xi^i+\int_t^Tf^i(s,Y_s^i,Z_s^i)\, d\Bq_s
	-\int_t^TZ_s^i\, dB_s-\int_t^T dK_s^i,~~\text{q.s.},~t\in[0,T].$$
 Moreover, we suppose
 $$\norm{L_z(1+|Z_1|+|Z_2|)}_{BMO_G}< \phi(q):=\left\{1+{1\over q^2}\log{2q-1\over2(q-1)}\right\}^{1\over2}-1.$$
Then for each $p>{q\over q-1}$, there exists a constant $~C:=C(p,T,L,M_0)~$ such that for any $t \in [0,T]$,
\begin{equation*}
|Y_t^1-Y_t^2|\leq C\left\{\E_t\left[|\xi^1-\xi^2|^p\right]^{1\over p}
+\E_t\left[\left(\int_t^T |(f^1-f^2)(s,Y_s^2,Z_s^2)|d\Bq_s\right)^p\right]^{1\over p}\right\},
~~\text{q.s.}
\end{equation*}
\end{Proposition}

\begin{Remark}
\label{remark_stable_pde} We still have
 \cite[inequality (3.2)]{Hu2018}, which means that  the constant
$C_1$ and $p$ in Proposition \ref{Propose_stable_pde} depend only on $T,L$ and $M_0$.
\end{Remark}

The main result of  this section is stated as follows.

\begin{Theorem}
Assume that $b,\sigma,f$, and $\phi$ satisfy assumptions \textbf{(A1)}-\textbf{(A4)}. Then $G$-BSDE (\ref{BSDE1}) has a unique solution $(Y,Z,K)\in\bigcap\limits_{p\geq2}\mathfrak{G}_G^p(0,T)$.
\end{Theorem}
\begin{proof}
The uniqueness result directly comes from Proposition \ref{Propose_stable_pde}. Now we focus on the existence result. We borrow the idea of  Hu et al. \cite{Hu2018} to mollify  the coefficients of the $G$-FBSDE.

\textbf{Step 1.} We assume $f$ satisfy the following condition.
 \begin{itemize}
\item[\textbf{(A5')}]The first-order time  derivative of $f$ in $t$, the spatial derivatives of $f$ up to the second-order are bounded on the set $[0,T]\times\bb R^n\times
    [-M_y,M_y]\times[-M_z,M_z]$, for any $M_y, M_z>0$.
\end{itemize}
We replace \textbf{(A5)} with \textbf{(A5')}. Assume $b,\sigma,f$, and $\phi$ satisfy assumptions \textbf{(A1)}-\textbf{(A4)}, assumption \textbf{(A5')} and assumptions \textbf{(A6)}-\textbf{(A8)}. Then we can obtain that $G$-BSDE (\ref{BSDE1}) has one a solution $(Y,Z,K)\in\bigcap_{p\geq2}\mathfrak{G}_G^p(0,T)$ with exactly the same method of step 1 in Hu et al. \cite[Section 5]{Hu2018}.

\textbf{Step 2.} We assume $b,\sigma,f$, and $\phi$ satisfy assumptions \textbf{(A1)}-\textbf{(A4)} and assumptions \textbf{(A6)}-\textbf{(A8)}. For each $(t,x,y,z)\in[0,T]\times \bb R^n\times\bb R\times\bb R$, we define
$$f^k(t,x,y,z):=\int_{\bb R^{n+1}}f(t-\tilde t,x-\tilde x,y,z)\rho_k(\tilde t,\tilde x)d\tilde t d\tilde x,$$
where $\rho_k$ is a positive smooth function such that its support is contained in a ${1\over k}$-ball in $\bb R^{n+1} $ and $\int_{\bb R^{n+1}}\rho_k=1$. In addition, we define the extension of $f$ on $\bb R$, i.e., $f(t,\cdot,\cdot,\cdot)=f(t^+\land T,\cdot,\cdot,\cdot)$. We can check that $f^k$ satisfies \textbf{(A5')}. Therefore, in view of the result in step 1, we obtain that the $G$-BSDE (\ref{BSDE1}) with the coefficients $(\phi,f^k)$ admits a solution $(Y^k,Z^k,K^k)\in\bigcap_{p\geq2}\mathfrak{G}_G^p(0,T)$. Noting that for each $k_1\geq k_2$ and $t\in[0,T]$,
\begin{eqnarray*}
&&|(f^{k_1}-f^{k_2})(t,X_t,Y_t^{k_2},Z_t^{k_2})|\\
&\leq& |(f^{k_1}-f)(t,X_t,Y_t^{k_2},Z_t^{k_2})|+|(f^{k_2}-f)(t,X_t,Y_t^{k_2},Z_t^{k_2})|\\
&\leq& \int_{\bb R^{n+1}}|f(t-\tilde t,X_t-\tilde x,Y_t^{k_2},Z_t^{k_2})-f(t,X_t,Y_t^{k_2},Z_t^{k_2})|\rho_{k_1}(\tilde t,\tilde x)d\tilde t d\tilde x\\
&&+ \int_{\bb R^{n+1}}|f(t-\tilde t,X_t-\tilde x ,Y_t^{k_2},Z_t^{k_2})-f(t,X_t,Y_t^{k_2},Z_t^{k_2})|\rho_{k_2}(\tilde t,\tilde x)d\tilde t d\tilde x\\
&\leq& w\left({1\over k_1}\right)+w\left({1\over k_2}\right)+L\left({1\over k_1}+{1\over k_2}\right),
\end{eqnarray*}
we could deduce that the sequence $\{Y^k\}_{k=1}^\infty$ is a Cauchy sequence in $S_G^p(0,T)$ for any $p\geq 2$ by Proposition \ref{Propose_stable_pde} and Remark \ref{remark_stable_pde}.
Thus we could conclude that $G$-BSDE (\ref{BSDE1}) has a solution $(Y,Z,K)\in\bigcap_{p\geq2}\mathfrak{G}_G^p(0,T)$  in a similar way as in step 1 in Hu et al. \cite[Section 5]{Hu2018}.

\textbf{Step 3.} We assume $b,\sigma,f$, and $\phi$ satisfy assumptions \textbf{(A1)}-\textbf{(A4)} and assumptions \textbf{(A7)}-\textbf{(A8)}. For each $(t,x)\in[0,T]\times \bb R^n$, we define
$$l^k(t,x):=\int_{\bb R^{n+1}}l(t-\tilde t,x-\tilde x)\rho_k(\tilde t,\tilde x)d\tilde t d\tilde x,\quad l=b,\sigma,$$
where $\rho_k$ is a positive smooth function such that its support is contained in a ${1\over k}$-ball in $\bb R^{n+1} $ and $\int_{\bb R^{n+1}}\rho_k=1$.
We can check that $b^k$ and $\sigma^k$ satisfies \textbf{(A6)}.
Moreover, \textbf{(A4)} still hold here for $\sigma^k$ when $k$ is large enough. Actually, if we assume $\sigma$ satisfies \textbf{(A4)} with constants $\varepsilon$ and $K$, we can check that for $k$ large enough and for each $(t,x)\in[0,T]\times \bb R^n$,
 $${\varepsilon\over 2} I\leq \sigma^k(t,x)({\sigma^k})^\mathrm{T}(t,x)\leq 2KI.$$
Let $X^k$ be the solution of the following $G$-SDE:
$$X_t^k=x_0+\int_0^tb^k(u, X_u^k)du+\int_0^t\sigma^k(u, X_u^k)dB_u,~~\text{q.s.},~t\in[0,T].$$
From step 2, we can let $(Y^k,Z^k,K^k)$ be the solution to the following $G$-BSDE:
	$$ Y_t^k=\phi(X_T^k)+\int_t^Tf(s,X_s^k,Y_s^k,Z_s^k)d\Bq_s-\int_t^TZ_s^kdB_s-\int_t^T dK_s^k,~~\text{q.s.},~t\in[0,T].$$
For each $k_1,k_2 \in \bb N$ and $t\in[0,T]$, set
$$\hat l_t:=l^{k_1}(t,X_t^{k_2})-l^{k_2}(t,X_t^{k_2}),\quad l=b,\sigma.$$
It is easy to check that for $l=b,\sigma$ and each $t\in[0,T]$,
\begin{eqnarray*}
|\hat l_t|&\leq &|l^{k_1}(t,X_t^{k_2})-l(t,X_t^{k_2})|+|l^{k_2}(t,X_t^{k_2})-l(t,X_t^{k_2})|\\
&\leq& \int_{\bb R^{2}}|l(t-\tilde t,X_t^{k_2}-\tilde x)-l(t,X_t^{k_2})|\rho_{k_1}(\tilde t,\tilde x)d\tilde t d\tilde x\\
&&+\int_{\bb R^{2}}|l(t-\tilde t,X_t^{k_2}-\tilde x)-l(t,X_t^{k_2})|\rho_{k_2}(\tilde t,\tilde x)d\tilde t d\tilde x\\
&\leq& w\left({1\over k_1}\right)+w\left({1\over k_2}\right)+L\left({1\over k_1}+{1\over k_2}\right).
\end{eqnarray*}
In view of Proposition \ref{Propose_stable_sde}, we obtain for each $p\geq 2$ and $t\in[0,T]$,
\begin{eqnarray}
\nonumber \E \left[\left|X_{t}^{k_1}-X_t^{k_2}\right|^p\right]
 &\leq& C\left(\E\left[\left(\int_0^{t}|\hat b_s|ds\right)^{p} \right]
 +\E\left[\left(\int_0^{t}|\hat \sigma_s|^2ds\right)^{p\over2} \right] \right)\\
 \label{eq_X_estimate1}
&\leq& C_p \left\{w\left({1\over k_1\land k_2}\right)+\left({1\over k_1\land k_2}\right)\right\}^p.
\end{eqnarray}
On the other hand, in view of Proposition \ref{Propose_stable_pde} and \cite[(3.2)]{Hu2018}, we obtain that for some $p>1$ and
each $t\in[0,T]$,
\begin{eqnarray*}
|Y_t^{k_1}-Y_t^{k_2}|&\leq &C\E_t\left[|\phi(X_{T}^{k_1})-\phi(X_{T}^{k_2})|^p\right]^{1\over p}\\
&&+C\E_t\left[\left(\int_t^T |f(s,X_{s}^{k_1},Y_s^{k_2},Z_s^{k_2})-f(s,X_{s}^{k_2},Y_s^{k_2},Z_s^{k_2})|d\Bq_s\right)^p\right]^{1\over p}\\
&\leq & C_1\E_t\left[|X_{T}^{k_1}-X_{T}^{k_2}|^p\right]^{1\over p}+C_2\E_t\left[\int_t^T |X_{s}^{k_1}-X_{s}^{k_2}|^pds\right]^{1\over p}~~\text{q.s.}
\end{eqnarray*}
Therefore,
\begin{eqnarray*}
&&\E\left[\sup_{t\in[0,T]}|Y_t^{k_1}-Y_t^{k_2}|\right]\\
&\leq & C_3\E\left[\sup_{t\in[0,T]}\E_t\left[|X_{T}^{k_1}-X_{T}^{k_2}|^p\right]\right]^{1\over p}+C_4\E\left[\sup_{t\in[0,T]}\E_t\left[\int_0^T |X_{s}^{k_1}-X_{s}^{k_2}|^pds\right]\right]^{1\over p}.
\end{eqnarray*}
In view of (\ref{eq_X_estimate1}), we obtain that for each $\delta>0$
\begin{eqnarray*}
&&\E\left[\int_0^T |X_{s}^{k_1}-X_{s}^{k_2}|^{p+\delta}ds\right]\\
&\leq& \int_0^T \E\left[|X_{s}^{k_1}-X_{s}^{k_2}|^{p+\delta}\right]ds
\leq C_{p,\delta}T \left\{w\left({1\over k_1\land k_2}\right)+\left({1\over k_1\land k_2}\right)\right\}^{p+\delta}.
\end{eqnarray*}
Then in view of Remark \ref{remark_doob_type_ineq}, we know $\{Y^k\}_{k=1}^\infty$ is a Cauchy sequence in $S_G^p(0,T)$ for any $p\geq 2$.
Thus we could conclude that $G$-BSDE (\ref{BSDE1}) has a solution $(Y,Z,K)\in\bigcap_{p\geq2}\mathfrak{G}_G^p(0,T)$.

\textbf{Step 4.} We now consider the situation that $\phi(\cdot)$  and $f(t,\cdot,y,z)$ can be locally Lipschitz. We assume $b,\sigma,f$, and $\phi$ satisfy assumptions \textbf{(A1)}-\textbf{(A4)} and assumption \textbf{(A7)}. For each
$(t,x,y,z)\in[0,T]\times\bb R^n\times\bb R\times\bb R$, we define
$$\phi^k(x):=\int_{\bb R^n}\phi(x-\tilde x)\rho_k(\tilde x)d\tilde x,$$
and
$$f^k(t,x,y,z):=\int_{\bb R^n}f(t,x-\tilde x,y,z)\rho_k(\tilde x)d\tilde x,$$
where $\rho_k$ is a positive smooth function such that its support is contained in a ${1\over k}$-ball in $\bb R^n $ and $\int_{\bb R^n}\rho_k=1$.
Noting that
$$\left|{\partial\phi^k\over\partial x}\right|\leq\int_{\bb R}\left|\phi(\tilde x){\partial\rho_k\over\partial x}(x-\tilde x)\right|d\tilde x
\leq C(k)M_0\norm{\partial\rho_k\over\partial x}_{\infty},$$
we obtain that $\phi^k$ is  Lipschitz in $x$. Similarly, it is easy to check $f^k(t,\cdot,y,z)$ is uniformly Lipschitz. Therefore, in view of the result in step 3, we obtain that the $G$-BSDE (\ref{BSDE1}) with the coefficients $(\phi^k,f^k)$ admits a solution $(Y^k,Z^k,K^k)\in\bigcap_{p\geq2}\mathfrak{G}_G^p(0,T)$.
It is easy to check that
\begin{eqnarray*}
&&|\phi^{k_1}(X_T)-\phi^{k_2}(X_T)|\\
&\leq &|\phi^{k_1}(X_T)-\phi(X_T)|+|\phi^{k_2}(X_T)-\phi(X_T)|\\
&\leq& \int_{\bb R}|\phi(X_T-\tilde x)-\phi(X_T)|\rho_{k_1}(\tilde x) d\tilde x
+\int_{\bb R}|\phi(X_T-\tilde x)-\phi(X_T)|\rho_{k_2}(\tilde x)d\tilde x\\
&\leq& 2L\left({1\over k_1}+{1\over k_2}\right)\left(1+|X_T|^m\right).
\end{eqnarray*}
Similarly, for each $t\in[0,T]$,
$$|(f^{k_1}-f^{k_2})(t,X_t,Y_t^{k_2},Z_t^{k_2})|\leq 2L\left({1\over k_1}+{1\over k_2}\right)\left(1+|X_t|^m\right).$$
In view of Proposition \ref{propose_estimate_sde}, we obtain for each $p\geq2$,
$$\E\left[|\phi^{k_1}(X_T)-\phi^{k_2}(X_T)|^p\right]\leq C\left({1\over k_1}+{1\over k_2}\right)^p\left(1+\E\left[|X_T|^{mp}\right]\right)
\leq C^\prime\left(1+|x_0|^{mp}\right)\left({1\over k_1}+{1\over k_2}\right)^p,$$
and
\begin{eqnarray*}
&&\E\left[\left(\int_0^T|(f^{k_1}-f^{k_2})(t,X_t,Y_t^{k_2},Z_t^{k_2})|dt\right)^p\right]\\
&\leq& C^{\prime\prime}\left(1+|x_0|^{mp}\right)\left({1\over k_1}+{1\over k_2}\right)^p,
\end{eqnarray*}
Then again in view of Remark \ref{remark_doob_type_ineq}, Proposition \ref{Propose_stable_pde} and Remark \ref{remark_stable_pde}, we know $\{Y^k\}_{k=1}^\infty$ is a Cauchy sequence in $S_G^p(0,T)$ for any $p\geq 2$.
Thus we could conclude that $G$-BSDE (\ref{BSDE1}) has a solution $(Y,Z,K)\in\bigcap_{p\geq2}\mathfrak{G}_G^p(0,T)$.

\textbf{Step 5.} Finally, we remove the boundedness condition on $b$ and $f$. We assume $b,\sigma,f$, and $\phi$ satisfy assumptions \textbf{(A1)}-\textbf{(A4)}. Set $b^k:=[b\lor (-k)]\land k$ and $f^k:=[f\lor (-k)]\land f$. It is easy to check that $b^k,\sigma,f^k$, and $\phi$ satisfy assumptions \textbf{(A1)}-\textbf{(A4)} and assumption \textbf{(A7)}. Let $(X^k,Y^k,Z^k,K^k)$ be the solution of the following $G$-FBSDE:
\begin{equation}
\begin{cases}\displaystyle
X_t^k=x_0+\int_0^tb^k(u, X_u^k)du+\int_0^t\sigma(u, X_u^k)dB_u,~~\text{q.s.},~t\in[0,T],\\
\displaystyle
Y_t^k=\phi(X_T^k)+\int_t^Tf^k(s,X_s^k,Y_s^k,Z_s^k)d\Bq_s-\int_t^TZ_s^kdB_s-\int_t^T dK_s^k,~~\text{q.s.},~t\in[0,T].
\end{cases}
\end{equation}
For each $k_1,k_2 \in \bb N$ and $t\in[0,T]$, set
$$\hat b_t:=b^{k_1}(t,X_t^{k_2})-b^{k_2}(t,X_t^{k_2}),$$
and
$$\hat f_t:=f^{k_1}(t,X_t^{k_1},Y_t^{k_2},Z_t^{k_2})-f^{k_2}(t,X_t^{k_2},Y_t^{k_2},Z_t^{k_2}).$$
Assume $k_1<k_2$, then for each $\delta>0$ and $t\in[0,T]$, we have
$$|\hat b_t|\leq |b(t,X_t^{k_2})|\textbf{1}_{\{|b(t,X_t^{k_2})|>k_1\}}\leq {1\over k_1^\delta}|b(t,X_t^{k_2})|^{1+\delta},$$
and
\begin{eqnarray*}
|\hat f_t|
&\leq& |f^{k_1}(t,X_t^{k_1},Y_t^{k_2},Z_t^{k_2})-f^{k_1}(t,X_t^{k_2},Y_t^{k_2},Z_t^{k_2})|\\
&& + |(f^{k_1}-f^{k_2})(t,X_t^{k_2},Y_t^{k_2},Z_t^{k_2})|\\
&\leq& L(1+|X_t^{k_1}|^m+|X_t^{k_2}|^m)|X_t^{k_1}-X_t^{k_2}|+{1\over k_1^\delta}|f(t,X_t^{k_2},Y_t^{k_2},Z_t^{k_2})|^{1+\delta}.
\end{eqnarray*}
In view of Proposition \ref{Propose_stable_sde}, we obtain for each $p\geq 2$ and $t\in[0,T]$,
\begin{eqnarray}
\nonumber \E \left[\left|X_{t}^{k_1}-X_t^{k_2}\right|^p\right]
 &\leq& C\E\left[\left(\int_0^{t}|\hat b_s|ds\right)^{p} \right]\\
\nonumber &\leq&{C_1\over k_1^{p}}+ {C_1\over k_1^{p\delta}}\E\left[\int_0^{t}|b(s,X_s^{k_2})|^{p(1+\delta)}ds \right]\\
 \nonumber &\leq&{C_1\over k_1^{p}}+ {C_2\over k_1^{p\delta}}\left(\lambda_T+\left[\int_0^{t}\E|X_s^{k_2}|^{p(1+\delta)}ds \right]
 \right)\\
\label{eq_x_hat} &\leq&{C_1\over k_1^{p}}+ {C_3\over k_1^{p\delta}}\left(1+\lambda_T+|x_0|^{p(1+\delta)} \right),
\end{eqnarray}
where $\lambda_T:=\int_0^{T}|b(s,0)|^{p(1+\delta)}ds$.
In view of Proposition \ref{propose_estimate_sde}, we obtain for each $p\geq2$,
\begin{eqnarray}
\nonumber\E\left[|\phi(X_T^{k_1})-\phi(X_T^{k_2})|^p\right]&\leq& C\E\left[\left(1+|X_T^{k_1}|^{mp}+|X_T^{k_2}|^{mp} \right)|X_T^{k_1}-X_T^{k_2}|^p\right]\\
\nonumber&\leq& C_1\E\left[1+|X_T^{k_1}|^{2mp}+|X_T^{k_2}|^{2mp} \right]^{1\over2}\E\left[|X_T^{k_1}-X_T^{k_2}|^{2p}\right]^{1\over2}\\
\label{eq_phi_hat}&\leq& C_2\left(1+|x_0|^{mp}\right)\E\left[|X_T^{k_1}-X_T^{k_2}|^{2p}\right]^{1\over2}.
\end{eqnarray}
On the other hand, for each $p\geq2$ and $0<\delta\leq 1$,
\begin{eqnarray*}
\E\left[\left(\int_0^T|\hat f_t|dt\right)^p\right]
&\leq& C\int_0^T\E\left[(1+|X_t^{k_1}|^{mp}+|X_t^{k_2}|^{mp})|X_t^{k_1}-X_t^{k_2}|^p\right]dt\\
&&+{C\over k_1^\delta}\E\left[\left(\int_0^T|f(t,X_t^{k_2},Y_t^{k_2},Z_t^{k_2})|^{1+\delta}dt\right)^p\right]\\
&\leq& C_1\left(1+|x_0|^{mp}\right)\int_0^T\E\left[|X_t^{k_1}-X_t^{k_2}|^{2p}\right]^{1\over2}dt\\
&&+{C_1\over k_1^\delta}\E\left[\left(\int_0^T|f(t,X_t^{k_2},0,0)|^{2}+|Y_t^{k_2}|^2+|Z_t^{k_2}|^2dt\right)^p\right].\\
\end{eqnarray*}
Note that $\cite[(3.2)]{Hu2018}$ still hold here, which means there exists a constant $C=C(M_0,L,G,T)$,
such that
$$\norm{Y^{k_2}}_{S_G^\infty}+\norm{Z^{k_2}}_{BMO_G}\leq C.$$
In view of Lemma \ref{energy_ineq} and \textbf{(A3)}, we obtain
\begin{equation}
\label{eq_f_hat}
\E\left[\left(\int_0^T|\hat f_t|dt\right)^p\right]
\leq C_1\left(1+|x_0|^{mp}\right)\int_0^T\E\left[|X_t^{k_1}-X_t^{k_2}|^{2p}\right]^{1\over2}dt\\+{C_2\over k_1^\delta}\\
\end{equation}
In view of Remark \ref{remark_doob_type_ineq}, inequalities (\ref{eq_x_hat}), (\ref{eq_phi_hat}) and (\ref{eq_f_hat}), Proposition \ref{Propose_stable_pde} and Remark \ref{remark_stable_pde}, we see that $\{Y^k\}_{k=1}^\infty$ is a Cauchy sequence in $S_G^p(0,T)$ for any $p\geq 2$.
Thus we could conclude that $G$-BSDE (\ref{BSDE1}) has a solution $(Y,Z,K)\in\bigcap_{p\geq2}\mathfrak{G}_G^p(0,T)$.
\end{proof}
Moreover, we have the following result with a similar argument before.
\begin{Theorem}
Assume that $\xi\in\bigcap\limits_{p\geq2}L_G^p(\Omega_t;\bb R^n)$  and $b,h,\sigma,g,f$, and $\phi$ satisfy assumptions \textbf{(A1)}-\textbf{(A4)} . Then $G$-BSDE (\ref{BSDE0}) has a unique solution $(Y,Z,K)\in\bigcap\limits_{p\geq2}\mathfrak{G}_G^p(0,T)$.
\end{Theorem}
\begin{Remark}
\label{remark_qbsde_bound}
Once we have a solution$(Y,Z,K)\in\bigcap\limits_{p\geq2}\mathfrak{G}_G^p(0,T)$ to $G$-BSDE (\ref{BSDE0}). In view of  $\cite[(3.2),(3.3)]{Hu2018}$,  there are two constants $C1=C1(M_0,L,G,T)$ and $C2=C2(p,M_0,L,G,T)$
such that for all $p\geq 2$,
$$\norm{Y}_{S_G^\infty}+\norm{Z}_{BMO_G}\leq C_1\quad\text{and}\quad\E[|K_T|^p]\leq C_2.$$
\end{Remark}

Now we can give the relationship between quadratic $G$-FBSDEs and parabolic PDEs. For  $(t,\xi)\in[0,T]\times\bigcap\limits_{p\geq2}L_G^p(\Omega_t;\bb R^n)$, denote by $(X^{t,\xi},Y^{t,\xi},Z^{t,\xi},K^{t,\xi})$ the solution to the $G$-FBSDE (\ref{SDE0})-(\ref{BSDE0}).
\begin{Proposition}
\label{Propose_u_estimate}
For each $t\in[0,T]$ and $\xi,\xi^{\prime}\in\bigcap\limits_{p\geq2}L_G^p(\Omega_t;\bb R^n)$, we have
\begin{eqnarray*}
&&|Y_t^{t,\xi}-Y_t^{t,\xi^\prime}|\leq C(1+|\xi|^m+|\xi^\prime|^m)|\xi-\xi^\prime|,~~\text{q.s.},\\
&&|Y_t^{t,\xi}|\leq C,~~\text{q.s.},
\end{eqnarray*}
where the constant $C$ depends on $L,G$ and $T$.
\end{Proposition}
\begin{proof}
For simplicity, we assume $g=0$ and $h=0$.
In view of Proposition \ref{Propose_stable_pde}, Remark \ref{remark_qbsde_bound}, and Proposition \ref{propose_estimate_sde}, we obtain
\begin{eqnarray*}
|Y_t^{t,\xi}-Y_t^{t,\xi^\prime}|&\leq &C\E_t\left[|\phi(X_{T}^{t,\xi})-\phi(X_{T}^{t,\xi^\prime})|^p\right]^{1\over p}\\
&&+C\E_t\left[\left(\int_t^T |f(s,X_{s}^{t,\xi},Y_s^{t,\xi},Z_s^{t,\xi})-f(s,X_{s}^{t,\xi^\prime},Y_s^{t,\xi},Z_s^{t,\xi})|d\Bq_s\right)^p\right]^{1\over p}\\
&\leq &C_1\E_t\left[(1+|X_{T}^{t,\xi}|^{mp}+|X_{T}^{t,\xi^\prime}|^{mp})|X_{T}^{t,\xi}-X_{T}^{t,\xi^\prime}|^p\right]^{1\over p}\\
&&+C_1\E_t\left[\int_t^T (1+|X_{s}^{t,\xi}|^{mp}+|X_{s}^{t,\xi^\prime}|^{mp})|X_{s}^{t,\xi}-X_{s}^{t,\xi^\prime}|^pds\right]^{1\over p}\\
&\leq &C_2(1+|\xi|^m+|\xi^\prime|^m)\E_t\left[|X_{T}^{t,\xi}-X_{T}^{t,\xi^\prime}|^{2p}\right]^{1\over 2p}\\
&&+C_2(1+|\xi|^m+|\xi^\prime|^m)\left(\int_t^T \E_t\left[|X_{s}^{t,\xi}-X_{s}^{t,\xi^\prime}|^{2p}ds\right]^{1\over 2}\right)^{1\over p}\\
&\leq &C_2(1+|\xi|^m+|\xi^\prime|^m)|\xi-\xi^\prime|~~\text{q.s.}
\end{eqnarray*}
On the other hand, we get $|Y_t^{t,\xi}|\leq C$, q.s., directly from Remark \ref{remark_qbsde_bound}.
\end{proof}
Now for each $(t,x)\in[0,T]\times\bb R^n$, we define $u(t,x):=Y_t^{t,x}$. Identically as  in \cite[Reamrk 4.3]{HuM2014_2}, we deduce that $u$ is a deterministic function. In view of Proposition \ref{Propose_u_estimate}, we immediately have the following estimates:
\begin{eqnarray*}
&&|u(t,x)-u(t,x^\prime)|\leq C(1+|x|^m+|x^\prime|^m)|x-x^\prime|,\\
&&|u(t,x)|\leq C,
\end{eqnarray*}
where the constant $C$ depends on $L,G$ and $T$. Moreover, with the same proof \cite[Theorem 4.4]{HuM2014_2}, we have the following Proposition.
\begin{Proposition}
\label{Propose_markov_fbsde}
For each $\xi\in\bigcap\limits_{p\geq2}L_G^p(\Omega_t;\bb R^n)$, we have $u(t,\xi)=Y_t^{t,\xi}$.
\end{Proposition}

Now we give the main result of this section.
\begin{Theorem}
\label{Thm_qbsde_pde}
Let $u(t,x):=Y_t^{t,x}$ for $(t,x)\in [0,T]\times \bb R^n$. Then, the function  $u$ is a viscosity solution to the PDE \eqref{PDE_thm}.
\end{Theorem}
\begin{proof}
 Without loss of genearlity, we still assume that $h=0$ and $g=0$. First, we show $u$ is a continuous function. Fix some $(t,x)\in[0,T]\times\bb R^n$.
 In view of $Y_{t+\delta}^{t,x}=Y_{t+\delta}^{t+\delta,X_{t+\delta}^{t,x}}$ and Proposition \ref{Propose_markov_fbsde}, we obtain $Y_{t+\delta}^{t,x}=u(t+\delta,X_{t+\delta}^{t,x})$ for $\delta\in[0,T-t]$.
Thus we obtain
$$
 Y_t^{t,x}=u(t+\delta,X_{t+\delta}^{t,x})+\int_t^{t+\delta}f(s, X_s^{t,x},Y_s^{t,x},Z_s^{t,x})d\Bq_s
 -\int_t^{t+\delta}Z_s^{t,x}dB_s-\int_t^{t+\delta}dK_s^{t,x}~~\text{q.s.}
$$
The generator can be written  as  in Proposition \ref{Propose_Y_bound} the following form:  for each $s\in[0,T]$,
$$f(s,X_s^{t,x},Y_s^{t,x},Z_s^{t,x})=f(s,X_s^{t,x},0,0)+m_s^\varepsilon+a_s^\varepsilon Y_s^{t,x}+b_s^\varepsilon Z_s^{t,x},$$
with
\begin{eqnarray*}
a_s^\varepsilon&:=&(1-l(Y_s^{t,x})){f(s,X_s^{t,x},Y_s^{t,x},Z_s^{t,x})-f(s,X_s^{t,x},0,Z_s^{t,x})\over Y_s^{t,x}}\textbf{1}_{\{|Y_s^{t,x}|>0\}},\\
b_s^\varepsilon&:=&(1-l(Z_s^{t,x})){f(s,X_s^{t,x},0,Z_s^{t,x})-f(s,X_s^{t,x},0,0)\over |Z_s^{t,x}|^2}Z_s^{t,x}\textbf{1}_{\{|Z_s^{t,x}|>0\}},\\
m_s^\varepsilon&:=& l(Y_s^{t,x})(f(s,X_s^{t,x},Y_s^{t,x},Z_s^{t,x})-f(s,X_s^{t,x},0,Z_s^{t,x}))\\
&&+l(Z_s^{t,x})(f(s,X_s^{t,x},0,Z_s^{t,x})-f(s,X_s^{t,x},0,0))
\end{eqnarray*}
for a  Lipschitz continuous function $l$  such that $\textbf{1}_{[-\varepsilon,\varepsilon]}(x)\leq l(x)\leq \textbf{1}_{[-2\varepsilon,2\varepsilon]}(x)$ at each $x\in (-\infty, +\infty)$. Moreover,
\begin{eqnarray*}
|a_s^\varepsilon|&\leq& L,\quad |b_s^\varepsilon|\leq L(1+|Z_s^{t,x}|),\quad
|m_s^\varepsilon|\leq4L\varepsilon(1+ \varepsilon).
\end{eqnarray*}
In view of \cite[Lemma 3.6]{Hu2018} and Remark \ref{remark_qbsde_bound}, we know that $b^\varepsilon\in BMO_G$. Set $\tilde{B}_t:=B_t-\int_0^tb_s^\varepsilon d\Bq_s$ for $t\in [0,T]$. Thus we can define a new $G$-expectation $\Et[\cdot]$ by $\Exp(b^\varepsilon)$, such that $\tilde{B}$ is a $G$-Brownian motion under $\Et[\cdot]$.
Thus we have
$$
 Y_t^{t,x}=u(t+\delta,X_{t+\delta}^{t,x})+\int_t^{t+\delta}\left(f_s+m_s^\varepsilon+a_s^\varepsilon Y_s^{t,x}\right)d\Bq_s
 -\int_t^{t+\delta}Z_s^{t,x}d\tilde{B}_s-\int_t^{t+\delta}dK_s^{t,x},~~\text{q.s.},
$$
where $f_s:=f(s,X_s^{t,x},0,0)$.
Taking $G$-expectation $\Et[\cdot]$, we get
$$u(t,x)=\Et\left[u(t+\delta,X_{t+\delta}^{t,x})+\int_t^{t+\delta}\left(f_s+m_s^\varepsilon+a_s^\varepsilon Y_s^{t,x}\right)d\Bq_s\right].$$
In view of Proposition \ref{Propose_u_estimate}, we obtain
\begin{eqnarray*}
|u(t,x)-u(t+\delta,x)|&\leq& \Et\left[|u(t+\delta,X_{t+\delta}^{t,x})-u(t+\delta,x)|
+\int_t^{t+\delta}\left|f_s+m_s^\varepsilon+a_s^\varepsilon Y_s^{t,x}\right|d\Bq_s\right]\\
&\leq& \Et\left[(1+|x|^m+|X_{t+\delta}^{t,x}|^m)|X_{t+\delta}^{t,x}-x|\right]\\
&&+\sqrt{\delta}\siup^2\Et\left[\int_t^{t+\delta}\left|f_s+m_s^\varepsilon+a_s^\varepsilon Y_s^{t,x}\right|^2ds\right]^{1\over 2}.
\end{eqnarray*}
Let $\varepsilon<1$. In view of Remark \ref{remark_qbsde_bound} and \textbf{(A3)}, there exists a constant $C$ depending on $M_0,L,G$ and $T$, such that
$$\siup^2\Et\left[\int_t^{t+\delta}\left|f_s+m_s^\varepsilon+a_s^\varepsilon Y_s^{t,x}\right|^2ds\right]^{1\over 2}\leq C.$$
Thus we know
\begin{equation}
\label{eq_u_t}
|u(t,x)-u(t+\delta,x)|\leq \Et\left[(1+|x|^m+|X_{t+\delta}^{t,x}|^m)^2\right]^{1\over2}\Et\left[|X_{t+\delta}^{t,x}-x|^2\right]^{1\over2}+C\sqrt{\delta}.
\end{equation}
Note that $\norm{ b_s^\varepsilon}_{BMO_G}\leq\norm{L(1+|Z^{t,x}|)}_{BMO_G}$. Then according to Lemma \ref{reverse holder} and Remark \ref{remark_qbsde_bound}, there exists $p>1$ depending on $M_0,L,G$ and $T$, such that, for each $(s,X)\in[0,T]\times L_G^p(\Omega_T)$
$$\Et_s[|X|]=\E_s\Big[{\Exp(b^\varepsilon)_T\over \Exp( b^\varepsilon)_s}|X|\Big]\leq \E_s\Big[\Big({\Exp(t b^\varepsilon)_T\over \Exp( b^\varepsilon)_s}\Big)^{p^\prime}\Big]^{1\over p^\prime}\E_s[|X|^p]^{1\over p}\leq C_p\E_s[|X|^p]^{1\over p},~~\text{q.s.},$$
where $C_p$ depending only on $p$.
In view of Proposition \ref{propose_estimate_sde}, we get
\begin{eqnarray*}
\Et\left[|X_{t+\delta}^{t,x}|^{2m}\right]\leq  C_p\E\left[|X_{t+\delta}^{t,x}|^{2mp}\right]^{1\over p}\leq C(1+|x|^{2m}),
\end{eqnarray*}
and
$$ \Et\left[|X_{t+\delta}^{t,x}-x|^2\right]\leq C(1+|x|^2)\delta.$$
Then from (\ref{eq_u_t}), we have for each $(t,x)\in[0,T]\times\bb R^n$,
$$|u(t,x)-u(t+\delta,x)|\leq C(1+x^{m+1})\sqrt\delta,$$
where $C$ depends on $M_0,L,G$ and $T$.
On the other hand, we get from Proposition \ref{Propose_u_estimate} that
for each $(t,x,x^\prime)\in[0,T]\times\bb R^n\times \bb R^n$,
$$|u(t,x)-u(t,x^\prime)|\leq C(1+|x|^m+|x^\prime|^m)|x-x^\prime|.$$
It follows that $u$ is continuous.

For any fixed $(t_0,x_0)\in (0,T)\times \bb R^n$, let $\psi\in C^{1,2}([0,T]\times \bb R^n)$ such that for each $(t,x)\in \left([0,T]\times\bb R^n\right)\backslash \{(t_0,x_0)\}$
\begin{equation}
\label{strict_minimum_property}
\psi(t,x)-u(t,x)>\psi(t_0,x_0)-u(t_0,x_0)=0.
\end{equation}
Without loss of generality, we may assume that there exists some $m_1>0$ such that for each $(t,x)\in[0,T]\times\bb R^n$,
\begin{equation}
\label{ineq_psi_bound}
|\psi(t,x)|+|D_x^2\psi(t,x)|\leq C(1+|x|^{m_1})\quad\text{and}\quad|D_x\psi(t,x)|\leq C.
\end{equation}
We want to prove that
$$\partial_t \psi+G\left(\sigma^\mathrm{T}(t_0,x_0) D_x^2\psi \sigma(t_0,x_0)+2f(t_0,x_0,\psi,\sigma^\mathrm{T}(t_0,x_0)D_x\psi)\right)+b^\mathrm{T}(t_0,x_0)D_x\psi\geq0.$$
Let us assume the inequality before does not hold.
Let $O_{\delta}(t_0,x_0)$ be a open ball centered at $(t_0,x_0)$, with radius $\delta$.
By continuity, there exists some $\delta\in(0,T-t_0)$ such that
for each $(t,x)\in O_{\delta}(t_0,x_0)$,
$$\partial_t \psi+G\left(\sigma^\mathrm{T}(t,x) D_x^2\psi \sigma(t,x)+2f(t,x,\psi,\sigma^\mathrm{T}(t,x)D_x\psi)\right)+b^\mathrm{T}(t,x)D_x\psi\leq 0.$$
Setting $\tilde Y_t:=\psi(t,X_t^{t_0,x_0})$ and $\tilde Z_t:=\sigma^\mathrm{T}(t,X_t^{t_0,x_0})D_x\psi(t,X_t^{t_0,x_0})$, it is easy to check that for
each $t\in[t_0,T]$,
\begin{eqnarray*}
 \tilde Y_{t}&=&\psi(T,X_{T}^{t_0,x_0})-\int_{t}^{T}\left\{\partial_t\psi(s,X_s^{t_0,x_0})+b^\mathrm{T}(s,X_s^{t_0,x_0})D_x\psi(s,X_s^{t_0,x_0})\right\}ds\\
 &&-\int_{t}^{T}{1\over 2}\sigma^\mathrm{T}(s,X_s^{t_0,x_0}) D_x^2\psi(s,X_s^{t_0,x_0}) \sigma(s,X_s^{t_0,x_0})d\Bq_s-\int_{t}^{T} \tilde Z_sdB_s,
~~\text{q.s.}
\end{eqnarray*}
For each $t\in[t_0,T]$, set $\tilde K_t:=\int_{t_0}^tF_sd\Bq_s-\int_{t_0}^tG(2F_s)ds $, where
$$F_s:={1\over 2}\sigma^\mathrm{T}(s,X_s^{t_0,x_0}) D_x^2\psi(s,X_s^{t_0,x_0}) \sigma(s,X_s^{t_0,x_0})+f(s,X_s^{t_0,x_0},\tilde Y_s,\tilde Z_s),~~s\in[t_0,T].$$
We can check that $\tilde K$ is a decreasing $G$-martingale.
Noting that $\sigma$ is bounded here, by Proposition \ref{propose_estimate_sde} and inequality \eqref{ineq_psi_bound}, we deduce that
$\tilde K_T\in\bigcap\limits_{p\geq1}L_G^p(\Omega_T)$.
Now we have for each $t\in[t_0,T]$,
\begin{eqnarray*}
 \tilde Y_{t}&=&\psi(T,X_{T}^{t_0,x_0})-\int_{t}^{T}\left\{\partial_t\psi(s,X_s^{t_0,x_0})+b^\mathrm{T}(s,X_s^{t_0,x_0})D_x\psi(s,X_s^{t_0,x_0})+G(2F_s)\right\}ds\\
 &&+\int_{t}^{T}f(s,X_s^{t_0,x_0},\tilde Y_s,\tilde Z_s)d\Bq_s-\int_{t}^{T} \tilde Z_sdB_s-\int_{t}^{T} d\tilde K_s,
~~\text{q.s.}
\end{eqnarray*}
Now we set $\delta Y:=\tilde Y-Y^{t_0,x_0}$, $\delta Z:=\tilde Z-Z^{t_0,x_0}$ and for each $s\in [t_0,T]$,
$$\tilde F_s:=\partial_t\psi(s,X_s^{t_0,x_0})+b^\mathrm{T}(s,X_s^{t_0,x_0})D_x\psi(s,X_s^{t_0,x_0})+G(2F_s).$$
Then for each $t\in[t_0,T]$,
\begin{eqnarray}
 \nonumber \delta Y_{t}&=&(\psi-u)(T,X_{T}^{t_0,x_0})
 +\int_{t}^{T}f(s,X_s^{t_0,x_0},\tilde Y_s,\tilde Z_s)-f(s,X_s^{t_0,x_0},Y_s^{t_0,x_0},\tilde Z_s^{t_0,x_0})d\Bq_s\\
 &&-\int_{t}^{T}\tilde F_sds-\int_{t}^{T} \delta Z_sdB_s-\int_{t}^{T} d\tilde K_s+\int_{t}^{T} dK_s^{t_0,x_0},
~~\text{q.s.}
\label{eq_delta_y1}
\end{eqnarray}
As what we do in Proposition \ref{Propose_stable}, we have for each $s\in[t_0,T]$,
$$f(s,X_s^{t_0,x_0},\tilde Y_s,\tilde Z_s)-f(s,X_s^{t_0,x_0},Y_s^{t_0,x_0},\tilde Z_s^{t_0,x_0})=m_s^\varepsilon+a_s^\varepsilon \delta Y_s+b_s^\varepsilon \delta Z_s,$$
where
\begin{eqnarray*}
|a_s^\varepsilon|\leq L,\quad |b_s^\varepsilon|\leq L(1+|Z_s^{t_0,x_0}|+|\tilde Z_s|),\quad
| m_s^\varepsilon|\leq 4L\varepsilon(1+ \varepsilon+|\tilde Z_s|).
\end{eqnarray*}
In view of \cite[Lemma 3.6]{Hu2018}, Remark \ref{remark_qbsde_bound}, and inequality \eqref{ineq_psi_bound}, we know that $b^\varepsilon\in BMO_G$.
Set $\tilde{B}_t:=B_t-\int_0^tb_s^\varepsilon d\Bq_s$ for $t\in [0,T]$. Thus we can define a new $G$-expectation $\Et[\cdot]$ by $\Exp(b^\varepsilon)$, such that $\tilde{B}$ is a $G$-Brownian motion under $\Et[\cdot]$.
Thus the equality \eqref{eq_delta_y1} can be written as
\begin{eqnarray}
 \nonumber \delta Y_{t}&=&(\psi-u)(T,X_{T}^{t_0,x_0})
 +\int_{t}^{T}m_s^\varepsilon+a_s^\varepsilon \delta Y_sd\Bq_s\\
 \nonumber &&-\int_{t}^{T}\tilde F_sds-\int_{t}^{T} \delta Z_sd\tilde B_s-\int_{t}^{T} d\tilde K_s+\int_{t}^{T} dK_s^{t_0,x_0},
~~\text{q.s.}
\end{eqnarray}
	Applying It\^o's formula to $e^{\int_0^ta_s^\varepsilon d\Bq_s}\delta Y_t$,~ we have
	\begin{eqnarray}
\nonumber		&&e^{\int_0^t a_s^\varepsilon d\Bq_s}\delta Y_t\\
\nonumber		&=& e^{\int_0^T a_s^\varepsilon d\Bq_s}(\psi-u)(T,X_{T}^{t_0,x_0})-\int_t^T e^{\int_0^s a_u^\varepsilon d\Bq_u} \tilde F_s d\Bq_s+\int_t^T e^{\int_0^s a_u^\varepsilon d\Bq_u}  m_s^\varepsilon d\Bq_s\\
		&&-\int_t^T e^{\int_0^s a_u^\varepsilon d\Bq_u} \delta Z_s d\tilde{B}_s- \int_t^T e^{\int_0^sa_u^\varepsilon d\Bq_u} d\tilde  K_s+\int_t^T e^{\int_0^s a_u^\varepsilon d\Bq_u} dK_s^{t_0,x_0},~~\text{q.s.}
\label{eq_delta_y2}
	\end{eqnarray}
Let $\mathcal P$ be the weakly compact set that represents $\E$. For each $\bb P\in \mathcal P$, let $\tau^\bb P$ be the following stopping time under $\bb P$:
$$\tau^\bb P:=\inf\{s\geq t_0:(s,X_s^{t_0,x_0})\notin O_\delta(t_0,x_0)\}.$$
By the strict minimum property \eqref{strict_minimum_property}, we notice that
$$\eta:=\min_{(t,x)\in\partial O_\delta(t_0,x_0)}(\psi-u)(t,x)>0.$$
It is easy to check that $\tau^\bb P<T$, $\bb P\text{\,-a.s.}$ and $(\psi-u)(\tau^\bb P,X_{\tau^\bb P}^{t_0,x_0})\geq \eta$, $\bb P\text{\,-a.s.}$.
From equality \eqref{eq_delta_y2}, we have for each $t\in[t_0,T]$,
	\begin{eqnarray}
\nonumber		&&e^{\int_0^{t\land \tau^\bb P} a_s^\varepsilon d\Bq_s}\delta Y_{t\land \tau^\bb P}\\
\nonumber		&=& e^{\int_0^{\tau^\bb P} a_s^\varepsilon d\Bq_s}(\psi-u)(\tau^\bb P,X_{\tau^\bb P}^{t_0,x_0})-\int_{t\land \tau^\bb P}^{\tau^\bb P} e^{\int_0^s a_u^\varepsilon d\Bq_u} \tilde F_s d\Bq_s+\int_{t\land \tau^\bb P}^{\tau^\bb P}  e^{\int_0^s a_u^\varepsilon d\Bq_u}  m_s^\varepsilon d\Bq_s\\
\nonumber		&&-\int_{t\land \tau^\bb P}^{\tau^\bb P} e^{\int_0^s a_u^\varepsilon d\Bq_u} \delta Z_s d\tilde{B}_s- \int_{t\land \tau^\bb P}^{\tau^\bb P} e^{\int_0^sa_u^\varepsilon d\Bq_u} d\tilde  K_s+\int_{t\land \tau^\bb P}^{\tau^\bb P} e^{\int_0^s a_u^\varepsilon d\Bq_u} dK_s^{t_0,x_0},~\bb P\text{\,-a.s.}.
	\end{eqnarray}
Note that for each $(s,\omega)\in [t_0,T]\times\Omega_T$ satisfying $t_0\leq s\leq \tau^\bb P(\omega)$, $\tilde F_s\leq0$.
Thus we have
	\begin{eqnarray}
\nonumber		e^{\int_0^{t\land \tau^\bb P} a_s^\varepsilon d\Bq_s}\delta Y_{t\land \tau^\bb P}
	           &\geq& e^{-LT\siup^2} \eta+\int_{t\land \tau^\bb P}^{\tau^\bb P}  e^{\int_0^s a_u^\varepsilon d\Bq_u}  m_s^\varepsilon d\Bq_s\\
\nonumber		&&-\int_{t\land \tau^\bb P}^{\tau^\bb P} e^{\int_0^s a_u^\varepsilon d\Bq_u} \delta Z_s d\tilde{B}_s+\int_{t\land \tau^\bb P}^{\tau^\bb P} e^{\int_0^s a_u^\varepsilon d\Bq_u} dK_s^{t_0,x_0},~\bb P\text{\,-a.s.}.
	\end{eqnarray}
Since $\tilde B$ is a martingale under the new probability $\bb Q$ with $d\bb Q:=\Exp(b^\varepsilon)_Td\bb P$,
we have in particular
	\begin{eqnarray}
\nonumber		E^{\bb Q}\left[\delta Y_{t_0}\right]
	           \geq e^{-2LT\siup^2} \eta- e^{2LT\siup^2}E^{\bb Q}\left[\int_{t_0}^{\tau^\bb P} | m_s^\varepsilon| d\Bq_s\right]
+ e^{2LT\siup^2}E^{\bb Q}\left[\int_{t\land \tau^\bb P}^{\tau^\bb P} dK_s^{t_0,x_0}\right].
	\end{eqnarray}
While $\delta Y_{t_0}=(\psi-u)(t_0,x_0)$ and $| m_s^\varepsilon|\leq \rho(\varepsilon)$ for a nonnegative continuous function $\rho$ defined on $\bb R^{+}$ with $\rho(0)=0$, we have
\begin{eqnarray*}
(\psi-u)(t_0,x_0)&\geq &
e^{-2LT\siup^2} \eta-e^{2LT\siup^2}T\siup^2\rho(\varepsilon)+ e^{2LT\siup^2}E^{\bb Q}\left[\int_{t\land \tau^\bb P}^{\tau^\bb P} dK_s^{t_0,x_0}\right]\\
&\geq& e^{-2LT\siup^2} \eta-e^{2LT\siup^2}T\siup^2\rho(\varepsilon)+ e^{2LT\siup^2}E^{\bb Q}\left[K_T^{t_0,x_0}\right] \\
&=& e^{-2LT\siup^2} \eta-e^{2LT\siup^2}T\siup^2\rho(\varepsilon)+ e^{2LT\siup^2}E^{\bb P}\left[\Exp(b^\varepsilon)_TK_T^{t_0,x_0}\right]
\end{eqnarray*}
 for each $\bb P\in\mathcal{P}$. Consequently, we have
\begin{eqnarray}
\nonumber (\psi-u)(t_0,x_0)&\geq &
 e^{-2LT\siup^2} \eta-e^{2LT\siup^2}T\siup^2\rho(\varepsilon)+ e^{2LT\siup^2}\sup_{\bb P\in\mathcal{P}}E^{\bb P}\left[\Exp(b^\varepsilon)_TK_T^{t_0,x_0}\right]\\
 &=&e^{-2LT\siup^2} \eta-e^{2LT\siup^2}T\siup^2\rho(\varepsilon)+ e^{2LT\siup^2}\E\left[\Exp(b^\varepsilon)_TK_T^{t_0,x_0}\right].
 \label{ineq_subsolution}
\end{eqnarray}
In view of Lemma \ref{lemma_Gisr_K} and Remark \ref{remark_qbsde_bound}, the process  $K^{t_0,x_0}$ is a $G$-martingale under $\Et[\cdot]$, and
$$\E\left[\Exp(b^\varepsilon)_TK_T^{t_0,x_0}\right]=\Et\left[K_T^{t_0,x_0}\right]=0.$$
Letting $\varepsilon\to 0$ in the last inequality, we have
$$0=(\psi-u)(t_0,x_0)\geq e^{-2LT\siup^2} \eta>0,$$
which is a contradiction. Hence,  $u$ is a viscosity subsolution.

In a similar way,  $u$ can be shown to be  a viscosity supersolution.
\end{proof}
\begin{Remark}
When the functions $f$ and $g$ do not depend on $y$, one can get the uniqueness of viscosity solution to PDE by the uniqueness
result in Da Lio and Ley \cite{DaLio2006} concerning Bellman-Isaacs equation.
\end{Remark}

\subsection{Relation between reflected quadratic $G$-BSDEs and obstacle problems for nonlinear parabolic PDEs}

With the preceding nonlinear Feynman-Kac formula, we can give the relationship between solutions of the obstacle problem for nonlinear parabolic PDEs and the related reflected quadratic $G$-BSDEs. For each $(t,\xi)\in[0,T]\times \bigcap\limits_{p\geq2}L_G^p(\Omega_t;\bb R^n)$, we consider the following $G$-SDE:
\begin{equation}
\label{SDE2}
X_s^{t,\xi}=\xi+\int_t^sb(u, X_u^{t,\xi})du+\int_t^sh(u, X_u^{t,\xi})d\Bq_u+\int_t^s\sigma(u, X_u^{t,\xi})dB_u,~s\in[t,T],
\end{equation}
and the following type of reflected $G$-BSDE:
\begin{equation}
\begin{cases}
\label{RBSDE2}
\displaystyle
Y_s^{t,\xi}=\phi(X_T^{t,\xi})+\int_s^Tg(u, X_u^{t,\xi},Y_u^{t,\xi},Z_u^{t,\xi})du+\int_s^Tf(u, X_u^{t,\xi},Y_u^{t,\xi},Z_u^{t,\xi})d\Bq_u\\
\displaystyle
\qquad \quad-\int_s^TZ_u^{t,\xi}dB_u+\int_s^TdA_u^{t,\xi},~~\text{q.s.},~s\in[t,T];\\[3mm]
\displaystyle
Y_s^{t,\xi}\geq l(s,X_s^{t,\xi}),~~\text{q.s.}, ~s\in[t,T];\\
\displaystyle
\int_t^\cdot(l(u,X_u^{t,\xi})-Y_u^{t,\xi})dA_u^{t,\xi} \text{ is a non-increasing }G\text{-martingale on $[s,T]$.}
\end{cases}
\end{equation}
where $b,h,\sigma,l: [0,T]\times\bb R^n\to\bb R$, $\phi:\bb R^n\to\bb R$, $f,g:[0,T]\times\bb R^n\times \bb R\times
\bb R\to\bb R$ are deterministic functions and satisfy \textbf{(A1)}-\textbf{(A4)}. Moreover, we have the following assumption on $l$:
\begin{itemize}
\item[\textbf{(A9)}]The function $l(t,\cdot)$ is uniformly Lipschitz and $l(T,x)\leq\phi(x)$ for any $x\in\bb R^n$. Furthermore,  there exists a constant $N_0$ such that $ l(t,x)\leq N_0,$  for any $t\in[0,T]$.
\item[\textbf{(A10)}]The function $l(\cdot,x)$ is uniformly continuous, i.e. there is a non-decreasing continuous function $w: [0, +\infty)\to [0, +\infty)$ such that $w(0)=0$ and
    $$\sup_{x \in\bb{R}^n}|l(t,x)-l(t^\prime,x)|\leq w(|t-t^\prime|).$$
\end{itemize}

\begin{Remark}
In the Markovian case, Assumptions \textbf{(H2)} and \textbf{(H5)} may not hold directly. However, in view of Remark \ref{Remark_tech_condition}, one can still get the results under Assumptions  \textbf{(H1)}, \textbf{(H3)} and \textbf{(H4)} as long as the penalized quadratic $G$-BSDE has a solution. The reflected $G$-BSDE (\ref{RBSDE2}) has one solution in the sense of Definition \ref{def_solution} and all results in Sections 3-5 still hold here under Assumptions \textbf{(A1)}-\textbf{(A4)} and \textbf{(A9)}-\textbf{(A10)}.
\end{Remark}

Consider the following obstacle problem for a parabolic PDE:
\begin{equation}
\begin{cases}
\label{PDE_obstacle}
\min\left\{-\partial_t u-F(D_x^2u,D_xu,u,x,t),u(t,x)-l(t,x)\right\}=0,~~(t,x)\in[0,T)\times\bb R^n\\
u(T,x)=\phi(x),~~x\in\bb R^n,
\end{cases}
\end{equation}
where
\begin{eqnarray*}
F(A,p,y,x,t)
&:=&G\left(\sigma^\mathrm{T}(t,x) A \sigma(t,x)+2f(t,x,y,\sigma^\mathrm{T}(t,x)p)+2h^\mathrm{T}(t,x)p\right)\\
&&+b^\mathrm{T}(t,x)p+g(t,x,y,\sigma^\mathrm{T}(t,x)p),
\end{eqnarray*}
for each $(A,p,y,x,t)\in \bb S_n\times \bb R^n\times \bb R\times\bb R^n\times[0,T] $.

We need to recall the equivalent definition of the viscosity solution of the obstacle problem (\ref{PDE_obstacle}) as in \cite{Li2017_1} or \cite{Peng2010}.
\begin{Definition}
Let $u\in C([0,T]\times \bb R^n)$ and $(t,x)\in[0,T]\times \bb R^n$. We denote by $\mathcal P^{2,+}u(t,x)$ the set of triples $(p,q,A)\in\bb R\times\bb R^n\times \bb S_n$ satisfying
$$u(s,y)\leq u(t,y)+p(s-t)+q^\mathrm{T}(y-x)+{1\over 2}A(y-x)^2+o(|s-t|+|y-x|^2).$$
Similarly, we define $\mathcal P^{2,-}u(t,x):=-\mathcal P^{2,+}(-u)(t,x)$.
\end{Definition}
\begin{Definition}
It can be said that $u\in C([0,T]\times \bb R^n)$ is a viscosity subsolution of (\ref{PDE_obstacle}) if $u(T,x)\leq\phi(x),x\in\bb R^n$, and for each $(t,x)\in[0,T)\times \bb R^n$ and $(p,q,A)\in\mathcal P^{2,+}u(t,x)$,
$$\min\left\{-p-F(A,q,u(t,x),x,t),u(t,x)-l(t,x)\right\}\leq 0.$$
It can be said that $u\in C([0,T]\times \bb R^n)$ is a viscosity supersolution of (\ref{PDE_obstacle}) if $u(T,x)\geq\phi(x),x\in\bb R^n$, and for each $(t,x)\in[0,T]\times \bb R^n$ and $(p,q,X)\in\mathcal P^{2,-}u(t,x)$,
$$\min\left\{-p-F(A,q,u(t,x),x,t),u(t,x)-l(t,x)\right\}\geq 0.$$
$u\in C([0,T]\times \bb R^n)$ is said to be a viscosity solution of (\ref{PDE_obstacle}) if it is both a viscosity subsolution and supersolution.
\end{Definition}
We now define $u(t,x):=Y_t^{t,x}$. Similarly as before, we can note that $u$ is a deterministic function. We now should prove that
$u\in C([0,T]\times \bb R^n)$.
\begin{Lemma}
Let assumptions \textbf{(A1)}-\textbf{(A4)} and \textbf{(A9)}-\textbf{(A10)} hold. For each $t\in[0,T]$, $x_1,x_2\in\bb R^n$, we have
$$|u(t,x_1)-u(t,x_2)|^2\leq C(1+|x_1|^{2m}+|x_2|^{2m})|x_1-x_2|^2+C|x_1-x_2|$$
\end{Lemma}
\begin{proof}
Without loss of generality, we assume $h=0$ and $g=0$.
In view of Proposition \ref{Propose_Z_bound} and Proposition \ref{Propose_Y_bound}, we deduce that there exist a constant $C_1:=C_1(T,L,G,M_0,N_0)$ such that
$$\norm{Z^{t,x_1}}_{BMO_G}+\norm{Z^{t,x_2}}_{BMO_G}\leq C_1,$$
and a constant $C_2:=C_2(T,L,G,M_0,N_0,\alpha)$,
for any $\alpha\geq 1$, such that $$\E[|A_T^{t,x_1}|^\alpha+|A_T^{t,x_2}|^\alpha]\leq C_2.$$
In view of Proposition \ref{Propose_stable} and its proof and noting that $u$ is deterministic, we obtain that there exist a constant $~C:=C(T,L,G,M_0,N_0)~$ and $p\geq 2$ such that for each $t\in[0,T]$,
\begin{eqnarray*}
&&|u(t,x_1)-u(t,x_2)|^2\\
&\leq& C\left\{\E\left[|\phi(X_T^{t,x_1})-\phi(X_T^{t,x_2})|^{2p}\right]^{1\over p}
+\E\left[\sup_{s\in [t,T]}|l(s,X_s^{t,x_1})-l(s,X_s^{t,x_2})|^{2p}\right]^{1\over 2p}\right\}\\
&&+C\E\left[\left(\int_t^T |f(s, X_s^{t,x_1},Y_s^{t,x_2},Z_s^{t,x_2})-f(s, X_s^{t,x_2},Y_s^{t,x_2},Z_s^{t,x_2})|^2ds\right)^p\right]^{1\over p}\\
&\leq& C^\prime\E\left[(1+|X_T^{t,x_1}|^{2pm}+|X_T^{t,x_2}|^{2pm})^2\right]^{1\over 2p}\E\left[|X_T^{t,x_1}-X_T^{t,x_2}|^{4p}\right]^{1\over 2p}\\
&&+C^\prime\E\left[\sup_{s\in [t,T]}|X_s^{t,x_1}-X_s^{t,x_2}|^{2p}\right]^{1\over 2p}\\
&&+C^\prime\E\left[\int_t^T (1+|X_s^{t,x_1}|^{2pm}+|X_s^{t,x_2}|^{2pm})^2ds\right]^{1\over 2p}\E\left[\int_t^T|X_s^{t,x_1}-X_s^{t,x_2}|^{4p}ds\right]^{1\over 2p}\\
&\leq& C^{\prime\prime}(1+|x_1|^{2m}+|x_2|^{2m})|x_1-x_2|^2+C^{\prime\prime}|x_1-x_2|.
\end{eqnarray*}
\end{proof}
\begin{Lemma}
Let assumptions \textbf{(A1)}-\textbf{(A4)} and \textbf{(A9)}-\textbf{(A10)} hold. The function $u(t,x):=Y_t^{t,x}$ is continuous in $t$.
\end{Lemma}
\begin{proof}
For simplicity, we assume $h=0$ and $g=0$.
We define $X_s^{t,x}:=x$, $Y_s^{t,x}:=Y_t^{t,x}$, $Z_s^{t,x}:=0$ and $A_s^{t,x}:=0$ for $s\in[0,t]$. It is easy to check that $(Y^{t,x},Z^{t,x},A^{t,x})$ is
a solution to the following $G$-BSDE on $[0,T]$:
\begin{equation}
\begin{cases}
\label{RBSDE_extend}
Y_s^{t,x}=\phi(X_T^{t,x})+\int_s^T \textbf{1}_{[t,T]}(s)f(r, X_r^{t,x},Y_r^{t,x},Z_r^{t,x})d\Bq_r
-\int_s^TZ_r^{t,x}dB_r+\int_s^TdA_r^{t,\xi},~s\in[0,T]\\
Y_s^{t,\xi}\geq S_s^{t,x},\quad \text{q.s.}, ~s\in[0,T];\\
\{-\int_0^s(Y_r^{t,x}-S_r^{t,x})dA_r^{t,x},~s\in[0,T]\}\text{ is a non-increasing }G\text{-martingale,}
\end{cases}
\end{equation}
where
\begin{equation*}
S_s^{t,x}=
\begin{cases}
l(s,X_s^{t,\xi}),\qquad s\in[t,T],\\
l(s,x),\quad ~\qquad s\in [0,t].
\end{cases}
\end{equation*}
Fix $x\in\bb R^n$. As before, in view of Propositions \ref{Propose_Z_bound}-\ref{Propose_stable}, we have for $0\leq t_1\leq t_2\leq T$ and some $p\geq 2$,
\begin{eqnarray}
\nonumber&&|u(t_1,x)-u(t_2,x)|^2=|Y_0^{t_1,x}-Y_0^{t_2,x}|^2\\
\nonumber&\leq& C\left\{\E\left[|\phi(X_T^{t_1,x})-\phi(X_T^{t_2,x})|^{2p}\right]^{1\over p}
+\E\left[\sup_{s\in [0,T]}|l(s,X_s^{t_1,x})-l(s,X_s^{t_2,x})|^{2p}\right]^{1\over 2p}\right\}\\
\nonumber&&+C\E\left[\left(\int_0^T |\textbf{1}_{[t_1,T]}(s)f(s, X_s^{t_1,x},Y_s^{t_2,x},Z_s^{t_2,x})-\textbf{1}_{[t_2,T]}(s)f(s, X_s^{t_2,x},Y_s^{t_2,x},Z_s^{t_2,x})|^2ds\right)^p\right]^{1\over p}\\
\nonumber&\leq& C^\prime\E\left[(1+|X_T^{t_1,x}|^{2pm}+|X_T^{t_2,x}|^{2pm})^2\right]^{1\over 2p}\E\left[|X_T^{t_1,x}-X_T^{t_2,x}|^{4p}\right]^{1\over 2p}\\
\nonumber&&+C^\prime\E\left[\sup_{s\in [0,T]}|X_s^{t_1,x}-X_s^{t_2,x}|^{2p}\right]^{1\over 2p}
+C^\prime\E\left[\left(\int_{t_1}^{t_2} |f(s, X_s^{t_1,x},Y_s^{t_2,x},Z_s^{t_2,x})|^2ds\right)^p\right]^{1\over p}\\
\nonumber&&+C^\prime\E\left[\int_{t_2}^T (1+|X_s^{t_1,x}|^{2pm}+|X_s^{t_2,x}|^{2pm})^2ds\right]^{1\over 2p}\E\left[\int_{t_2}^T|X_s^{t_1,x}-X_s^{t_2,x}|^{4p}ds\right]^{1\over 2p}\\
\nonumber&\leq& C^{\prime\prime}(1+|x|^{2m})\E\left[\sup_{s\in[t_2,T]}|X_s^{t_1,x}-X_s^{t_2,x}|^{4p}\right]^{1\over 2p}+C^{\prime\prime}\E\left[\sup_{s\in [0,T]}|X_s^{t_1,x}-X_s^{t_2,x}|^{2p}\right]^{1\over 2p}\\
&&+C^{\prime\prime}\E\left[\left(\int_{t_1}^{t_2} |f(s, X_s^{t_1,x},Y_s^{t_2,x},0)|^2ds\right)^p\right]^{1\over p}.
\label{ineq_u_t}
\end{eqnarray}
For each $\alpha\geq 2$, we have
\begin{eqnarray*}
&&\E\left[\sup_{s\in [0,T]}|X_s^{t_1,x}-X_s^{t_2,x}|^{\alpha}\right]\\
&\leq&\E\left[\sup_{s\in [t_1,t_2]}|X_s^{t_1,x}-x|^{\alpha}\right]+\E\left[\sup_{s\in [t_2,T]}|X_s^{t_2,X_{t_2}^{t_1,x}}-X_s^{t_2,x}|^{\alpha}\right]\\
&\leq& C_1(1+|x|^\alpha)|t_1-t_2|^{\alpha/2}+C_1\E[|X_{t_2}^{t_1,x}-x|^{\alpha}]\\
&\leq& C_2(1+|x|^\alpha)|t_1-t_2|^{\alpha/2}.
\end{eqnarray*}
On the other hand, in view of Proposition \ref{Propose_Y_bound}, for each $\alpha\geq 2$,
\begin{eqnarray*}
&&\E\left[\left(\int_{t_1}^{t_2} |f(s, X_s^{t_1,x},Y_s^{t_2,x},0)|^2ds\right)^\alpha\right]\\
&\leq& C\E\left[\int_{t_1}^{t_2} |f(s, 0,0,0)|^{2\alpha}+|X_s^{t_1,x}|^{2\alpha}+|Y_s^{t_2,x}|^{\alpha}ds\right]\\
&\leq& \tilde C \int_{t_1}^{t_2} \left(|f(s, 0,0,0)|^{2\alpha}+1+\E[|X_s^{t_1,x}|^{2\alpha}]\right)ds\\
&\leq& \tilde C^\prime \int_{t_1}^{t_2} \left(|f(s, 0,0,0)|^{2\alpha}+1+|x|^{2\alpha}\right)ds.
\end{eqnarray*}
Then from (\ref{ineq_u_t}), we know $u$ is continuous in $t$.
\end{proof}
Now we consider the penalized $G$-BSDEs:
\begin{eqnarray*}
Y_s^{t,x,n}&=&\phi(X_T^{t,x})+\int_s^Tg(r, X_r^{t,x},Y_r^{t,x,n},Z_r^{t,x,n})dr+\int_s^Tf(r, X_r^{t,x},Y_r^{t,x,n},Z_r^{t,x,n})d\Bq_r\\
&&+n\int_s^T(Y_r^{t,x,n}-l(t,X_r^{t,x}))^-dr-\int_s^TZ_r^{t,x,n}dB_r-\int_s^TdK_r^{t,x,n},~~\text{q.s.},~s\in[t,T].\\
\end{eqnarray*}
We define $u_n(t,x):=Y_t^{t,x,n}$, $(t,x)\in[0,T]\times \bb R^n$. In view of Theorem \ref{Thm_qbsde_pde},
$u_n$ is the viscosity solution to the following PDE:
\begin{equation}
\begin{cases}
\label{PDE_penal}
\partial_t u_n+F_n(D_x^2u_n,D_xu_n,u_n,x,t)=0,~~(t,x)\in[0,T)\times\bb R^n\\
u_n(T,x)=\phi(x),~~x\in\bb R^n.
\end{cases}
\end{equation}
where
$$F_n(DA,p,y,x,t):=F(A,p,y,x,t)+n(y-l(t,x))^-,$$
for each $(A,p,y,x,t)\in \bb S_n\times \bb R^n\times \bb R\times\bb R^n\times[0,T] $.
\begin{Theorem}
Let assumptions \textbf{(A1)}-\textbf{(A4)} and \textbf{(A9)}-\textbf{(A10)} hold. The function $u(t,x):=Y_t^{t,x}$ is a viscosity solution of the obstacle problem (\ref{PDE_obstacle}).
\end{Theorem}
\begin{proof} We follow the procedure of \cite[Theorem 6.7]{Li2017_1}, and only sketch the main ideas.

From previous results, we have for each $(t,x)\in[0,T]\times\bb R^n$,
$$\lim_{n\to\infty}u_n(t,x)=u(t,x), \quad \text{and} \quad u_{n+1}(t,x)\geq u_n(t,x),~\forall n\in\bb Z^+.$$
Moreover, functions $u$ and $u_n$ are continuous. Then in view of Dini's Theorem, $u_n$ uniformly converges to $u$ on any compact subset.

We now show $u$ is a viscosity subsolution to (\ref{PDE_obstacle}). For each fixed $(t,x)\in[0,T]\times\bb R^n$, let $(p,q,A)\in\mathcal P^{2,+}u(t,x)$. We may assume $u(t,x)>l(t,x)$.
Similar as in the proof of \cite[Theorem 6.7]{Li2017_1}, we deduce that there exist sequences
$$n_j\to\infty,\qquad (t_j,x_j)\to(t,x),\qquad (p_j,q_j,A_j)\to(p,q,A),$$
where $(p_j,q_j,X_j)\in\mathcal P^{2,+}u_{n_j}(t_j,x_j) $. Since $u_n$ is the viscosity solution to (\ref{PDE_penal}), it follows that for any $j$,
$$\min\left\{-p_j-F_{n_j}(A_j,q_j,u_{n_j}(t_j,x_j),x_j,t_j),u(t_j,x_j)-l(t_j,x_j)\right\}\leq 0.$$
Noting that $u(t,x)>l(t,x)$, by the uniform convergence of $u_n$, we deduce that $u_j(t_j,x_j)>l(t_j,x_j)$ for sufficiently large integer $j$. Thus letting $j\to\infty$, we have
$$-p-F(A,q,u(t,x),x,t)\leq 0,$$ which means $u$ is a viscosity subsolution to (\ref{PDE_obstacle}).

In a similar way,  $u$ is proved to be a viscosity supersolution to (\ref{PDE_obstacle}).
\end{proof}


\begin{thebibliography}{99}
\bibitem{Ankirchner2007}Ankirchner S., Imkeller P. and Dos Reis G. : Classical and variational differentiability of BSDEs with quadratic growth[J]. Electronic Journal of Probability, 2007, 12: 1418-1453.

\bibitem{Bismut1976}  Bismut J-M. : Linear quadratic optimal stochastic control with random coefficients[J]. SIAM J. Control Optim., 14 (1976), pp. 419-444.

\bibitem{Bismut1978} Bismut J-M. : Contr\^ole des systems lin\'eares quadratiques: Applications de l'int\'egrale stochastique, in:  S\'eminaire de Probabilit\'es XII, Lecture Notes in Math. 649, C. Dellacherie, P. A. Meyer, and M. Weil, eds., Springer-Verlag, Berlin, 1978, pp. 180-264.

\bibitem{Briand2013}Briand Ph. and Elie R. : A simple constructive approach to quadratic BSDEs with or without delay[J]. Stochastic processes and their applications, 2013, 123(8): 2921-2939.

\bibitem{Briand2006}Briand Ph. and Hu Y. : BSDE with quadratic growth and unbounded terminal value[J]. Probability Theory and Related Fields, 2006, 136(4): 604-618.

\bibitem{Briand2008}Briand Ph. and Hu Y. : Quadratic BSDEs with convex generators and unbounded terminal conditions[J]. Probability Theory and Related Fields, 2008, 141(3): 543-567.

\bibitem{Buckdahn2008}Buckdahn R. and Li J. : Stochastic differential games and viscosity solutions of Hamilton¨CJacobi¨CBellman¨CIsaacs equations[J]. SIAM Journal on Control and Optimization, 2008, 47(1): 444-475.
\bibitem{Cheridito2007}Cheridito P., Soner H. M., Touzi N. and Victoir N. : Second-order backward stochastic differential equations and fully nonlinear parabolic PDEs[J]. Communications on Pure and Applied Mathematics, 2007, 60(7): 1081-1110.
\bibitem{DaLio2006}Da Lio F. and Ley O. : Uniqueness Results for Second-Order Bellman--Isaacs Equations under Quadratic Growth Assumptions and Applications[J]. SIAM journal on control and optimization, 2006, 45(1): 74-106.
\bibitem{Denis2011}Denis L., Hu M. and Peng S. : Function spaces and capacity related to a sublinear expectation: application to G-Brownian motion paths[J]. Potential Analysis, 2011, 34(2): 139-161.
\bibitem{El Karoui1997}El Karoui N., Kapoudjian C., Pardoux E., Peng S. and Quenez M. C. : Reflected solutions of backward SDE's, and related obstacle problems for PDE's[J]. The Annals of Probability, 25(2), 702-737.

\bibitem{HuM2014_1}Hu M., Ji S., Peng S. and Song Y. : Backward stochastic differential equations driven by $G$-Brownian motion[J]. Stochastic Processes and their Applications, 2014, 124(1): 759-784.
\bibitem{HuM2014_2}Hu M., Ji S., Peng S. and Song Y. : Comparison theorem, Feynman-Kac formula and Girsanov transformation for BSDEs driven by $G$-Brownian motion[J]. Stochastic Processes and their Applications, 2014, 124(2): 1170-1195.


\bibitem{HuM2009}Hu M. and Peng S. : On representation theorem of $G$-expectations and paths of $G$-Brownian motion[J]. Acta Mathematicae Applicatae Sinica (English Series), 2009, 25(3): 539-546.

\bibitem{Hu2018}Hu Y., Lin Y. and Soumana Hima A. : Quadratic backward stochastic differential equations driven by G-Brownian motion: Discrete solutions and approximation[J]. Stochastic Processes and their Applications, 2018, 128(11): 3724-3750.


\bibitem{HuTang2016}Hu Y. and Tang S. : Multi-dimensional backward stochastic differential equations of diagonally quadratic generators. Stochastic Processes and their Application, 2016, 126(4): 1066--1086.

\bibitem{Kazamaki1994}Kazamaki N. : Continuous exponential martingales and BMO[M]. Springer-Verlag Berlin Heidelberg, 1994.
\bibitem{Kobylanski2000}Kobylanski M. : Backward stochastic differential equations and partial differential equations with quadratic growth[J]. The Annals of Probability, 28(2), 558-602.
\bibitem{Kobylanski2002}Kobylanski M., Lepeltier J. P., Quenez M. C. and Torres S. : Reflected BSDE with superlinear quadratic coefficient[J]. Probability and Mathematical Statistics, 2002, 22(1): 51-83.
\bibitem{Krylov1987}Krylov N V. Nonlinear elliptic and parabolic equations of the second order[M]. Springer, 1987.
\bibitem{Lepeltier2007}Lepeltier J. P. and Xu M. : Reflected BSDE with quadratic growth and unbounded terminal value[J]. arXiv preprint arXiv:0711.0619, 2007.
\bibitem{Li2017_1} Li H., Peng S. and Soumana Hima A. : Reflected solutions of backward stochastic differential equations driven
by $G$-Brownian motion[J]. Sci China Math, 2018, 61: 1-26.
\bibitem{Li2017_2} Li H. and Peng S. : Reflected BSDE driven by $G$-Brownian motion with an upper obstacle[J]. arXiv preprint arXiv:1709.09817, 2017.
\bibitem{Li2017_3}Li H., Peng S. and Song Y. : Supermartingale decomposition theorem under $ G $-expectation[J]. Electronic Journal of Probability, 2018, 23.
\bibitem{LiX2011}Li X. and Peng S. : Stopping times and related It\^o's calculus with $G$-Brownian motion[J]. Stochastic Processes and their Applications, 2011, 121(7): 1492-1508.

\bibitem{Matoussi1}Matoussi A., Piozin L. and Possama\"{i} D. : Second-order BSDEs with general reflection and game options under uncertainty. { Stochastic Processes and their Applications}, 2014, { 124}(7): 2281-2321.

\bibitem{Matoussi2}Matoussi A., Possama\"{i} D. and Zhou C. : Second order reflected backward stochastic differential equations. { The Annals of Applied Probability}, 2013, { 23}(6), 2420-2457.
\bibitem{Matoussi3}Matoussi A., Possama\"{i} D. and Zhou C. : Corrigendum for "Second-order reflected backward stochastic differential equations" and "Second-order BSDEs with general reflection and game options under uncertainty", 2017, arXiv preprint arXiv, 1706.08588v2.


\bibitem{Pardoux1990}Pardoux E. and Peng S. : Adapted solution of a backward stochastic differential equation[J]. Systems \& Control Letters, 1990, 14(1): 55-61.

\bibitem{Peng1992}Peng S.  Stochastic Hamilton-Jacobi-Bellman equations[J]. SIAM J. Control Optim., 30 (1992),
pp. 284?304.

\bibitem{Peng1992}Peng S. A generalized dynamic programming principle and Hamilton-Jacobi-Bellman equation[J]. Stochastics: An International Journal of Probability and Stochastic Processes, 1992, 38(2): 119-134.

\bibitem{Peng2007}Peng S. : $G$-expectation, $G$-Brownian motion and related stochastic calculus of It\^o type[J]. Stochastic analysis and applications, 2007, 2(541-567): 3For.
\bibitem{Peng2008}Peng S. : Multi-dimensional $G$-Brownian motion and related stochastic calculus under $G$-expectation[J]. Stochastic Processes and their Applications, 2008, 118(12): 2223-2253.
\bibitem{Peng2010}Peng S. : Nonlinear expectations and stochastic calculus under uncertainty[J]. arXiv preprint arXiv:1002.4546, 2010.
\bibitem{Peng2010_1}Peng S. : Backward stochastic differential equation, nonlinear expectation and their applications[C]. Proceedings of the International Congress of Mathematicians. 2010, 1: 393-432.
\bibitem{PossZhou2013} Possama\"{i} D. and Zhou C. : Second order backward stochastic differential equations with quadratic growth[J]. Stochastic Processes and their applications, 2013, 123(10): 3770-3799.
\bibitem{Soner2011}Soner H. M., Touzi N. and Zhang J. : Martingale representation theorem for the $G$-expectation[J]. Stochastic Processes and their Applications, 2011, 121(2): 265-287.
\bibitem{Soner2012}Soner H. M., Touzi N. and Zhang J. : Wellposedness of second order backward SDEs[J]. Probability Theory and Related Fields, 2012, 153(1): 149-190.
\bibitem{Soner2013}Soner H. M., Touzi N., Zhang J. : Dual formulation of second order target problems[J]. The Annals of Applied Probability, 2013, 23(1): 308-347.
\bibitem{Song2011}Song Y. : Some properties on $G$-evaluation and its applications to $G$-martingale decomposition[J]. Science China Mathematics, 2011, 54(2): 287-300.

\bibitem{Tang2003} Tang S. : General linear quadratic optimal stochastic control problems with random coefficients: linear stochastic Hamilton systems and backward stochastic Riccati equations. SIAM J. Contr. Optim., 2003, 42(1): 53--75.

\bibitem{Xu2011}Xu J., Shang H. and Zhang B. : A Girsanov type theorem under G-framework[J]. Stochastic Analysis and Applications, 2011, 29(3): 386-406.

\end{thebibliography}
\end{document}